\documentclass[a4paper,12pt]{amsart}

\usepackage{fullpage}
\usepackage{amsmath,amssymb,graphicx,psfrag}
\usepackage[T1]{fontenc}
\usepackage{amsmath,amssymb,ifthen}
\usepackage{color}
\usepackage{pgfplots}
\usepackage{standalone}
\usepackage{adjustbox}
\usepackage{moreverb}

\newcommand{\OO}{\mathcal{O}}
\newcommand{\MM}{\mathcal{M}}
\newcommand{\PP}{\mathcal{P}}
\def\QQ{\mathcal Q}
\newcommand{\RR}{\mathcal{R}}
\newcommand{\TT}{\mathcal{T}}
\newcommand{\UU}{\mathcal{U}}

\def\AA{\mathbb A}
\renewcommand{\H}{\widetilde{H}}
\newcommand{\N}{\mathbb{N}}
\newcommand{\R}{\mathbb{R}}
\def\k{\underline{k}}

\def\fine{\circ}
\def\coarse{\bullet}

\def\P{\boldsymbol{P}}
\def\A{\boldsymbol{A}}
\def\b{\boldsymbol{b}}
\def\x{\boldsymbol{x}}
\def\y{\boldsymbol{y}}

\def\PPhi{\phi}
\def\PPsi{\psi}
\def\pphi{\phi}
\def\exact{\star}
\def\opt{{\rm opt}}

\newcommand{\dual}[3][]{#1\langle#2\,,\,#3#1\rangle}
\newcommand{\norm}[3][]{#1\|#2#1\|_{#3}}
\newcommand{\edual}[3][]{#1\langle\!\!#1\langle#2\,,\,#3#1\rangle\!\!#1\rangle}
\newcommand{\enorm}[2][]{#1|\!#1|\!#1|\,#2\,#1|\!#1|\!#1|}
\def\d#1{\,{\rm d}#1}
\newcommand{\diam}{{\rm diam}}
\newcommand{\set}[3][\big]{#1\{#2\,:\,#3#1\}}
\def\conv{{\rm conv}}
\def\refine{{\rm refine}}
\def\cond{{\rm cond}}

\def\Cmark{C_{\rm mark}}
\def\Cdrel{C_{\rm drl}}
\def\Crel{C_{\rm rel}}
\def\Ceff{C_{\rm eff}}
\def\Cstab{C_{\rm stb}}
\def\Cpcg{C_{\rm pcg}}
\def\Clin{C_{\rm lin}}
\def\Copt{C_{\rm opt}}

\def\qpcg{q_{\rm pcg}}
\def\qlin{q_{\rm lin}}

\def\reff#1#2{\!\stackrel{\eqref{#1}}{#2}\!}

\newcommand{\SX}{\ensuremath{\mathcal{X}}}
\newcommand{\EE}{\ensuremath{\mathcal{E}}}
\newcommand{\supp}{\ensuremath{\operatorname{supp}}}
\newcommand{\linhull}{\ensuremath{\operatorname{span}}}

\newcommand{\embed}{\ensuremath{\boldsymbol{I}}}
\newcommand{\D}{\ensuremath{\boldsymbol{D}}}
\newcommand{\Hmat}{\ensuremath{\boldsymbol{H}}}

\newcounter{statement}
\newenvironment{statement}[2][!]{%
\vskip3mm
\hrule
\hrule
\hrule
\vskip1mm
\noindent%
\refstepcounter{statement}%
\bf#2~\thestatement%
\ifthenelse{\equal{#1}{!}}{.\ }{~(#1).\ }%
\it%
}{%
\vskip1mm
\hrule
\hrule
\hrule
\vskip2mm
}

\newenvironment{theorem}[1][!]{\begin{statement}[#1]{Theorem}}{\end{statement}}
\newenvironment{lemma}[1][!]{\begin{statement}[#1]{Lemma}}{\end{statement}}
\newenvironment{proposition}[1][!]{\begin{statement}[#1]{Proposition}}{\end{statement}}
\newenvironment{corollary}[1][!]{\begin{statement}[#1]{Corollary}}{\end{statement}}
\newenvironment{remark}[1][!]{\begin{statement}[#1]{Remark}}{\end{statement}}
\newenvironment{algorithm}[1][!]{\begin{statement}[#1]{Algorithm}}{\end{statement}}

\renewcommand{\subsection}[1]{%
 \vskip2mm
 \refstepcounter{subsection}%
 {\bf\arabic{section}.\arabic{subsection}.~#1.~}%
}

\usepackage{fancyhdr}
\advance\footskip0.4cm
\advance\textheight-0.4cm
\calclayout
\definecolor{gray}{rgb}{0.75,0.75,0.75}
\newenvironment{explain}{\begin{list}{\large\color{gray}$\bullet$}{%
\setlength{\labelsep}{2.3mm}%
\setlength{\labelwidth}{2mm}%
\setlength{\leftmargin}{5mm}%
}}{\end{list}}

\newcommand*\patchAmsMathEnvironmentForLineno[1]{%
  \expandafter\let\csname old#1\expandafter\endcsname\csname #1\endcsname
  \expandafter\let\csname oldend#1\expandafter\endcsname\csname end#1\endcsname
  \renewenvironment{#1}%
     {\linenomath\csname old#1\endcsname}%
     {\csname oldend#1\endcsname\endlinenomath}}%
\newcommand*\patchBothAmsMathEnvironmentsForLineno[1]{%
  \patchAmsMathEnvironmentForLineno{#1}%
  \patchAmsMathEnvironmentForLineno{#1*}}%
\AtBeginDocument{%
\patchBothAmsMathEnvironmentsForLineno{equation}%
\patchBothAmsMathEnvironmentsForLineno{align}%
\patchBothAmsMathEnvironmentsForLineno{flalign}%
\patchBothAmsMathEnvironmentsForLineno{alignat}%
\patchBothAmsMathEnvironmentsForLineno{gather}%
\patchBothAmsMathEnvironmentsForLineno{multline}%
}
\usepackage[mathlines]{lineno}

\title{Adaptive BEM with inexact PCG solver\\yields almost optimal computational costs}

\author{Thomas F\"uhrer}
\address{Pontificia Universidad de Cat\'olica, Facultad de Matem\'aticas, Vicku\~{n}a Mackenna 4860, Santiago, Chile}
\email{tofuhrer@mat.uc.cl}

\author{Alexander Haberl}
\author{Dirk Praetorius}
\author{Stefan Schimanko}
\address{TU Wien, Institute for Analysis and Scientific Computing, Wiedner Hauptstr.~8--10/E101/4, 1040 Wien, Austria}
\email{\{ alexander.haberl , dirk.praetorius \} @asc.tuwien.ac.at}
\email{stefan.schimanko@asc.tuwien.ac.at\quad\rm(corresponding author)}

\thanks{\textbf{Acknowledgement.} The authors thankfully acknowledge the support by the Austrian Science Fund (FWF) through grant
P27005 (AH, DP, SS) as well as grant F65  (DP) and by CONICYT through FONDECYT project P11170050 (TF)}

\subjclass[2010]{}

\date{\today}

\begin{document}

\begin{abstract}
We consider the preconditioned conjugate gradient method (PCG) with optimal preconditioner in the frame of the boundary element method (BEM) for elliptic first-kind integral equations. Our adaptive algorithm steers the termination of PCG as well as the local mesh-refinement. 
Besides 
convergence with optimal algebraic rates, we also prove almost optimal computational complexity. In particular, we provide an additive Schwarz preconditioner which can be computed in linear complexity and which is optimal in the sense that the condition numbers of the preconditioned systems are uniformly bounded.
As model problem serves the 2D or 3D Laplace operator and the associated weakly-singular integral equation with energy space $\H^{-1/2}(\Gamma)$. The main results also hold for the hyper-singular  integral equation with energy space $H^{1/2}(\Gamma)$.
\end{abstract}

\maketitle

\vspace*{-10mm}
\section{Introduction}

\subsection{Model problem}
Let $\Omega \subset \R^d$ with $d=2,3$ be a bounded Lipschitz domain with polyhedral boundary $\partial\Omega$.
Let $\Gamma \subseteq \partial\Omega$ be a (relatively) open and connected subset.
Given $f : \Gamma \to \R$,  we seek the density $\PPhi^\exact\colon\Gamma\rightarrow\R$ of the weakly-singular integral equation
\begin{align}\label{eq:slp}
 (V \pphi^\exact)(x) := \int_\Gamma G(x-y) \pphi^\exact(y) \d{y} = f(x)
 \quad \text{for all } x \in \Gamma,
\end{align}
where $G(\cdot)$ denotes the fundamental solution of the Laplace operator, i.e.,
\begin{align}
 G(z) = - \frac{1}{2\pi} \, \log |z| \quad \text{for } d = 2
 \quad \text{resp.} \quad
 G(z) = \frac{1}{4\pi} \, \frac{1}{|z|} \quad \text{for } d = 3.
\end{align}
Given a triangulation $\TT_\coarse$ of $\Gamma$, we employ a lowest-order Galerkin boundary element method (BEM) to compute a $\TT_\coarse$-piecewise constant function $\PPhi_\coarse^\exact \in \PP^0(\TT_\coarse)$ such that
\begin{align}\label{eq:bem}
 \int_\Gamma (V \PPhi_\coarse^\exact)(x) \, \PPsi_\coarse(x) \d{x}
 = \int_\Gamma f(x) \, \PPsi_\coarse(x) \d{x} 
 \quad \text{for all } \PPsi_\coarse \in \PP^0(\TT_\coarse).
\end{align}
With the numbering $\TT_\coarse = \{T_1,\dots,T_N\}$, consider the standard basis $\set{\chi_{\coarse,j}}{j=1,\dots,N}$ of $\PP^0(\TT_\coarse)$ consisting of characteristic functions $\chi_{\coarse,j}$ of $T_j\in\TT_\coarse$. We make the ansatz  
\begin{align}
 \PPhi_\coarse^\exact = \sum_{k=1}^N \x_\coarse^\exact[k] \, \chi_{\coarse,k}
 \quad \text{with coefficient vector } 
 \x_\coarse^\exact=(\x_\coarse^\exact[1], \dots, \x_\coarse^\exact[N]) \in \R^N.
\end{align} 
Then, the Galerkin formulation~\eqref{eq:bem} is equivalent to the linear system
\begin{align}\label{eq:linearsystem}
 \A_\coarse \x_\coarse^\exact = \b_\coarse
 \text{\ \ with\ \ }
 \A_\coarse[j,k] := \int_{T_j} (V \chi_{\coarse,k})(x) \d{x},
 \quad
 \b_\coarse[j] := \int_{T_j} f(x) \d{x},
\end{align}
where the matrix $\A_\coarse \in \R^{N \times N}$ is positive definite and symmetric. For a given initial triangulation $\TT_0$, we consider an adaptive mesh-refinement strategy of the type
\begin{align}\label{eq:semr}
 \boxed{~solve~} 
 \longrightarrow
 \boxed{~estimate~} 
 \longrightarrow
 \boxed{~mark~} 
 \longrightarrow
 \boxed{~refine~} 
\end{align}
which generates a sequence $\TT_\ell$ of successively refined triangulations $\TT_\ell$ for all $\ell \in \N_0$. We note that the condition number of the Galerkin matrix $\A_\ell$ from~\eqref{eq:linearsystem} depends on the number of elements of $\TT_\ell$, as well as the minimal and maximal diameter. Therefore, the step $~solve~$ requires an efficient preconditioner as well as an appropriate iterative solver.

\subsection{State of the art}
In the last decade, the mathematical understanding of adaptive mesh-refinement has matured. We refer to~\cite{doerfler, mns, bdd, stevenson07, ckns, ffp14} for some milestones for adaptive finite element methods for second-order linear elliptic equations, \cite{gantumur, fkmp13, partOne, partTwo, invest} for adaptive BEM, and~\cite{axioms} for a general framework of rate-optimality of adaptive mesh-refining algorithms. The interplay between adaptive mesh-refinement, optimal convergence rates, and inexact solvers has been addressed and analyzed for adaptive FEM for linear problems in~\cite{stevenson07,MR3095916,MR3068564}, for eigenvalue problems in~\cite{MR2970733}, and recently also for strongly monotone nonlinearities in~\cite{ghps2017}.
In particular, all available results for adaptive BEM~\cite{gantumur, fkmp13, partOne, partTwo, invest} assume that the Galerkin system~\eqref{eq:linearsystem} is solved exactly. Instead, the present work 
analyzes an adaptive algorithm which steers both, the local mesh-refinement and the iterations of the PCG algorithm. 

In principle, it is known~\cite[Section~7]{axioms} that convergence and optimal convergence rates are preserved if the linear system is solved inexactly, but with sufficient accuracy.  The purpose of this work is to guarantee the latter by incorporating an appropriate stopping criterion for the PCG solver into the adaptive algorithm. Moreover, to prove that the proposed algorithm does not only lead to optimal algebraic convergence rates, but also to (almost) optimal computational costs, we  provide an appropriate symmetric and positive definite preconditioner $\P_\ell \in \R^{N\times N}$ such that
\begin{explain}
\item first, the matrix-vector products with $\P_\ell^{-1}$ can be computed at linear cost;
\item second, the system matrix $\P^{-1/2}_\ell \A_\ell \P_\ell^{-1/2}$ of the preconditioned linear system
\begin{align}\label{eq:linearsystem:pcg}
 \P^{-1/2}_\ell \A_\ell \P_\ell^{-1/2} \widetilde \x_\ell^\exact = \P_\ell^{-1/2} \b_\ell
\end{align} 
has a uniformly bounded condition number which is independent of $\TT_\ell$.
\end{explain}
Then, $\x_\ell^\exact = \P_\ell^{-1/2} \widetilde \x_\ell^\exact$ solves the original system~\eqref{eq:linearsystem}.
To that end, we exploit the multilevel structure of adaptively generated meshes in the framework of adaptive Schwarz methods. For hyper-singular integral equations, such a multilevel additive Schwarz preconditioner has been proposed and analyzed in~\cite{MR3612925,fmpr15} for $d = 2,3$ and for weakly-singular integral equations in~\cite{MR3634453} for $d = 2$. In particular, the present work closes this gap by analyzing an optimal additive Schwarz preconditioner for weakly-singular integral equations for $d = 3$. We note that the proofs of~\cite{MR3612925,MR3634453} do not transfer to weakly-singular integral equations for $d=3$. Instead, we build on recent results for finite element discretizations~\cite{hiptwuzheng2012,MR3522965} which are then transferred to the present BEM setting by use of an abstract concept from~\cite{oswald99}.

\subsection{Outline and main results}
Section~\ref{section:preliminaries} introduces the functional analytic framework and fixes the necessary notation. 
Section~\ref{section:main_results} states our main results. In Section~\ref{section:oasp}, we define a local multilevel additive Schwarz preconditioner~\eqref{eq:defPrec} for a sequence of locally refined meshes. Theorem~\ref{thm:precond} states that the $\ell_2$-condition number of the preconditioned systems is uniformly bounded for all these meshes, i.e., the preconditioner is optimal. 
In Section~\ref{section:main_results_algorithm}, we first state our adaptive algorithm which steers the local mesh-refinement as well as the stopping of the PCG iteration (Algorithm~\ref{algorithm}). Theorem~\ref{theorem:algorithm} proves 
\begin{explain}
\item
that the overall error in the energy norm can be controlled {\sl a~posteriori}, 
\item
that the \emph{quasi-error} (which consists of energy norm error plus error estimator) is linearly convergent in each step of the adaptive algorithm (i.e., independent of whether the algorithm decides for local mesh-refinement or for one step of the PCG iteration), 
\item
that the quasi-error even decays with optimal rate (i.e., with each possible algebraic rate) with respect to the degrees of freedom, i.e., Algorithm~\ref{algorithm} is \emph{rate optimal} in the sense of, e.g.,~\cite{stevenson07,ckns,fkmp13,axioms}.
\end{explain}
Finally, Section~\ref{section:main:complexity} considers the computational costs.
Under realistic assumptions on the treatment of the arising discrete integral operators, Corollary~\ref{corollary:algorithm} states that the quasi-error converges at almost optimal rate (i.e., with rate $s-\varepsilon$ for any $\varepsilon>0$ if rate $s>0$ is possible for the exact Galerkin solution) with respect to computational costs, i.e., Algorithm~\ref{algorithm} requires \emph{almost optimal computational time}. Section~\ref{section:numerics} underpins our theoretical findings by some 2D and 3D experiments. The proof of Theorem~\ref{thm:precond} is given in Section~\ref{proof:precond}, the proofs of Theorem~\ref{theorem:algorithm} and Corollary~\ref{corollary:algorithm} are given in Section~\ref{proof:algorithm}. The final Section~\ref{section:hypsing} shows that our main results also apply to the hyper-singular integral equation.

\section{Preliminaries and notation}
\label{section:preliminaries}

\subsection{Functional analytic setting}
We briefly recall the most important facts and refer to~\cite{mclean} for further details and proofs.
With the Sobolev space $H^\alpha(\partial\Omega)$ defined as in~\cite{mclean} for $0 \le \alpha \le 1$, let $H^\alpha(\Gamma) := \set{v|_\Gamma}{v \in H^\alpha(\partial\Omega)}$ be associated with the natural quotient norm. Let $\H^{-\alpha}(\Gamma)$ be the dual space of $H^\alpha(\Gamma)$ with respect to the extended $L^2(\Gamma)$ scalar product $\dual{\psi}{f} = \int_\Gamma \psi(x) \, f(x) \d{x}$. Then, the single-layer potential $V$ from~\eqref{eq:slp} gives rise to a bounded linear operator $V:\H^{-1/2+s}(\Gamma) \to H^{1/2+s}(\Gamma)$ for all $-1/2 \le s \le 1/2$ which is even an isomorphism for $-1/2 < s < 1/2$. For $d=2$, the latter requires $\diam(\Omega) < 1$ which can always be ensured by scaling of $\Omega$. For $s=0$, the operator $V$ is even symmetric and elliptic, i.e.,
\begin{align}
 \edual{\pphi}{\psi} := \int_\Gamma (V\pphi)(x) \, \psi(x) \d{x}\qquad\textrm{for all $\pphi,\psi\in \H^{-1/2}(\Gamma)$}
\end{align}
defines a scalar product and $\enorm{\pphi}^2 := \edual{\pphi}{\pphi}$ is an equivalent norm on $\H^{-1/2}(\Gamma)$. For a given right-hand side $f \in H^{1/2}(\Gamma)$, the weakly-singular integral equation~\eqref{eq:slp} can thus equivalently be reformulated as 
\begin{align}\label{eq:weakform}
 \edual{\pphi^\exact}{\psi} = \dual{f}{\psi} 
 \quad \text{ for all } \psi \in \H^{-1/2}(\Gamma).
\end{align}
In particular, the Lax--Milgram theorem proves existence and uniqueness of the solution $\pphi^\exact \in \H^{-1/2}(\Gamma)$ to~\eqref{eq:weakform}.

\subsection{Boundary element method (BEM)}
Given a mesh $\TT_\coarse$ of $\Gamma$, let 
\begin{align}
 \PP^0(\TT_\coarse) 
 := \set{\PPsi_\coarse : \Gamma \to \R}{\forall T \in \TT_\coarse \quad \PPsi_\coarse|_T \text{ is constant}} 
\end{align}
be the space of $\TT_\coarse$-piecewise constant functions. Note that $\PP^0(\TT_\coarse) \subset L^2(\Gamma) \subset \H^{-1/2}(\Gamma)$. The Galerkin formulation~\eqref{eq:bem} can be reformulated as
\begin{align}\label{eq:galerkin}
 \edual{\PPhi_\coarse^\exact}{\PPsi_\coarse} = \dual{f}{\PPsi_\coarse}
 \quad \text{for all } \PPsi_\coarse \in \PP^0(\TT_\coarse).
\end{align}
Therefore, the Lax--Milgram theorem proves existence and uniqueness of the discrete solution $\PPhi_\coarse^\exact \in \PP^0(\TT_\coarse)$.

\begin{figure}[t]
 \centering
 \includegraphics[width=35mm]{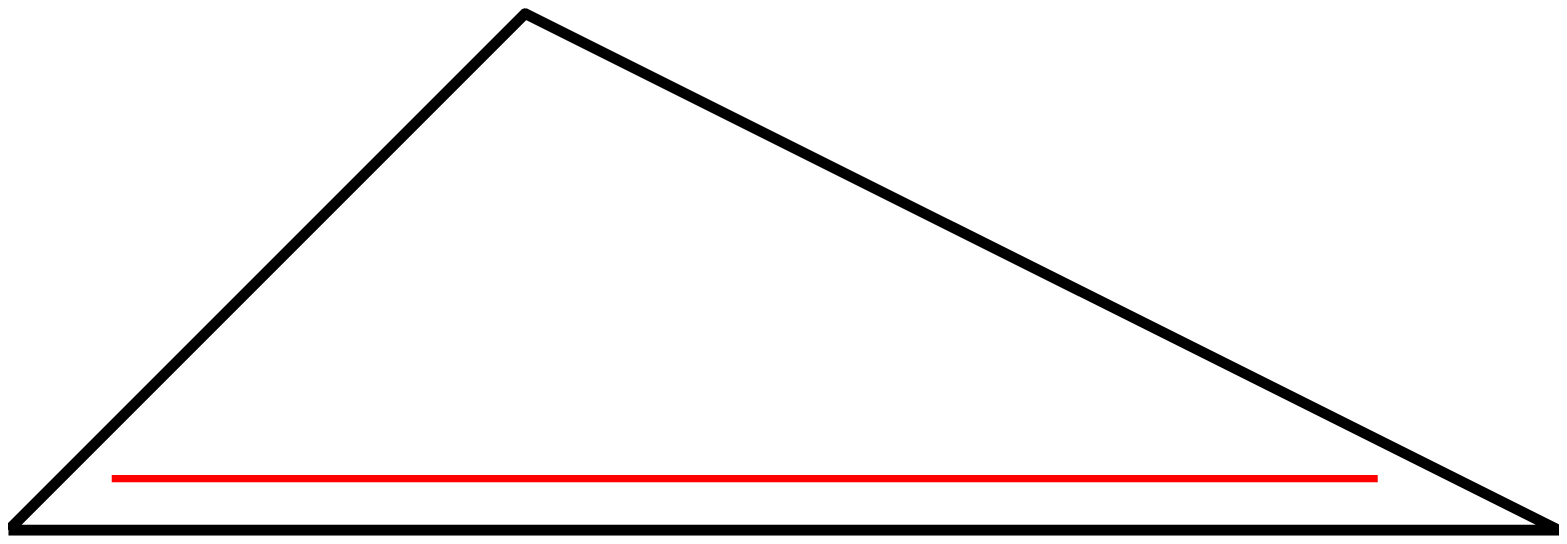} \quad
 \includegraphics[width=35mm]{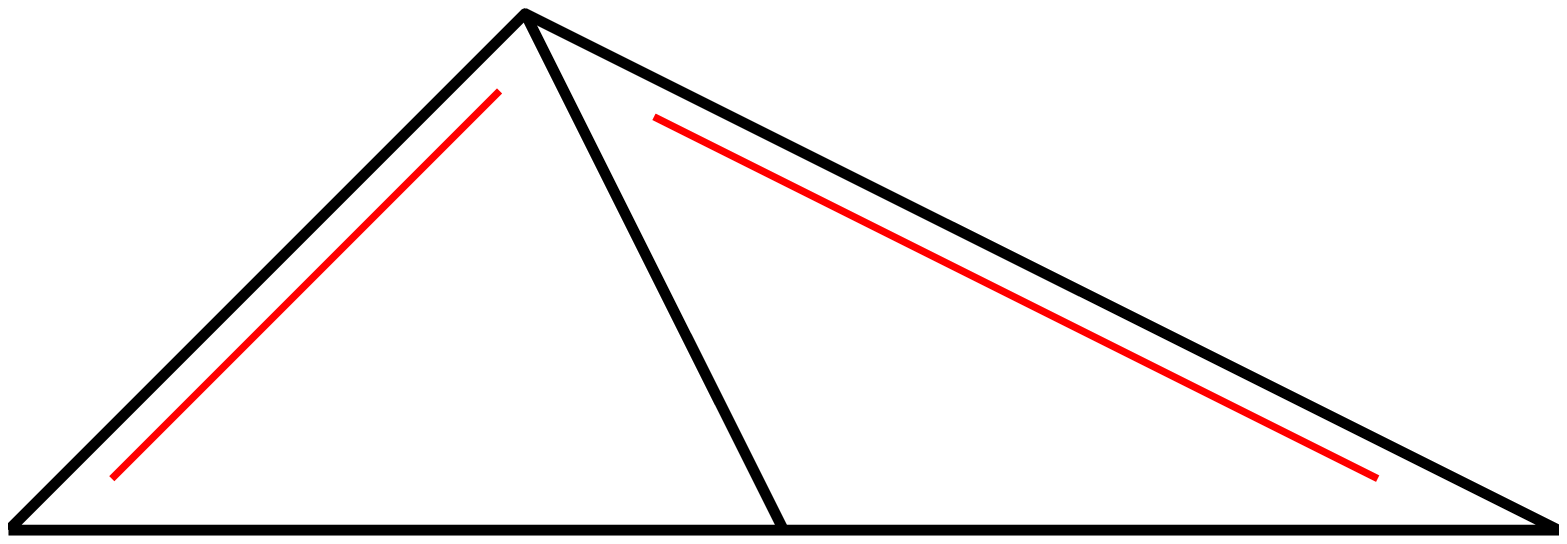} \quad
 \includegraphics[width=35mm]{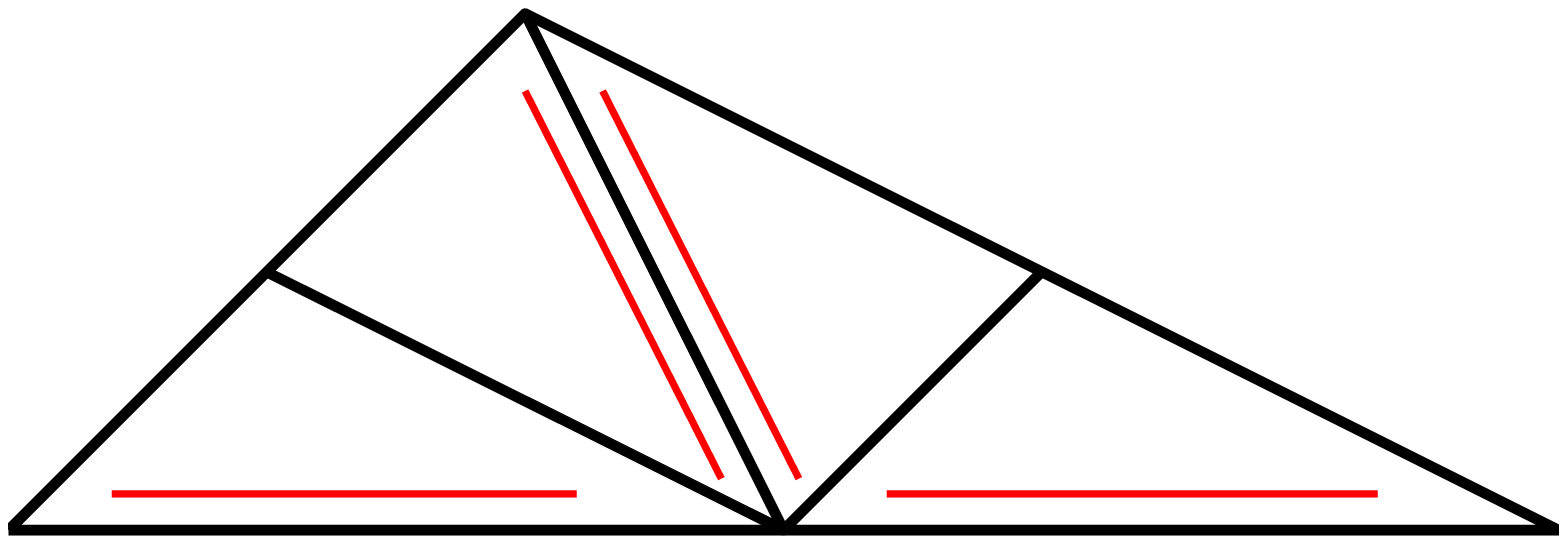} \quad
 \caption{\small For newest vertex bisection (NVB) in 2D, each triangle $T\in\TT$ has one \emph{reference edge}, indicated by the double line (left). 
 Bisection of $T$ is achieved by halving the reference edge (middle). The reference edges of the sons are always opposite to the new vertex. Recursive application of this refinement rule leads to conforming triangulations.}
  \label{fig:nvb}
\end{figure}

\subsection{Mesh-refinement for 2D BEM}\label{section:refined2D}
For $d = 2$, a mesh $\TT_\coarse$ of $\Gamma$ is a partition into non-degenerate compact line segments. It is
called \emph{$\gamma$-shape regular}, if
\begin{align}\label{eq:shaperegular2d}
 \max\set{h_T / h_{T'}}{T,T' \in \TT_\coarse \text{ with } T\cap T' \neq \emptyset} 
 \le \gamma.
\end{align}
Here, $h_T:=\diam(T) > 0$ denotes the Euclidean diameter of $T$, i.e., the length of the line segment.

We employ the extended bisection algorithm from~\cite{cmam}.
For a mesh $\TT_\coarse$ and $\MM_\coarse \subseteq \TT_\coarse$, let $\TT_\fine := \refine(\TT_\coarse, \MM_\coarse)$ be the coarsest mesh such that all marked elements $T\in\MM_\coarse$ have been refined, i.e., $\MM_\coarse \subseteq \TT_\coarse \backslash \TT_\fine$.
We write $\TT_\fine \in \refine(\TT_\coarse)$, if there exists $n\in\N_0$, conforming triangulations $\TT_0,\dots,\TT_n$ and corresponding sets of marked elements $\MM_j\subseteq\TT_j$ such that
\begin{explain}
\item $\TT_\coarse = \TT_0$,
\item $\TT_{j+1} = \refine(\TT_j,\MM_j)$ for all $j=0,\dots,n-1$,
\item $\TT_\fine  = \TT_n$,
\end{explain}
i.e., $\TT_\fine$ is obtained from $\TT_\coarse$ by finitely many steps of refinement. Note that the bisection algorithm from~\cite{cmam} guarantees, in particular, that all $\TT_\fine \in \refine(\TT_\coarse)$ are uniformly $\gamma$-shape regular, where $\gamma$ depends only on $\TT_\coarse$.

\subsection{Mesh-refinement for 3D BEM}\label{section:refined3D}
For $d = 3$, a mesh $\TT_\coarse$ of $\Gamma$ is a conforming triangulation into non-degenerate compact surface triangles. In particular, we avoid hanging nodes. To ease the presentation, we suppose that the elements $T \in \TT_\coarse$ are flat. The triangulation is called \emph{$\gamma$-shape regular}, if
\begin{align}\label{eq:shaperegular}
 \max_{T\in\TT_\coarse} \frac{\diam(T)}{h_T} \le \gamma.
\end{align}
Here, $\diam(T)$ denotes the Euclidean diameter of $T$ and $h_T:=|T|^{1/2}$ with $|T|$ being the two-dimensional surface measure. Note that $\gamma$-shape regularity implies that $h_T \le \diam(T) \le \gamma \, h_T$ and hence excludes anisotropic elements.

For 3D BEM, we employ 2D newest vertex bisection (NVB) to refine triangulations locally; see~\cite{stevenson,kpp} for details on the refinement algorithm and Figure~\ref{fig:nvb} for an illustration. 
For a mesh $\TT_\coarse$ and $\MM_\coarse \subseteq \TT_\coarse$, we employ the same notation  
$\TT_\fine := \refine(\TT_\coarse, \MM_\coarse)$ resp.\ $\TT_\fine \in \refine(\TT_\coarse)$ as for $d = 2$.

\subsection{\textsl{A~posteriori} BEM error control}
For $\PPsi_\coarse \in \PP^0(\TT_\coarse)$ and $\UU_\coarse \subseteq \TT_\coarse$, define
\begin{align}\label{eq:estimator}
 \eta_\coarse(\UU_\coarse, \PPsi_\coarse)^2
 := \sum_{T\in\UU_\coarse} \eta_\coarse(T, \PPsi_\coarse)^2,
 \quad \text{where} \quad
 \eta_\coarse(T, \PPsi_\coarse)^2 
 := h_T \, \norm{\nabla_\Gamma(f-V\PPsi_\coarse)}{L^2(T)}^2.
\end{align}
Here $\nabla_\Gamma(\cdot)$ denotes the arclength derivative for $d = 2$ resp.\ the surface gradient for $d = 3$.
To abbreviate notation, let $\eta_\coarse(\PPsi_\coarse) := \eta_\coarse(\TT_\coarse, \PPsi_\coarse)$.
If $\PPsi_\coarse = \PPhi_\coarse^\exact$ is the discrete solution to~\eqref{eq:galerkin}, then there holds the reliability estimate (i.e., the global upper bound)
\begin{align}\label{eq:prelim:reliability}
 \enorm{\pphi^\exact - \PPhi_\coarse^\exact} \le \Crel \, \eta_\coarse(\PPhi_\coarse^\exact),
\end{align}
where $\Crel > 0$ depends only on $\Gamma$ and $\gamma$-shape regularity of $\TT_\coarse$; see~\cite{cs95, cc97} for $d = 2$ resp.~\cite{cms01} for $d = 3$. Provided that $\PPhi^\exact\in L^2(\Gamma)$, the following 
weak efficiency
\begin{align}\label{eq:weak_eff}
\enorm{\PPhi^\exact-\PPhi_\coarse^\exact}+\eta_\coarse(\PPhi_\coarse^\exact)\leq\Ceff\,\norm{h_\coarse^{1/2}\,(\PPhi^\exact-\PPhi_\coarse^\exact)}{L^2(\Gamma)}
\end{align}
has recently been proved in \cite{invest}, where $\Ceff>0$ depends only on $\Gamma$ and $\gamma$-shape regulartiy of $\TT_\coarse$. We note that the weighted $L^2$-norm on the right-hand side of \eqref{eq:weak_eff} is only slightly stronger than $\enorm{\cdot}\simeq\norm{\cdot}{\H^{-1/2}(\Gamma)}$, so that one empirically observes $\eta_\coarse(\PPhi_\coarse^\exact)\lesssim\enorm{\PPhi-\PPhi_\coarse^\exact}$ in practice, cf. \cite{cs95,cc97,cms01}.
In certain situations (e.g., weakly-singular integral formulation of the interior 2D Dirichlet problem), one can rigorously prove the latter (strong) efficiency estimate up to higher-order data oscillations; see~\cite{cmam}.

\subsection{Preconditioned conjugate gradient method (PCG)}
Suppose that $\P_\coarse, \A_\coarse \in \R^{N \times N}$ are symmetric and positive definite matrices.
Given $\b_\coarse \in \R^N$ and an initial guess $\x_{\coarse0}$, PCG (see~\cite[Algorithm~11.5.1]{matcomp}) aims to approximate the solution
$\x_\coarse^\exact \in \R^N$ to~\eqref{eq:linearsystem}.
We note that each step of PCG has the following computational costs:
\begin{explain}
\item $\OO(N)$ cost for vector operations (e.g., assignment, addition, scalar product),
\item computation of \emph{one} matrix-vector product with $\A_\coarse$,
\item computation of \emph{one} matrix-vector product with $\P_\coarse^{-1}$.
\end{explain}
Let $\widetilde\x_\coarse^\exact \in \R^N$ be the solution to~\eqref{eq:linearsystem:pcg} and recall that $\x_\coarse^\exact = \P_\coarse^{-1/2} \widetilde\x_\coarse^\exact$. We note that PCG formally applies the conjugate gradient method (CG, see~\cite[Algorithm~11.3.2]{matcomp}) for the matrix $\widetilde\A_\coarse := \P_\coarse^{-1/2} \A_\coarse \P_\coarse^{-1/2}$ and the right-hand side $\widetilde\b_\coarse = \P_\coarse^{-1/2} \b_\coarse$. The iterates $\x_{\coarse k} \in \R^N$ of PCG (applied to $\P_\coarse$, $\A_\coarse$, $\b_\coarse$, and the initial guess $\x_{\coarse 0}$) and the iterates $\widetilde\x_{\coarse k}$ of CG (applied to $\widetilde\A_\coarse$, $\widetilde\b_\coarse$, and the initial guess $\widetilde\x_{\coarse 0} := \P_\coarse^{1/2} \x_{\bullet 0}$) are formally linked by
\begin{align*}
 \x_{\coarse k} = \P_\coarse^{-1/2} \widetilde\x_{\coarse k};
\end{align*}
see~\cite[Section~11.5]{matcomp}. Moreover, direct computation proves that
\begin{align}\label{eq:normequi1}
 \norm{\widetilde\y_\coarse}{\widetilde\A_\coarse}^2 := \widetilde\y_\coarse \cdot \widetilde\A_\coarse \widetilde\y_\coarse = \y _\coarse\cdot \A_\coarse \y_\coarse
 =: \norm{\y_\coarse}{\A_\coarse}^2
 \quad \text{for all $\widetilde \y_\coarse \in \R^N$ and $\y_\coarse = \P_\coarse^{-1/2} \widetilde \y_\coarse$}.
\end{align}
Consequently,~\cite[Theorem~11.3.3]{matcomp} for CG (applied to $\widetilde\A_\coarse$, $\widetilde\b_\coarse$,  $\widetilde\x_{\coarse0}$) yields the following lemma for PCG (which follows from the implicit steepest decent approach of CG).  

\begin{lemma}\label{lemma:pcg}
Let $\A_\coarse, \P_\coarse \in \R^{N \times N}$ be symmetric and positive definite, $\b_\coarse\in \R^N$, $\x_\coarse^\exact := \A_\coarse^{-1} \b_\coarse$, and $\x_{\coarse 0} \in \R^N$. 
Suppose the $\ell_2$-condition number estimate
\begin{align}\label{eq1:pcg}
 \cond_2(\P_\coarse^{-1/2} \A_\coarse \P_\coarse^{-1/2}) \le \Cpcg.
\end{align}
Then, the iterates $\x_{\bullet k}$ of the PCG algorithm satisfy the contraction property
\begin{align}\label{eq2:pcg}
 \norm{\x_\coarse^\exact - \x_{\coarse(k+1)}}{\A_\coarse} 
 \le \qpcg \, \norm{\x_\coarse^\exact - \x_{\coarse k}}{\A_\coarse}
 \quad\text{for all } k \in \N_0,
\end{align}
where $\qpcg := (1-1/\Cpcg)^{1/2} < 1$.\qed
\end{lemma}

If the matrix $\A_\coarse \in \R^{N \times N}$ stems from the Galerkin discretization~\eqref{eq:linearsystem} for $\TT_\coarse = \{ T_1, \dots, T_N \}$, there is a one-to-one correspondence of vectors $\y_\coarse \in \R^N$ and discrete functions $\PPsi_\coarse \in \PP^0(\TT_\coarse)$ via $\PPsi_\coarse = \sum_{j=1}^N \y_\coarse[j] \, \chi_{\coarse,j}$. Let $\PPhi_{\coarse k} \in \PP^0(\TT_\coarse)$ denote the discrete function corresponding to the PCG iterate $\x_{\coarse k} \in \R^N$, while the Galerkin solution $\PPhi_\coarse^\exact \in \PP^0(\TT_\coarse)$ of~\eqref{eq:galerkin} corresponds to $\x_\coarse^\exact = \A_\coarse^{-1} \b_\coarse$. We note the elementary identity 
\begin{align}\label{eq:normequi2}
 \enorm{\PPhi_\coarse^\exact - \PPhi_{\coarse k}}^2 
 = (\x_\coarse^\exact - \x_{\coarse k})\cdot \A_\coarse (\x_\coarse^\exact - \x_{\coarse k})
 = \norm{\x_\coarse^\exact - \x_{\coarse k}}{\A_\coarse}^2.
\end{align}

\subsection{Optimal preconditioners}
\label{section:optimal_preconditioners}%
We say that $\P_\coarse$ is an \emph{optimal preconditioner}, if $\Cpcg \ge 1$ in the $\ell_2$-condition number estimate~\eqref{eq1:pcg} depends only on $\gamma$-shape regularity of $\TT_\coarse$ and the initial mesh $\TT_0$ (and is hence essentially independent of the mesh $\TT_\coarse$).

\section{Main results}
\label{section:main_results}

\subsection{Optimal additive Schwarz preconditioner}
\label{section:oasp}
In this work, we consider multilevel additive Schwarz preconditioners that build on the adaptive mesh-hierarchy. 

Let $\EE_\bullet$ denote the set of all nodes ($d=2$) resp.\ edges ($d=3$) of the mesh $\TT_\coarse$
which do not belong to the relative boundary $\partial\Gamma$.
Only for $\Gamma=\partial\Omega$, $\EE_\bullet$ contains all nodes resp.\ edges of $\TT_\coarse$.
For $E\in\EE_\bullet$, let
$T^\pm\in\TT_\bullet$ denote the two unique elements with $T^+\cap T^- = E$.
We define the Haar-type function $\varphi_{\bullet,E}\in\PP^0(\TT_\bullet)$ (associated to $E\in\EE_\bullet$) by
\begin{align}\label{eq:psi}
  \varphi_{\bullet,E}|_T := 
  \begin{cases}
    \pm\frac{|E|}{|T^\pm|} & \text{for } T\in \{T^+,T^-\}, \\
    0 & \text{else},
  \end{cases}
\end{align}
where $|E| := 1$ for $d=2$ and $|E| :=\diam(E)$ for $d=3$.
Note that 
\begin{align}\label{eq:P0*}
 \varphi_{\bullet,E} \in \PP_*^0(\TT_\bullet) :=\set[\Big]{\psi\in\PP^0(\TT_\bullet)}{\int_\Gamma \psi \d{x} = 0}.
\end{align}
For $d=3$, we additionally suppose that the orientation of each edge $E$ is arbitrary but fixed.
We choose $T^+ \in \TT_\coarse$ such that $\partial T^+$ and $E \subset \partial T^+$ have the same orientation.

Given a mesh $\TT_0$, suppose that $\TT_\ell$ is a sequence of locally refined meshes, i.e., for all $\ell \in \N_0$, there exists a set $\MM_\ell \subseteq \TT_\ell$ such that $\TT_{\ell+1} = \refine(\TT_\ell,\MM_\ell)$. 
Then, define
\begin{align*}
  \EE_\ell^\star := \EE_\ell \backslash \EE_{\ell-1} \cup
  \set{E\in\EE_\ell}{\supp(\varphi_{\ell,E})\subsetneqq \supp(\varphi_{\ell-1,E})}
  \quad\textrm{for all $\ell\geq1$,}
\end{align*}
which consist of new (interior) nodes/edges plus some of their neighbours. We note the following subspace decomposition which is, in general, \emph{not} direct.

\begin{lemma}
With $\SX_\bullet := \PP^0(\TT_\bullet)$ and $\SX_{\bullet,E} := \linhull\{\varphi_{\bullet,E}\}$,
it holds that 
  \begin{align}\label{eq:decomp}
    \SX_L = \SX_0 + \sum_{\ell=1}^L \sum_{E\in\EE_\ell^\star} \SX_{\ell,E}
    \quad \text{for all } L \in \N_0. 
    \qquad\qed
  \end{align}
\end{lemma}

Additive Schwarz preconditioners are based on (not necessarily direct) subspace decompositions.
Following the standard theory
(see, e.g.,~\cite[Chapter~2]{toswid}), 
\eqref{eq:decomp} yields a (local multilevel) preconditioner.
To provide its matrix formulation, let $\embed_{k,\ell}\in\R^{\#\TT_\ell\times\#\TT_k}$ be the 
matrix 
representation of the canonical embedding $\PP^0(\TT_k) \hookrightarrow \PP^0(\TT_\ell)$ for $k<\ell$, i.e.,
\begin{align*}
\sum_{i=1}^{\#\TT_k}\x_k[i]\,\chi_{k,i}=\sum_{i=1}^{\#\TT_\ell}\x_\ell[i]\,\chi_{\ell,i}\quad\textrm{for all $\x_k\in\R^{\#\TT_k}$ and $\x_\ell:=\embed_{k,\ell}\x_k\in\R^{\#\TT_{\ell}}$.}
\end{align*}
Let $\Hmat_\ell \in \R^{\#\TT_\ell\times \#\EE_\ell}$ denote the matrix that represents Haar-type functions,
i.e., 
\begin{align*}
\varphi_{\ell,E_j}=\sum_{i=1}^{\#\TT_\ell}\Hmat_\ell[i,j]\chi_{\ell,i}\quad\textrm{for all }E_j\in\EE_\ell.
\end{align*}
Since only two coefficients per column are non-zero, 
$\Hmat_\ell$ is sparse,  while $\embed_{k,\ell}$ is non-sparse in general. Finally, define the (non-invertible) diagonal matrix $\D_\ell \in \R^{\#\EE_\ell\times\#\EE_\ell}$ by
\begin{align*}
  (\D_\ell)_{jk} := 
  \begin{cases}
    \enorm{\varphi_{\ell,E_j}}^{-2} & E_j\in \EE_\ell^\star\textrm{ and }j=k, \\
    0 & \text{else}.
  \end{cases}
\end{align*}
Then, the matrix representation of the preconditioner associated to 
\eqref{eq:decomp} reads
\begin{align}\label{eq:defPrec}
  \P_L^{-1} := \embed_{0,L}\A_0^{-1}\embed_{0,L}^T + \sum_{\ell=1}^L 
  \embed_{\ell,L} \Hmat_\ell \D_\ell \Hmat_\ell^T \embed_{\ell,L}^T.
\end{align}
For $d=2$, the subsequent Theorem~\ref{thm:precond} is already proved in~\cite[Section~III.B]{MR3634453} for $\Gamma = \partial\Omega$ and in~\cite[Section~6.3]{dissFuehrer} for $\Gamma \subsetneqq \partial\Omega$. For $d = 3$, we need the following additional assumptions: 
\begin{explain}
\item First, suppose that $\Omega \subset \R^3$ is simply connected and $\Gamma = \partial\Omega$. 
\item Second, let $\widehat\TT_0$ be a conforming triangulation of $\Omega$ into non-degenerate compact simplices such that $\TT_0 = \widehat\TT_0|_\Gamma$ is the induced boundary partition on $\Gamma$.
\end{explain}
Then, the following theorem is our first main result. The proof is given in
Section~\ref{proof:precond}.

\begin{theorem}\label{thm:precond}
Under the foregoing assumptions, the preconditioner $\P_L$ from~\eqref{eq:defPrec} is
optimal, i.e., there holds~\eqref{eq1:pcg}, where $\Cpcg\geq1$ depends only on $\Omega$ and $\widehat\TT_0$, but is independent of $L\in\N$.
\end{theorem}

We stress that the matrix in~\eqref{eq:defPrec} will never be assembled in practice. The PCG algorithm only needs the action of $\P_L^{-1}$ on a vector. This can be done recursively by using the embeddings $\embed_{\ell,\ell+1}$ which are, in fact, sparse. Up to (storing and) inverting $\A_0$ on the coarse mesh, the evaluation of $\P_L^{-1}\x$ can be done in $\OO(\#\TT_L)$ operations; see, e.g.,~\cite[Section~3.1]{MR3612925} for a detailed discussion.
If the mesh $\TT_L$ is fine compared to the initial mesh $\TT_0$ (or if $\A_0$ is realized with, e.g., $\mathcal{H}$-matrix techniques), then the computational costs and storage requirements associated with $\A_0$ can be neglected.

\begin{remark}
Our proof for $d = 3$ requires additional assumptions on $\Omega$, $\Gamma = \partial\Omega$, and $\TT_0$. As stated above, the case $d = 2$ allows for a different proof (which, however, does not transfer to $d = 3$) and can thus avoid these assumptions; see~\cite{MR3634453,dissFuehrer}. We believe that Theorem~\ref{thm:precond} also holds for $d = 3$ and $\Gamma \subsetneqq \partial \Omega$. This is also underpinned by a numerical experiment in Section~\ref{subsection:screen}. The mathematical proof, however, remains open.
\end{remark}

\def\qest{q_{\rm est}}
\def\qlin{q_{\rm lin}}
\def\qred{q_{\rm red}}
\def\qctr{q_{\rm ctr}}
\def\Cest{C_{\rm est}}
\def\Cred{C_{\rm red}}
\def\Crel{C_{\rm rel}}
\def\Cdrl{C_{\rm drl}}
\def\Cstb{C_{\rm stb}}
\def\Cson{C_{\rm son}}
\def\HH{\mathcal{H}}
\def\T{\mathbb{T}}
\def\eps{\varepsilon}
\def\quasierror{\Lambda}
\subsection{Optimal convergence of adaptive algorithm}
\label{section:main_results_algorithm}%
We analyze the following adaptive strategy which is driven by the weighted-residual error estimator~\eqref{eq:estimator}. We note that Algorithm~\ref{algorithm} as well as the following results are independent of the precise preconditioning strategy as long as the employed preconditioners are optimal; see Section~\ref{section:optimal_preconditioners}.

\begin{algorithm}\label{algorithm}
{\bfseries Input:} Conforming triangulation $\TT_0$ of $\Gamma$, adaptivity parameters $0<\theta\le1$ and $\lambda > 0$, and $\Cmark > 0$, optimal preconditioning strategy $\P_\coarse$ for all $\TT_\coarse \in \refine(\TT_0)$.\\
{\bfseries Loop:} With $k := 0 =: j$ and $\PPhi_{00} := 0$, iterate the following steps~{\rm(i)--(vii)}:
\begin{itemize}
\item[\rm(i)] Update counter $(j,k) \mapsto (j,k+1)$.
\item[\rm(ii)] Do one step of the PCG algorithm with the optimal preconditioner $\P_j$ to obtain $\PPhi_{jk} \in \PP^0(\TT_j)$ from $\PPhi_{j(k-1)} \in \PP^0(\TT_j)$.
\item[\rm(iii)] Compute the local contributions $\eta_j(T,\PPhi_{jk})$ of the error estimator for all $T \in \TT_j$.
\item[\rm(iv)] If $\enorm{\PPhi_{jk} - \PPhi_{j(k-1)}} > \lambda \, \eta_j(\PPhi_{jk})$, continue with {\rm(i)}.
\item[\rm(v)] Otherwise, define $\k(j):=k$ and determine some set $\MM_j \subseteq \TT_j$ with up to the multiplicative factor $\Cmark$ minimal cardinality such that 
$\theta \, \eta_j(\PPhi_{jk}) \le \eta_j(\MM_j, \PPhi_{jk})$.
\item[\rm(vi)] Generate $\TT_{j+1} := \refine(\TT_j,\MM_j)$ and define $\PPhi_{(j+1)0} := \PPhi_{jk}$.
\item[\rm(vii)] Update counter $(j, k) \mapsto (j+1, 0)$ and continue with~{\rm(i)}. 
\end{itemize}
{\bfseries Output:} Sequences of successively refined triangulations $\TT_j$, discrete solutions $\PPhi_{jk}$, and corresponding error estimators $\eta_j(\PPhi_{jk})$, for all $j \ge 0$ and $k \ge 0$.\qed
\end{algorithm}

\begin{remark}
The choice $\lambda = 0$ corresponds to the case that 
\eqref{eq:galerkin} is solved exactly, i.e., $\PPhi_{(j+1)0} = \PPhi_j^\exact$. Then, optimal convergence of Algorithm~\ref{algorithm} has already been proved in~\cite{fkmp13,gantumur,cmam,partOne} for weakly-singular integral equations and~\cite{gantumur,partTwo} for hyper-singular integral equations.
The choice $\theta = 1$ will generically lead to uniform mesh-refinement, where for each mesh all elements $\MM_j = \TT_j$ are refined in step~{\rm(vi)} of Algorithm~\ref{algorithm}. Instead, small $0 < \theta \ll 1$, will lead to highly adapted meshes.
\end{remark}

\begin{remark}
Let $\QQ := \set{(j, k) \in \N_0 \times \N_0}{\text{index $(j, k)$ is used in Algorithm~\ref{algorithm}}}$. It holds that $(0,0) \in \QQ$. Moreover, for $j, k\in \N_0$, it holds that
\begin{explain}
\item for $j\ge1$, $(j, 0) \in \QQ$ implies that $(j-1, 0) \in \QQ$,
\item for $k\ge1$, $(j, k) \in \QQ$ implies that $(j, k-1) \in \QQ$.
\end{explain}
If $j$ is clear from the context, we abbreviate $\k := \k(j)$, e.g., $\PPhi_{j\k} := \PPhi_{j\k(j)}$. In particular, it holds that $\PPhi_{j\k} = \PPhi_{(j+1)0}$. 
Since PCG (like any Krylov method) provides the exact solution after at most $\#\TT_j$ steps, it follows that $1 \le \k(j) < \infty$.
Finally, we define the ordering
\begin{align*}
 (j',k') < (j,k) 
 \quad \stackrel{\rm def}\Longleftrightarrow \quad
 \left\{\begin{array}{ll}
  \text{either:}& j' < j\\
  \text{or:} & j' = j \text{ and } k' < k
 \end{array}\right\}
 \qquad \text{for all } (j', k'), (j, k) \in \QQ.
\end{align*}
Moreover, let 
\begin{align}
 |(j,k)| := 
 \begin{cases}
 0, & \text{if } j = 0 = k, \\
 \#\set{(j',k') \in \QQ}{(j',k') < (j,k) \text{ and } k'<\k(j^\prime)},
 & \text{if } j > 0 \text{ or } k > 0,
 \end{cases}
\end{align}
be the total number of PCG iterations until the computation of $\phi_{jk}$.
Note that $j'>j$ and $|(j',k')| = |(j,k)|$ imply that $j'=j+1$, $k=\k(j)$, and $k'=0$ and hence $\PPhi_{j'k'} = \PPhi_{jk}$.\qed
\end{remark}

\begin{theorem}\label{theorem:algorithm}
The output of Algorithm~\ref{algorithm} satisfies the following assertions~{\rm(a)--(c)}.
The constants $\Crel^\star, \Ceff^\star>0$ depend only on $\qpcg$, $\Gamma$, and the uniform $\gamma$-shape regularity of $\TT_j \in \refine(\TT_0)$, whereas
$\Clin>1$ and $0<\qlin<1$ depend additionally only on 
$\theta$ and $\lambda$,
and $\Copt>0$ depends additionally only on
$s$, $\TT_0$, and $\quasierror_{0\k}$.

{\rm(a)} There exists a constant $\Crel^\star > 0$ such that
\begin{align}\label{eq:algorithm:reliability}
 \enorm{\phi^\exact - \PPhi_{jk}}
 \le \Crel^\star \, \big( \eta_j(\PPhi_{jk}) + \enorm{\PPhi_{jk} - \PPhi_{j(k-1)}} \big)
 \quad \text{for all } (j,k) \in \QQ \text{ with } k \ge 1.
\end{align}
There exists a constant $\Ceff^\star > 0$ such that, provided that $\phi^\exact \in L^2(\Gamma)$, it holds that
\begin{align}\label{eq:algorithm:efficiency}
 \eta_j(\PPhi_{jk})
 \le \Ceff^\star \, \big( \norm{h_j^{1/2}(\phi^\exact - \PPhi_{jk})}{L^2(\Gamma)}
 + \enorm{\PPhi_{jk} - \PPhi_{j(k-1)}} \big)
 \text{ for all } (j,k) \in \QQ, k \ge 1.
\end{align}

{\rm(b)} For arbitrary $0 < \theta \le 1$ and arbitrary $\lambda > 0$, there exist constants $\Clin \ge 1$ and $0 < \qlin < 1$ such that the quasi-error
\begin{align}\label{eq:def:quasierror}
 \quasierror_{jk}^2 := \enorm{\phi^\star - \PPhi_{jk}}^2 + \eta_{j}(\PPhi_{jk})^2
\end{align}
is linearly convergent in the sense of
\begin{align}\label{eq1:theorem:algorithm}
 \quasierror_{j'k'} \le \Clin \, \qlin^{|(j',k')|-|(j,k)|} \, \quasierror_{jk}
 \quad \text{for all } (j,k), (j',k') \in \QQ
 \text{ with } (j',k') \ge (j,k).
\end{align}

{\rm(c)}  For $s > 0$, define the approximation class
\begin{align}
 \norm{\phi^\exact}{\AA_s}
 := \sup_{N \in \N_0} \Big( (N+1)^s \, \min_{\substack{\TT_\coarse \in \refine(\TT_0) \\ \#\TT_\coarse - \#\TT_0 \le N}} \eta_\coarse(\PPhi_\coarse^\star)\Big).
\end{align}
Then, for sufficiently small $0 < \theta \ll 1$ and $0 < \lambda \ll 1$, cf.\ Assumption~\eqref{eq:ass_theta} below, and all $s>0$, it holds that
\begin{align}\label{eq:theorem:opt_rate}
 \norm{\phi^\exact}{\AA_s} < \infty
 \,\,\, \Longleftrightarrow \,\,\,
  \exists \, \Copt > 0:
 \sup_{(j,k) \in \QQ} \big( \#\TT_j - \#\TT_0 + 1 \big)^{s} \, \quasierror_{jk}\leq\Copt\,\norm{\phi^\exact}{\AA_s} < \infty.
\end{align}
\end{theorem}

\begin{remark}\label{remark:quasierror}
By definition, it holds that $\eta_{j}(\PPhi_{jk}) \le \quasierror_{jk}$ for all $(j,k) \in \QQ$. 
If $\PPhi_{jk} \in \{\PPhi_j^\exact,\PPhi_{j\k}\}$, then there also holds the converse inequality $\eta_{j}(\PPhi_{jk}) \simeq \quasierror_{jk}$. To see this, note that $\PPhi_{jk} = \PPhi_j^\exact$ and~\eqref{eq:prelim:reliability} prove that $\quasierror_{jk} \le (1+\Crel) \, \eta_{j}(\PPhi_{jk})$. If $\PPhi_{jk} = \PPhi_{j\k}$, then Theorem~\ref{theorem:algorithm}{\rm(a)} and Step~{\rm(iv)} of Algorithm~\ref{algorithm} prove that $\quasierror_{j\k} \le (1+\Crel^\star) \, \eta_{j}(\PPhi_{j\k}) + \enorm{\PPhi_{j\k} - \PPhi_{j(\k-1)}} \le (1 + \Crel^\star + \lambda) \, \eta_{j}(\PPhi_{j\k})$.
\end{remark}

\subsection{Almost optimal computational complexity}
\label{section:main:complexity}%
Suppose that we use $\HH^2$-matrices for the efficient treatment of the discrete single-layer integral operator. 
Recall that the storage requirements (resp.\ the cost for one matrix-vector multiplication) of an $\mathcal{H}^2$-matrix are of order $\OO(Np^2)$, where $N$ is the matrix size and $p \in \N$ is the local block rank. For $\mathcal{H}^2$-matrices (unlike $\mathcal{H}$-matrices), these costs are, in particular, independent of a possibly unbalanced binary tree which underlies the hierarchical data structure~\cite{hackbusch}.  

For a mesh $\TT_\coarse \in \T$, we employ the local block rank $p = \OO(\log(1+\#\TT_\coarse))$ to ensure that the matrix compression is asymptotically exact as $N = \#\TT_\coarse \to \infty$, i.e., the error between the exact matrix and the $\HH$-matrix decays exponentially fast; see~\cite{hackbusch}. We stress that we neglect this error in the following and assume that the matrix-vector multiplication (based on the $\HH^2$-matrix) yields the exact matrix-vector product.

The computational cost for storing $\A_\coarse$ (as well as for one matrix-vector multiplication) is $\OO((\#\TT_\coarse)\log^2(1+\#\TT_\coarse))$. 
In an idealized optimal case, the computation of $\phi_\coarse^\star$ is hence (at least) of cost $\OO((\#\TT_\coarse)\log^2(1+\#\TT_\coarse))$. 

We consider the computational costs for one step of Algorithm~\ref{algorithm}:
\begin{explain}
\item We assume that one step of the PCG algorithm with the employed optimal preconditioner is of cost $\OO\big((\#\TT_j)\log^2(1+\#\TT_j)\big)$; cf. the preconditioner from Section~\ref{section:oasp}.
\item We assume that we can compute $\eta_j(\psi_j)$ for any $\psi_j \in \PP^0(\TT_j)$ (by means of numerical quadrature) with $\OO\big((\#\TT_j)\log^2(1+\#\TT_j)\big)$ operations. 
\item Clearly, the D\"orfler marking in Step~(v) can be done in $\OO\big((\#\TT_j)\log(1+\#\TT_j)\big)$ operations by sorting. Moreover, for $\Cmark = 2$, Stevenson~\cite{stevenson07} proposed a realization of the D\"orfler marking based on binning, which can be performed at linear cost $\OO(\#\TT_j)$. 
\item Finally, the mesh-refinement in Step~(vi) can  be done in linear complexity $\OO(\#\TT_j)$ if the data structure is appropriate. 
\end{explain}
Overall, one step of Algorithm~\ref{algorithm} is thus done in $\OO((\#\TT_j)\log^2(1+\#\TT_j))$ operations.
However, an adaptive step $(j',k') \in \QQ$ depends on the full history of previous steps.
\begin{explain}
\item
Hence, the cumulative computational complexity for the adaptive step $(j',k') \in \QQ$ is of order $\OO\big(\sum_{(j,k) \le (j',k')} (\#\TT_j)\log^2(1+\#\TT_j)\big)$.
\end{explain}
The following corollary proves that Algorithm~\ref{algorithm} does not only lead to convergence of the quasi-error $\quasierror_{jk}$ 
with optimal rate with respect to the degrees of freedom (see Theorem~\ref{theorem:algorithm}), but also 
with \emph{almost} optimal rate with respect to the computational costs.

\def\opt#1{\widehat{#1}}
\begin{corollary}\label{corollary:algorithm}
For $j \in \N_0$, let $\opt\TT_{j+1} = \refine(\opt\TT_j,\opt\MM_j)$ with arbitrary $\opt\MM_j \subseteq \opt\TT_j$ and $\opt\TT_0 = \TT_0$. Let $s > 0$ and suppose that the corresponding error estimator $\opt\eta_j(\opt\phi_j^\star)$ converges at rate $s$ with respect to the single-step computational costs, i.e.,
\begin{align}\label{eq:cost:opt}
 \sup_{j \in \N_0} \big[(\#\opt\TT_j) \, \log^2(1+\#\opt\TT_j)\big]^s \, \opt\eta_j(\opt\phi_j^\star) < \infty.
\end{align}
Suppose that $\lambda$ and $\theta$ satisfy the assumptions of Theorem~\ref{theorem:algorithm}{\rm(c)}. Then, the quasi-errors $\quasierror_{jk}$ generated by Algorithm~\ref{algorithm} converge almost at rate $s$ with respect to the cumulative computational costs, i.e.,
\begin{align}\label{eq:cost}
 \sup_{(j',k') \in \QQ} \Big[\sum_{(j,k) \le (j',k')} (\#\TT_j)\log^2(1+\#\TT_j)\big)\Big]^{s-\eps} \, \quasierror_{j'k'} < \infty
 \quad \text{for all } \eps > 0.
\end{align}
\end{corollary}

\section{Numerical experiments}
\label{section:numerics}

In this section, we present numerical experiments that underpin our theoretical findings.
We use lowest-order BEM for direct and indirect formulations in 2D as well as 3D. 
For each problem, we compare the performance of Algorithm~\ref{algorithm} for
\begin{itemize}
\item different values of $\lambda \in \{1,10^{-1},10^{-2},10^{-3},10^{-4}\}$,
\item different values of $\theta \in \{0.2,0.4,0.6,0.8,1\}$,
\end{itemize}
where $\theta = 1$ corresponds to uniform mesh-refinement. In particular, we monitor the
condition numbers of the arising BEM systems for diagonal preconditioning~\cite{amt99}, the proposed 
additive Schwarz preconditioning from Section~\ref{section:oasp}, and no preconditioning. The 2D implementation is based on our
MATLAB implementation {\sc Hilbert}~\cite{hilbert}, while the 3D implementation relies on an extension
of the BEM++ library~\cite{bempp}.

\begin{figure}
	\centering
	\begin{minipage}{.5\textwidth}
	\centering
	\raisebox{-0.5\height}{\includegraphics[width=.9\textwidth]{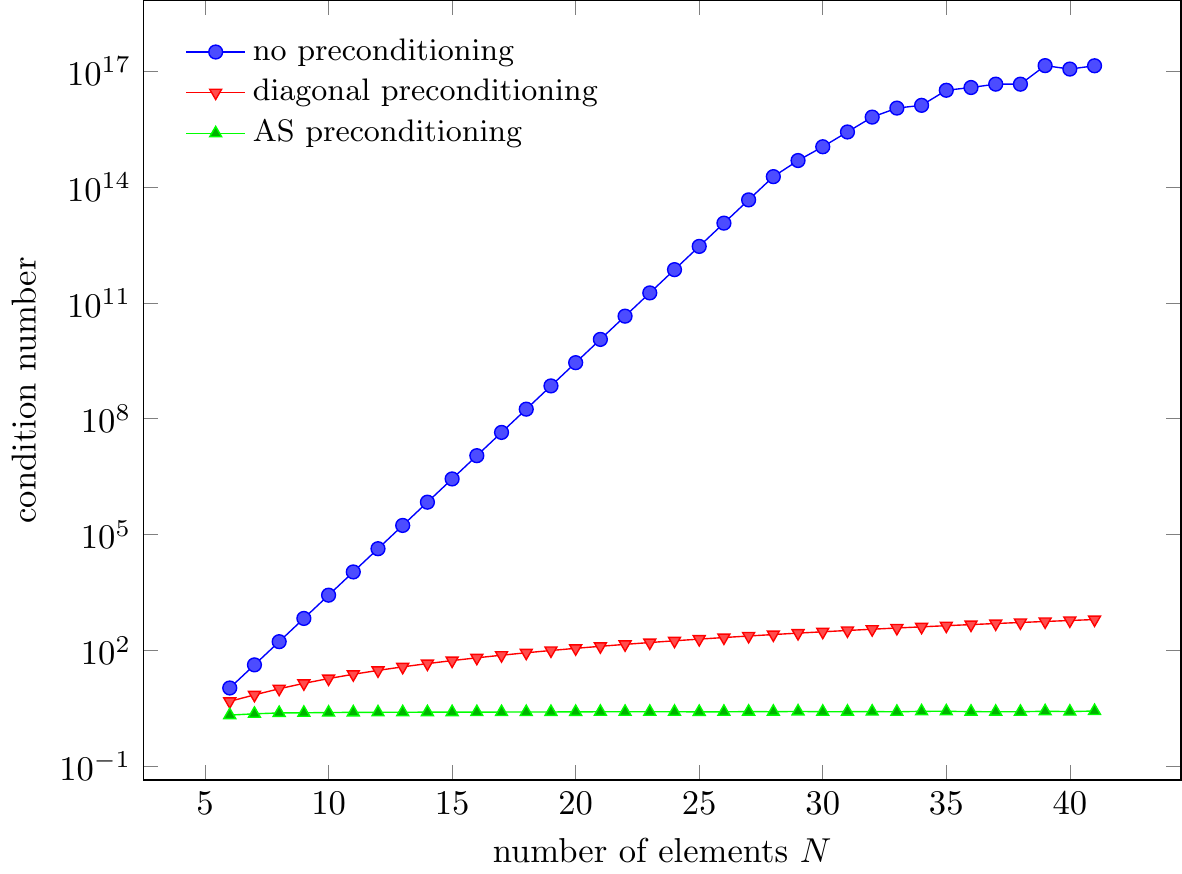}}\\%
	\raisebox{-0.5\height}{\includegraphics[width=.7\textwidth]{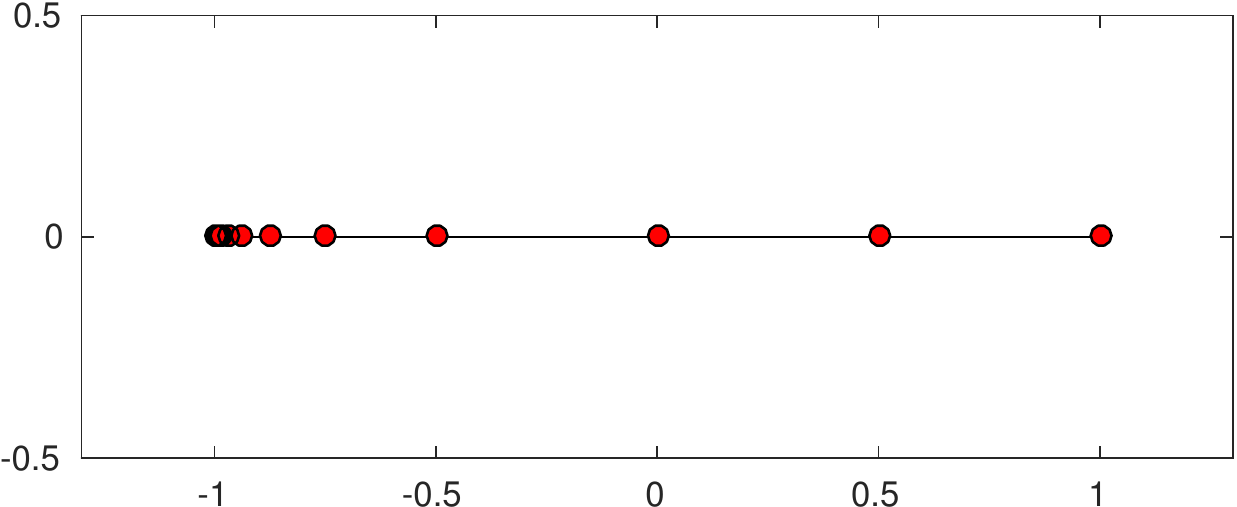}}\hspace*{-5mm}%
	\end{minipage}%
	\hfill%
	\begin{minipage}{.5\textwidth}
	\centering
	\raisebox{-0.5\height}{\includegraphics[width=.9\textwidth]{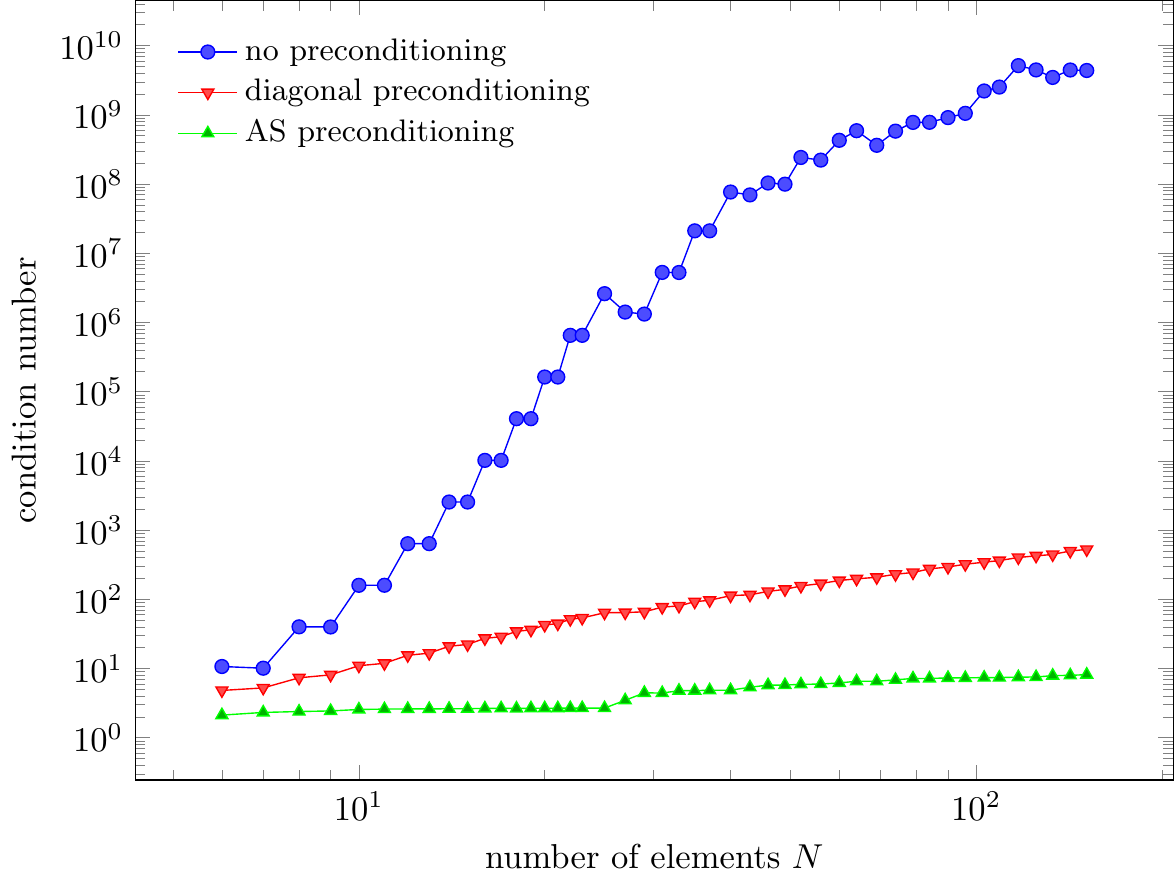}}\\%
	\raisebox{-0.5\height}{\includegraphics[width=.7\textwidth]{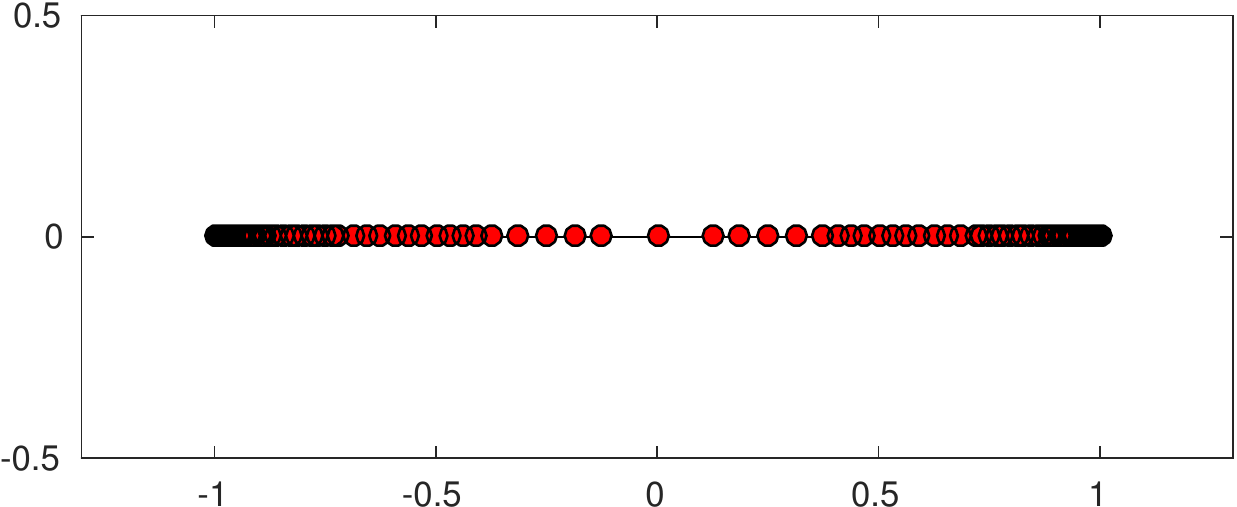}}\hspace*{-5mm}
	\end{minipage}%
	\caption{Example~\ref{subsection:slit}: Condition numbers of the preconditioned and non-preconditioned Galerkin matrix for an artificial refinement towards the left end point (left) and for the matrices arising from Algorithm~\ref{algorithm} (right).}
\label{fig:slit_cond_number}
\end{figure}
\begin{figure}
	\centering
	\raisebox{-0.5\height}{\includegraphics[width=0.48\textwidth]{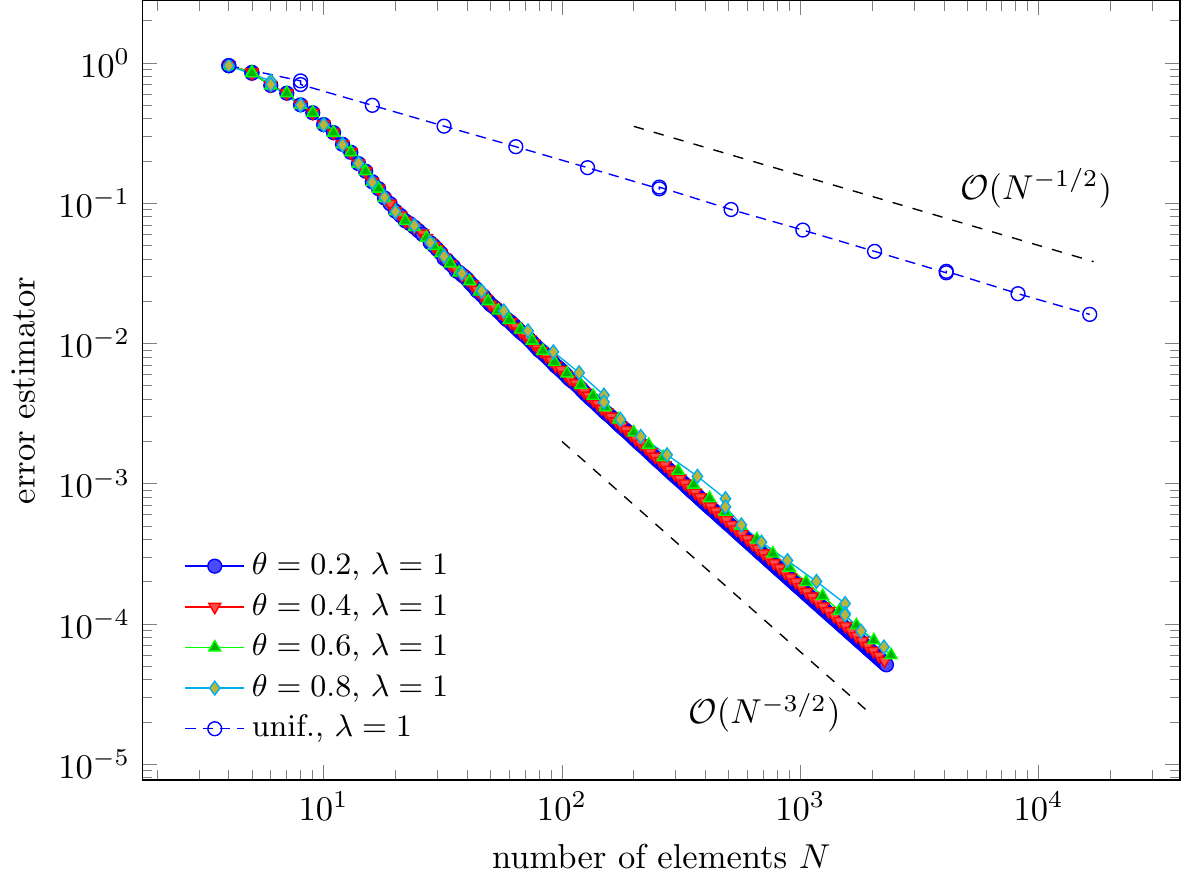}}
	\hfill
	\raisebox{-0.5\height}{\includegraphics[width=0.48\textwidth]{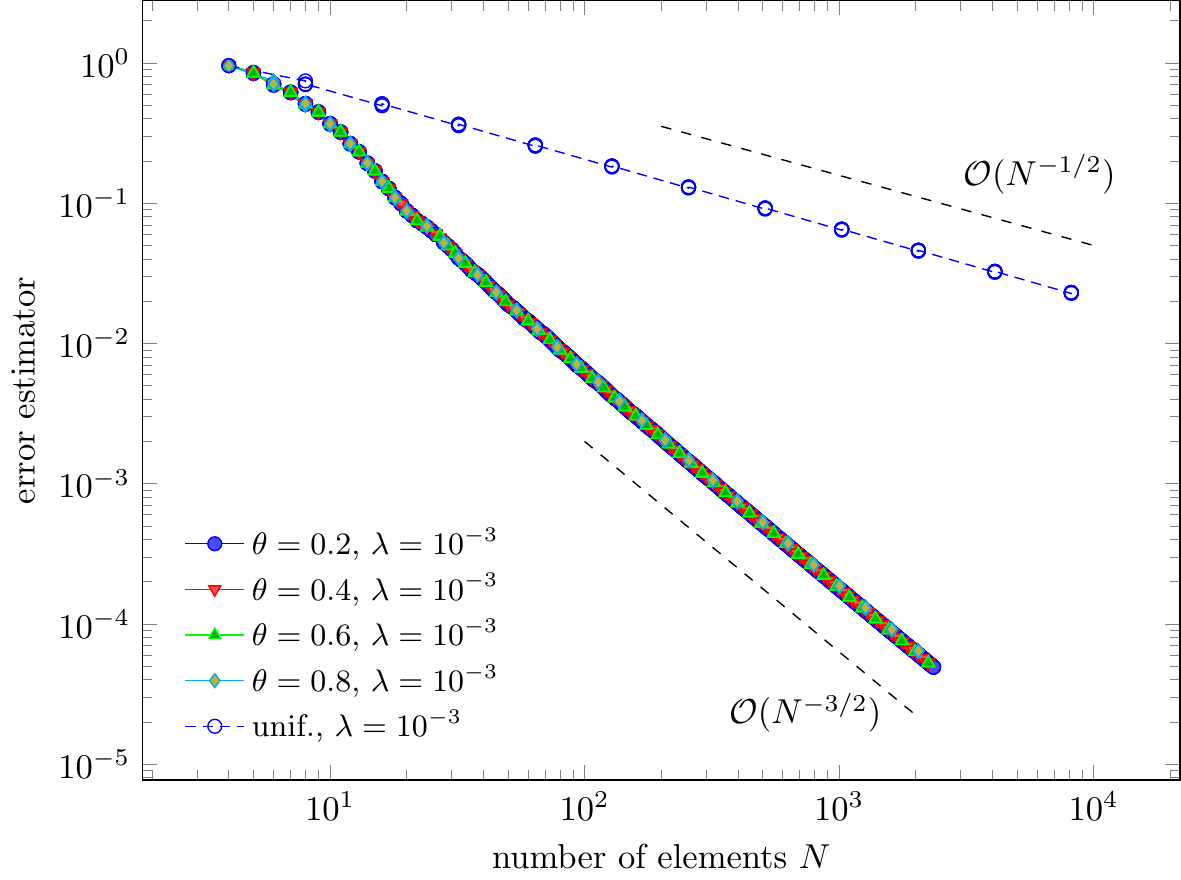}}\vspace{0.5cm}
	\raisebox{-0.5\height}{\includegraphics[width=0.48\textwidth]{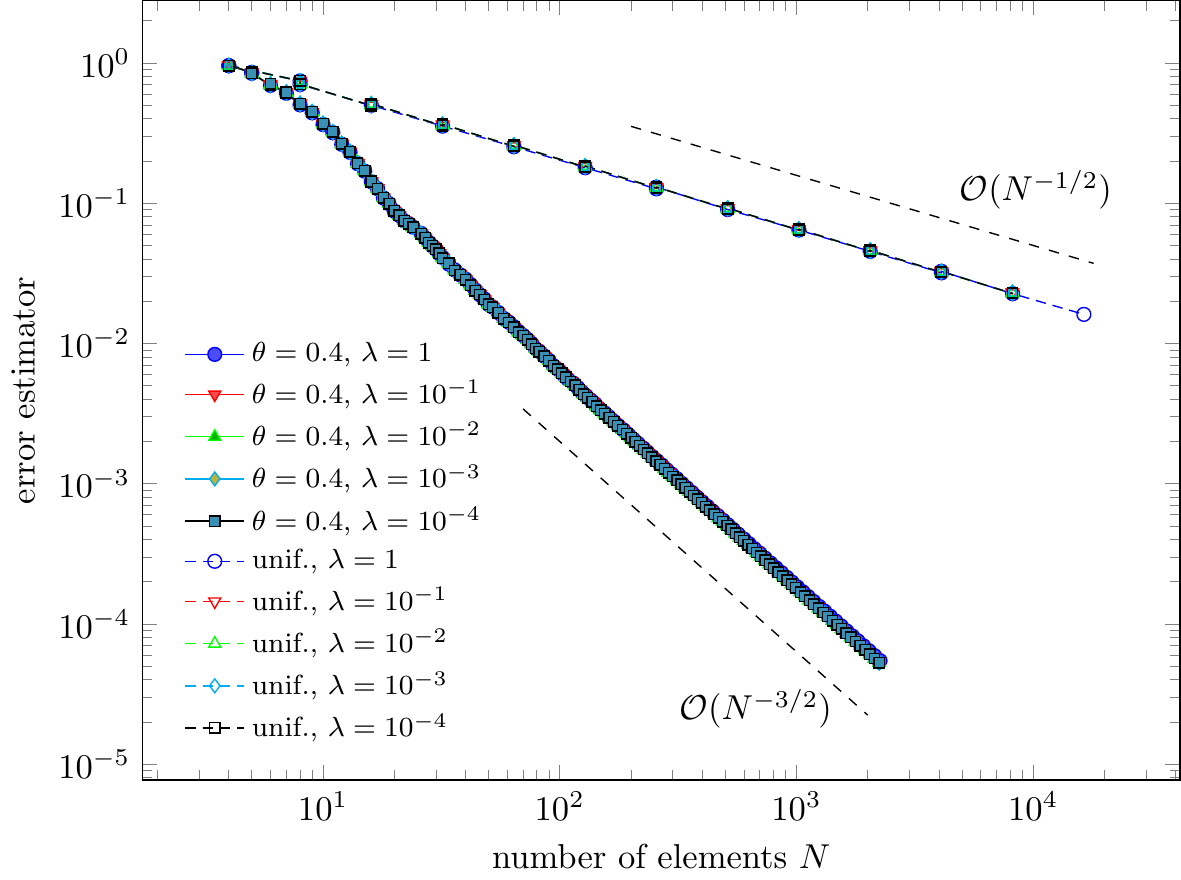}}
	\hfill
	\raisebox{-0.5\height}{\includegraphics[width=0.48\textwidth]{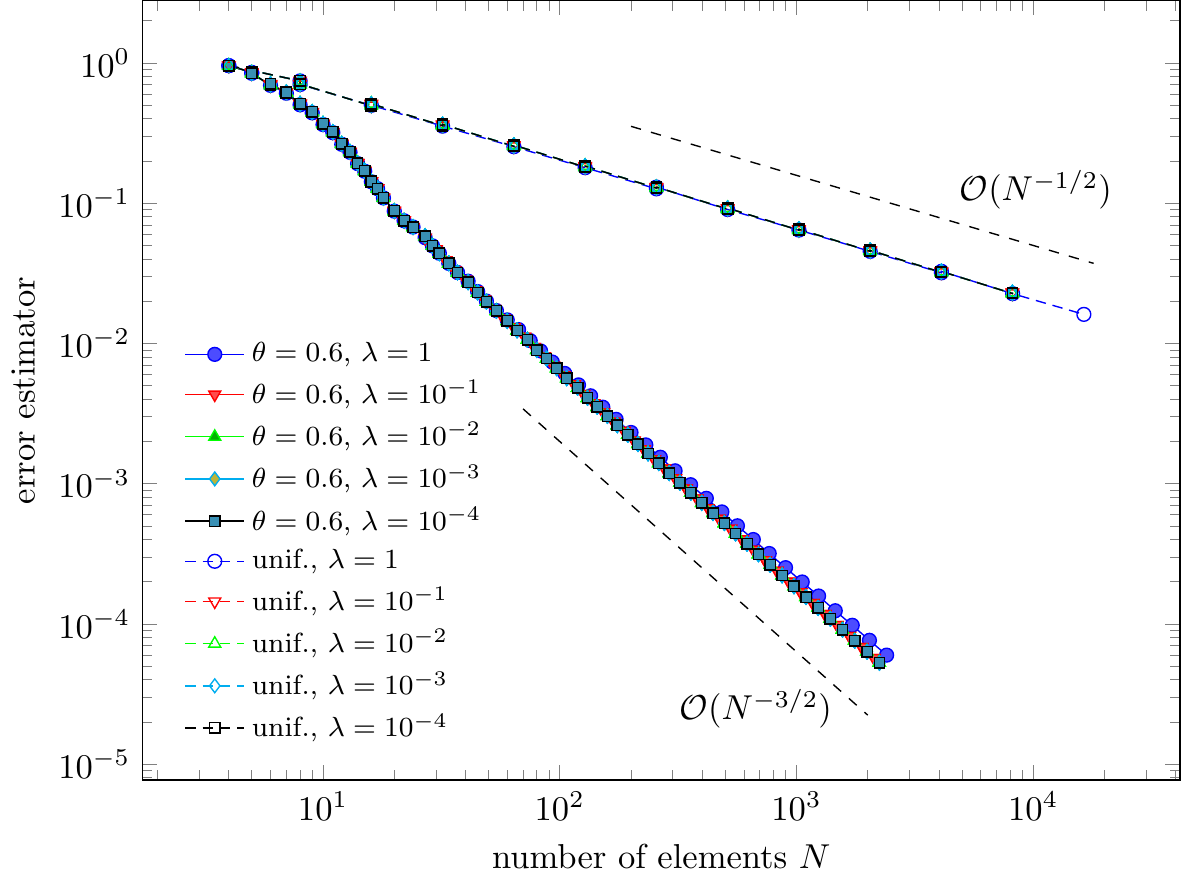}}
	\caption{Example~\ref{subsection:slit}: Estimator convergence for fixed values of $\lambda$ (left: $\lambda=1$, right: $\lambda=10^{-3}$) and $\theta\in\{0.2,0.4,0.6,0.8\}$ (top) and for fixed values of $\theta$ (left: $\theta=0.4$, right: $\theta=0.6$) and $\lambda\in\{1,10^{-1},\ldots,10^{-4}\}$ (bottom).}
\label{fig:slit_conv}
\end{figure}
\begin{figure}
  	\centering
  	\raisebox{-0.5\height}{\includegraphics[width=0.48\textwidth]{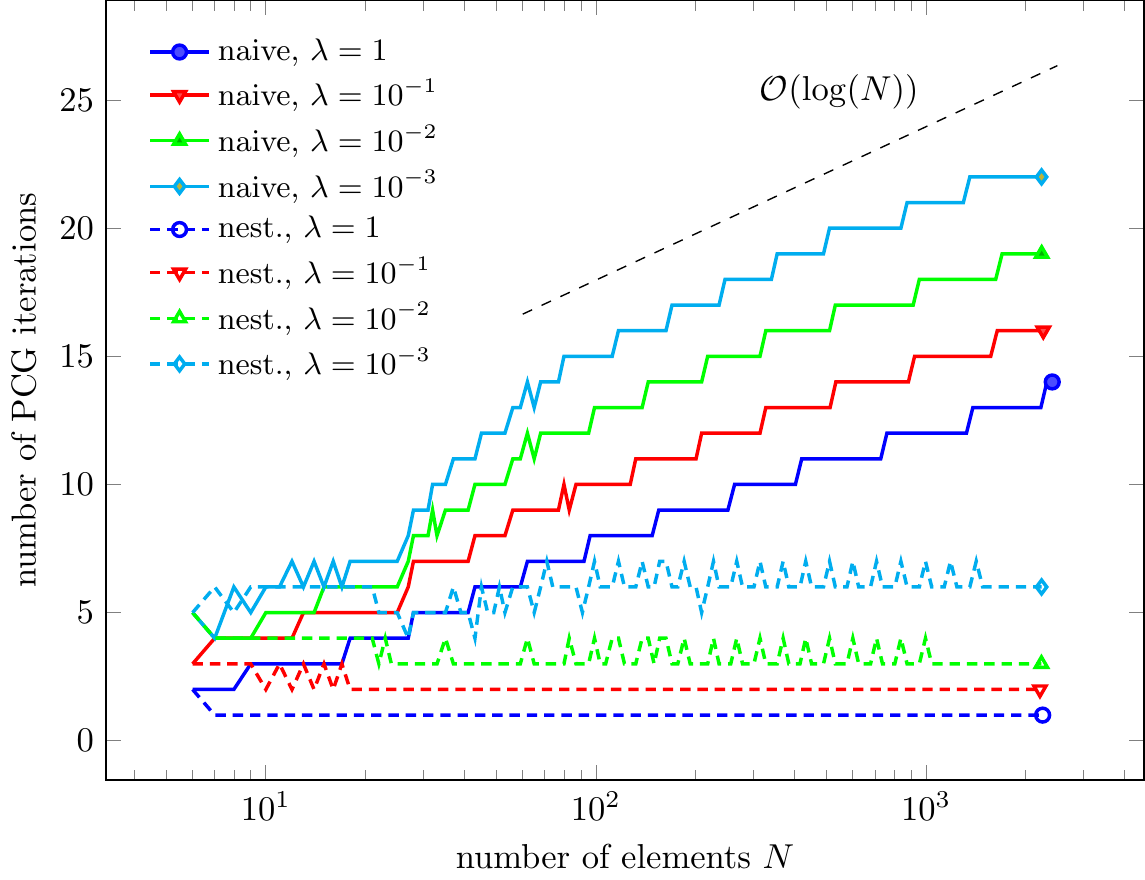}}
	\hfill
  	\raisebox{-0.5\height}{\includegraphics[width=0.48\textwidth]{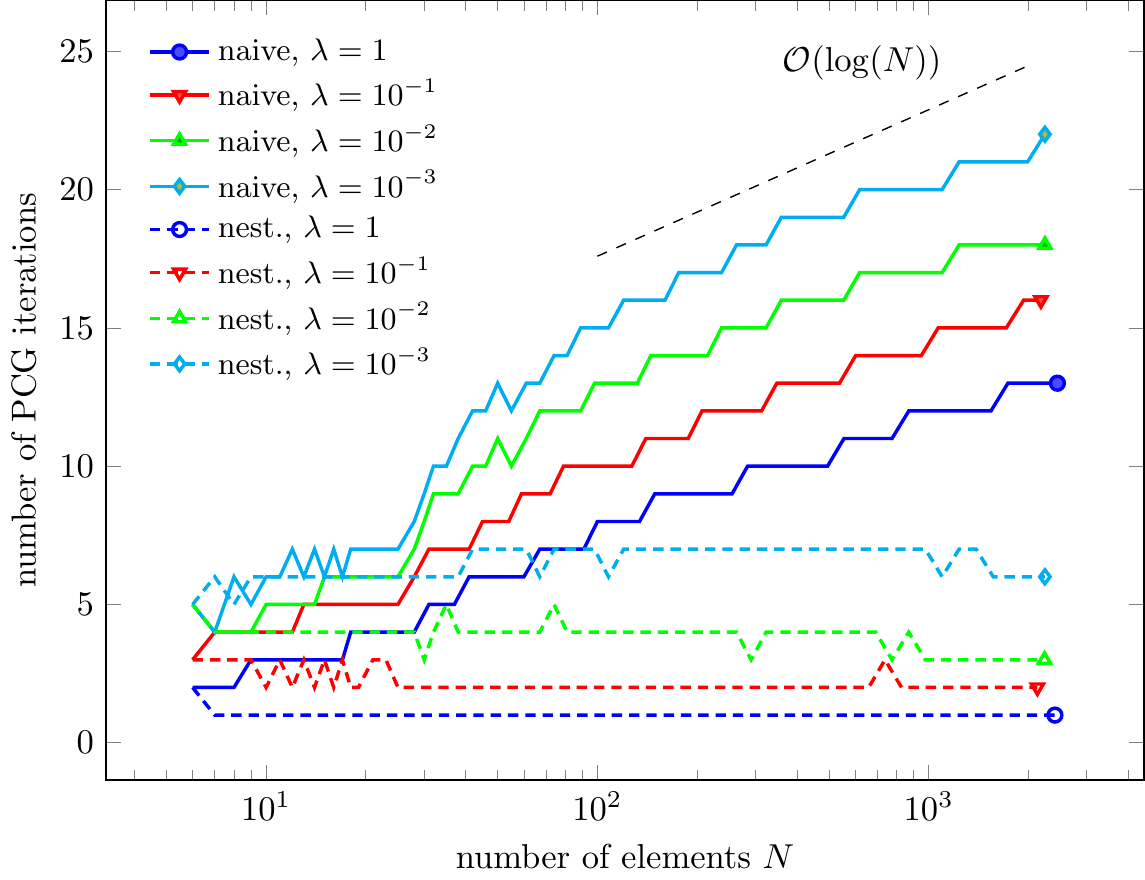}}
  	\caption{Example~\ref{subsection:slit}: Number of PCG iterations in Algorithm~\ref{algorithm} for nested iteration (dashed lines), i.e., $\pphi_{(j+1)0}:=\pphi_{j\k}$ in Step~(vi), and naive initial guess (solid lines), i.e., $\pphi_{(j+1)0}:=0$. We compare fixed values of $\theta$ (left: $\theta=0.4$, right: $\theta=0.6$) and $\lambda=\in\{1,10^{-1},\ldots,10^{-3}\}$.}
  	\label{fig:slit_not_nested}
\end{figure}
%
\subsection{Slit Problem in  2D}
\label{subsection:slit}
\noindent Let $\Gamma:=(-1,1)\times\{0\}$, cf.\ Figure~\ref{fig:slit_cond_number}.
We consider
\begin{align}\label{eq:slit_prob}
V \pphi=1\quad\textrm{ on }\Gamma.
\end{align}
The unique exact solution of \eqref{eq:slit_prob} reads $\phi^\exact(x,0):=-2x/\sqrt{1-x^2}$. For uniform mesh-refinement, we thus expect a convergence order of $\OO(N^{-1/2})$, while the optimal rate is $\OO(N^{-3/2})$ with respect to the number of elements.

In Figure~\ref{fig:slit_conv}, we compare Algorithm~\ref{algorithm} for different values for $\theta$ and $\lambda$ as well as uniform mesh-refinement. 
Uniform mesh-refinement leads only to the rate $\OO(N^{-1/2})$, while adaptivity, independently of the value of $\theta$ and $\lambda$, regains the optimal rate $\OO(N^{-3/2})$.
A naive initial guess in Step~(vi) of Algorithm~\ref{algorithm} (i.e., if $\pphi_{(j+1)0}:=0$) leads to a logarithmical growth of the number of PCG iterations, whereas for nested iteration $\pphi_{(j+1)0}:=\pphi_{j\k}$ (as formulated in Algorithm~\ref{algorithm}) the number of PCG iterations stays uniformly bounded, cf. Figure~\ref{fig:slit_not_nested}.
Finally, Figure~\ref{fig:slit_cond_number} shows the condition numbers for an artificial refinement towards the left end point and for Algorithm~\ref{algorithm} with $\lambda = 10^{-3}$ and $\theta=0.5$.

%
\begin{figure}
	\centering
	\raisebox{-0.5\height}{\includegraphics[width=0.38\textwidth]{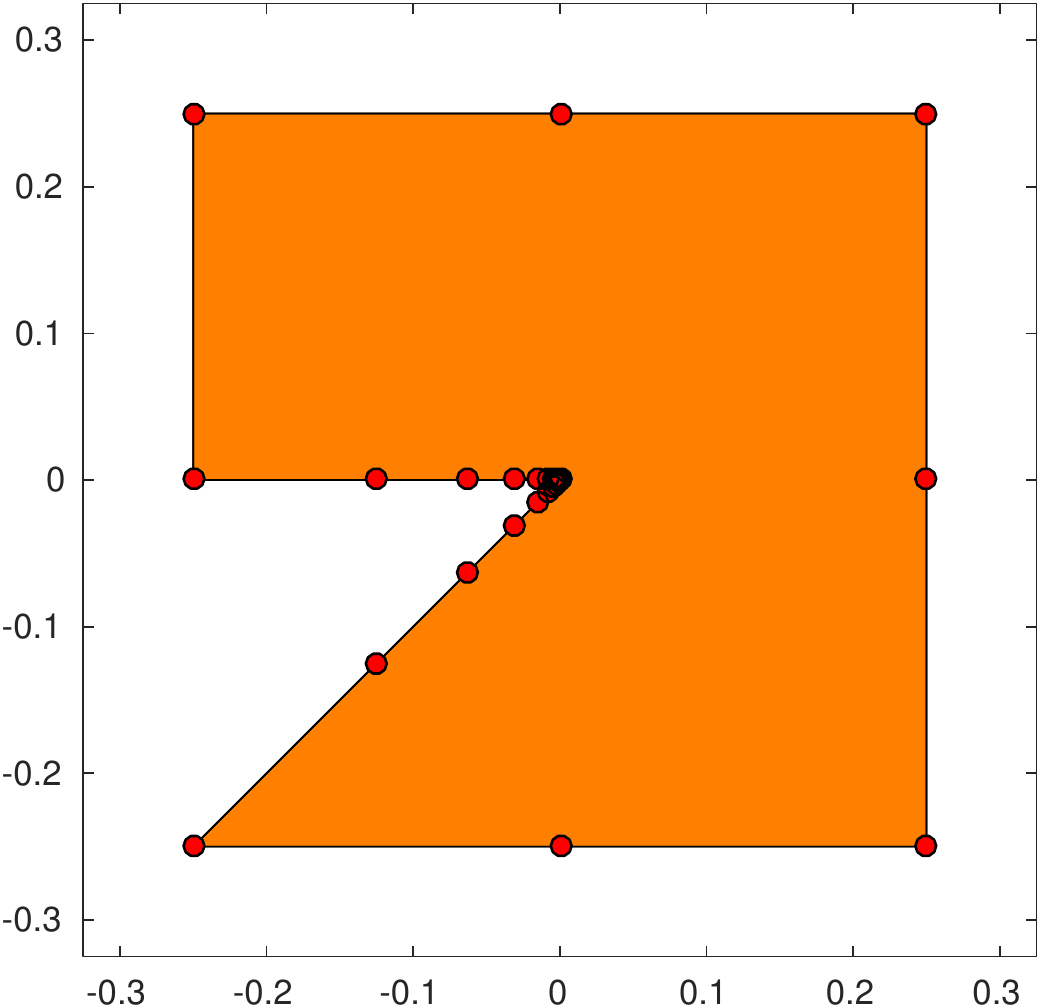}}
	\hfill
	\raisebox{-0.5\height}{\includegraphics[width=0.58\textwidth]{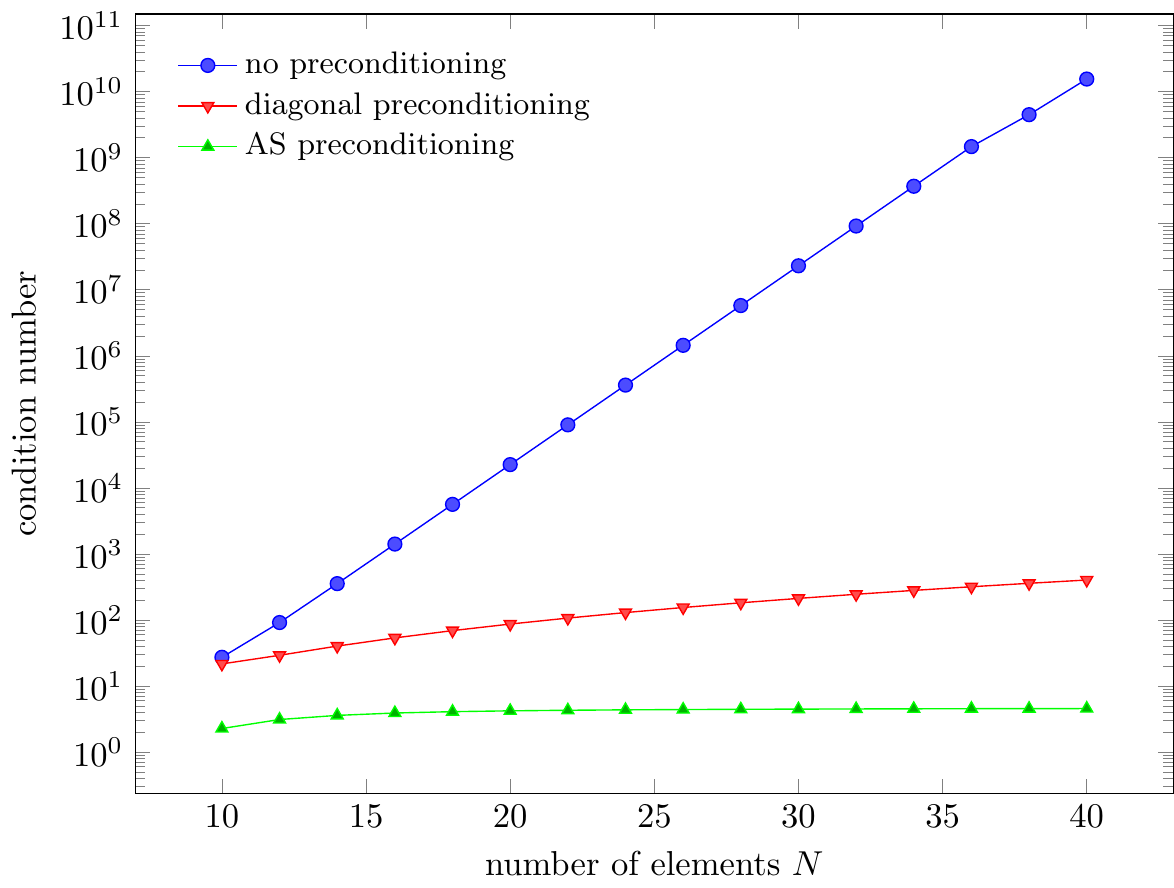}}\vspace{0.5cm}
	\raisebox{-0.5\height}{\includegraphics[width=0.38\textwidth]{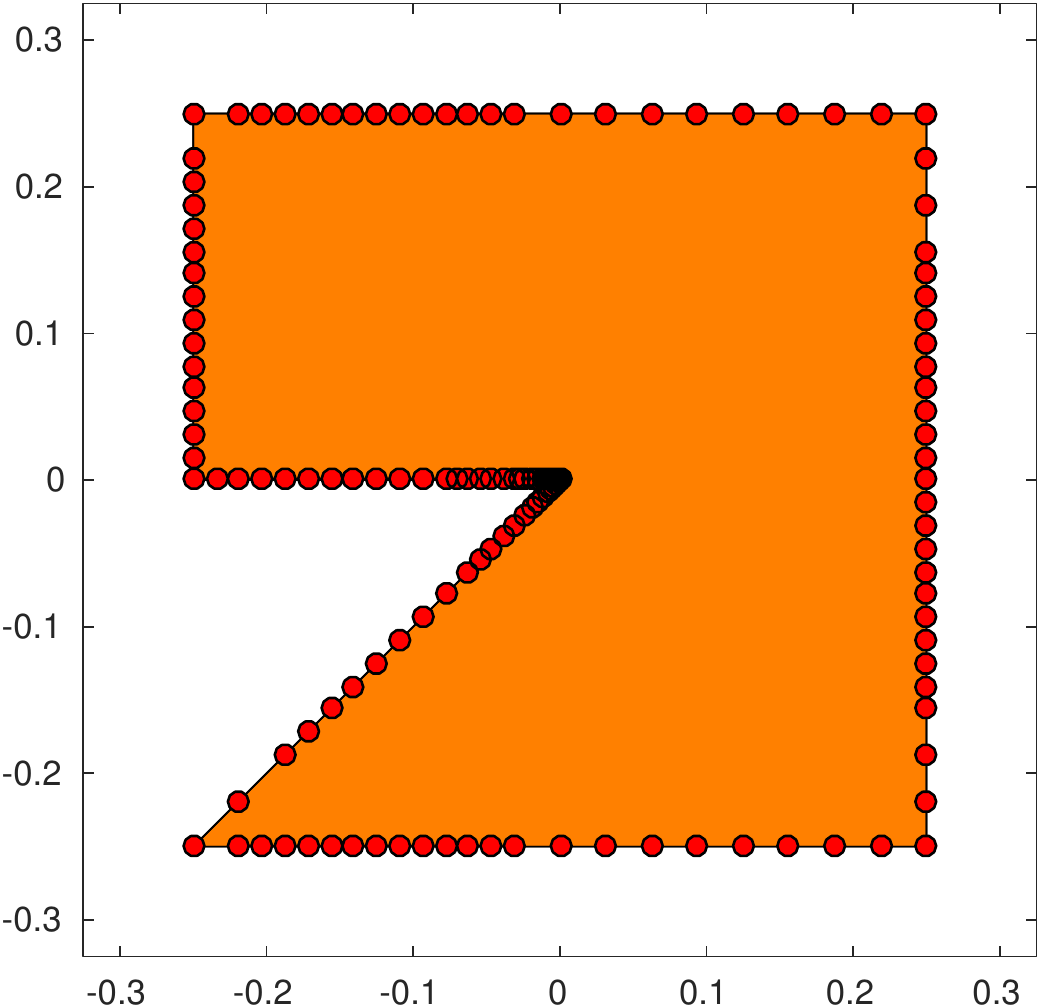}}
	\hfill
	\raisebox{-0.5\height}{\includegraphics[width=0.58\textwidth]{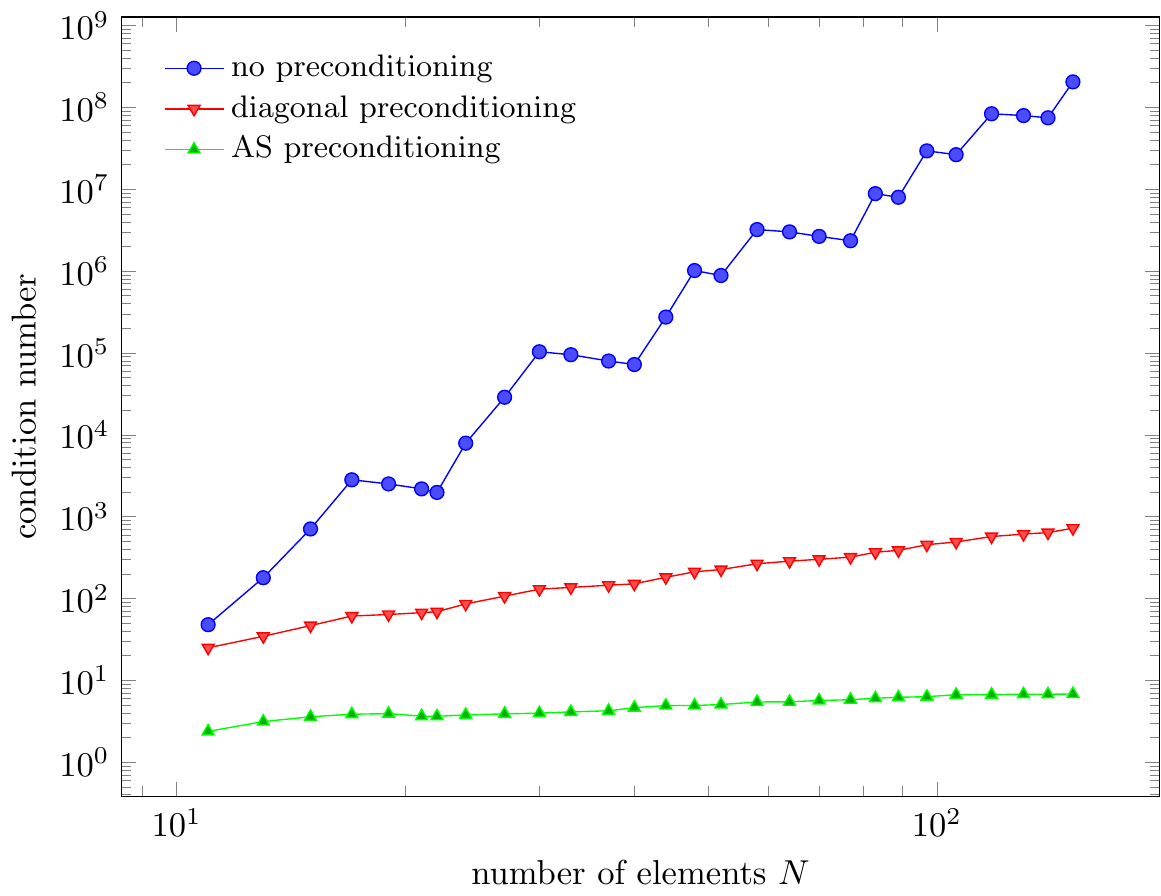}}	
	\caption{Example~\ref{subsection:zshape}: Condition numbers of the preconditioned and non-preconditioned Galerkin matrix for an artificial refinement towards the reentrant corner (top) and for Algorithm~\ref{algorithm} (bottom).}
\label{fig:z_shape_cond_number}
\end{figure}
\begin{figure}
	\centering
	\raisebox{-0.5\height}{\includegraphics[width=0.48\textwidth]{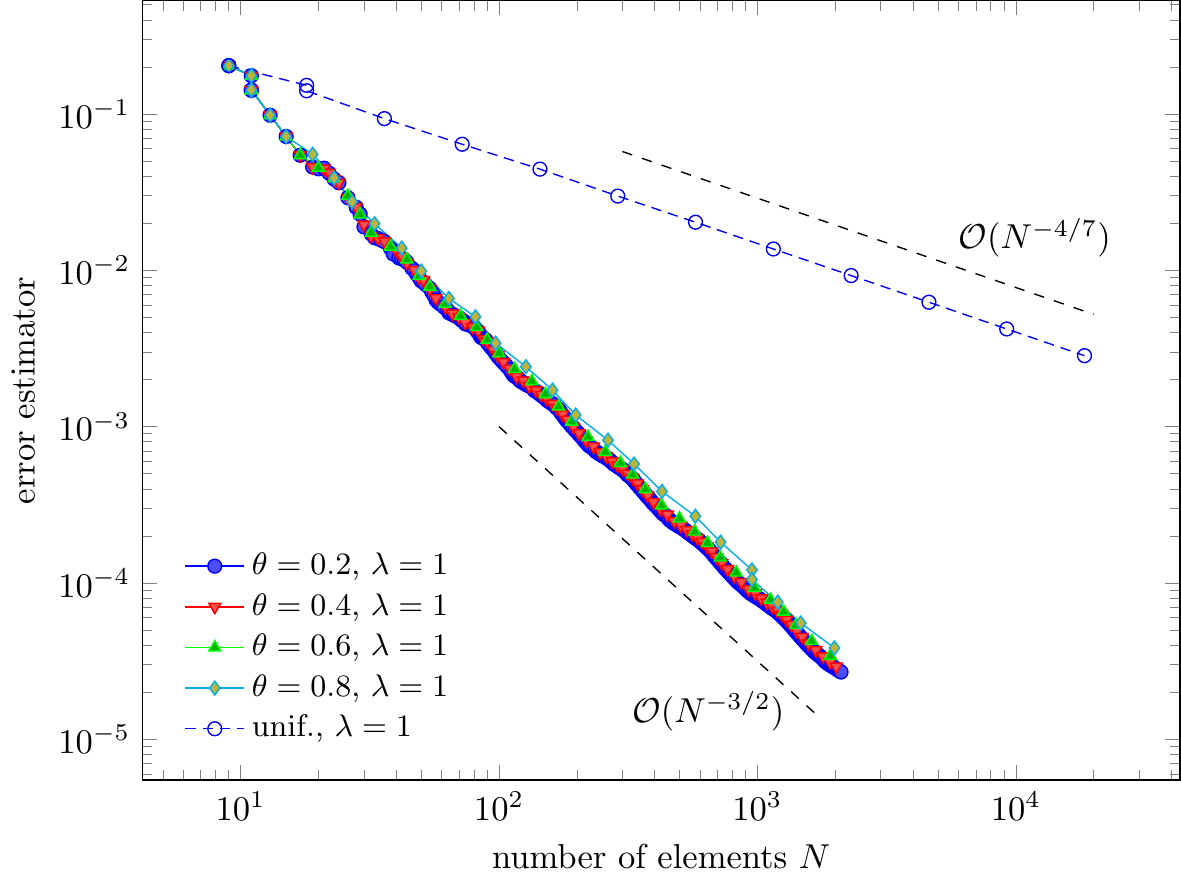}}
	\hfill
	\raisebox{-0.5\height}{\includegraphics[width=0.48\textwidth]{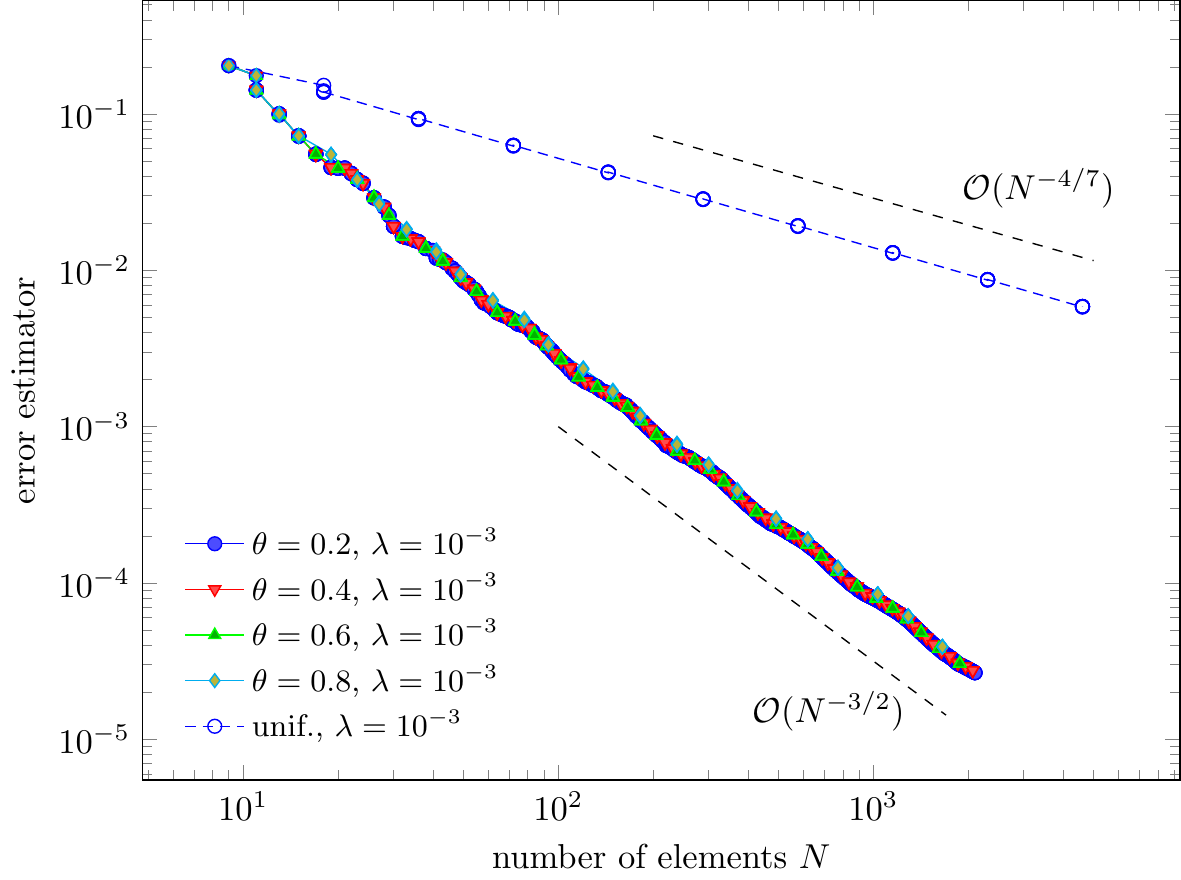}}\vspace{0.5cm}
	\raisebox{-0.5\height}{\includegraphics[width=0.48\textwidth]{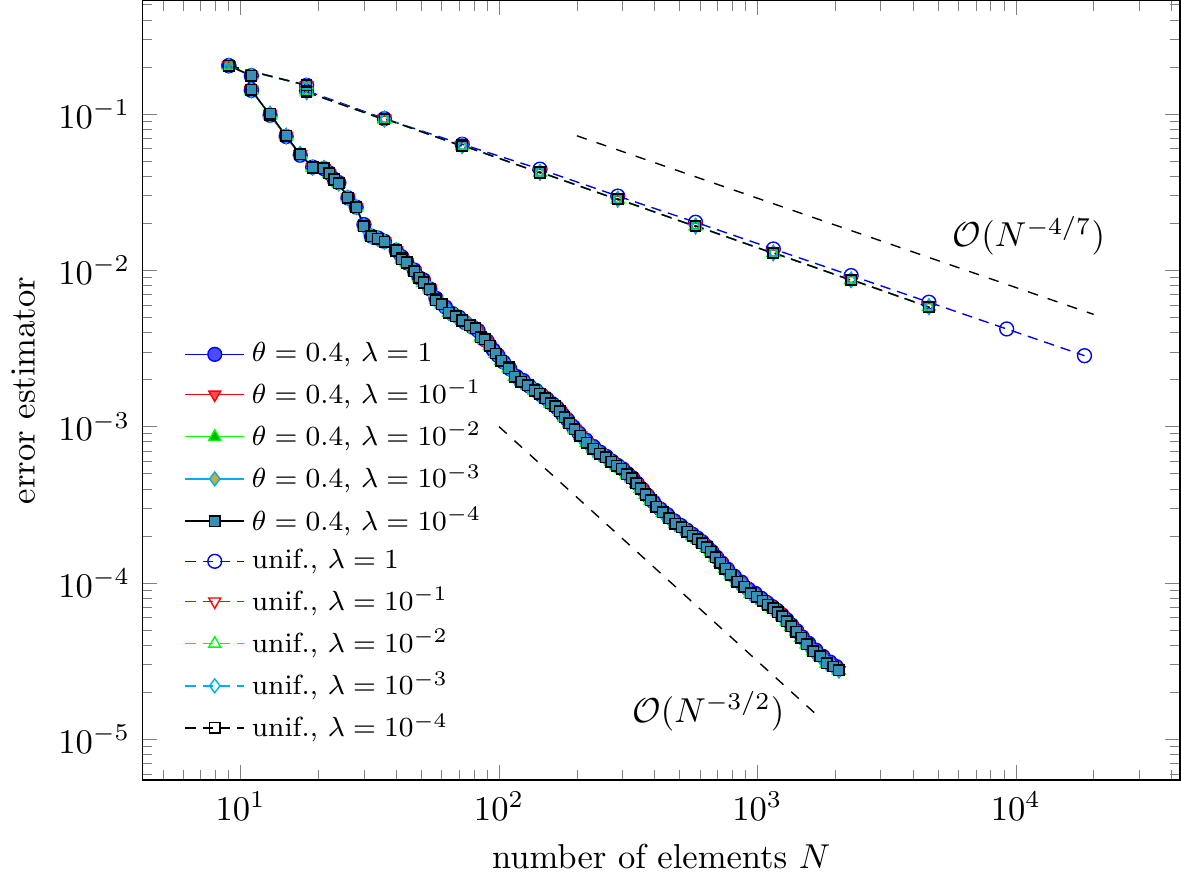}}
	\hfill
	\raisebox{-0.5\height}{\includegraphics[width=0.48\textwidth]{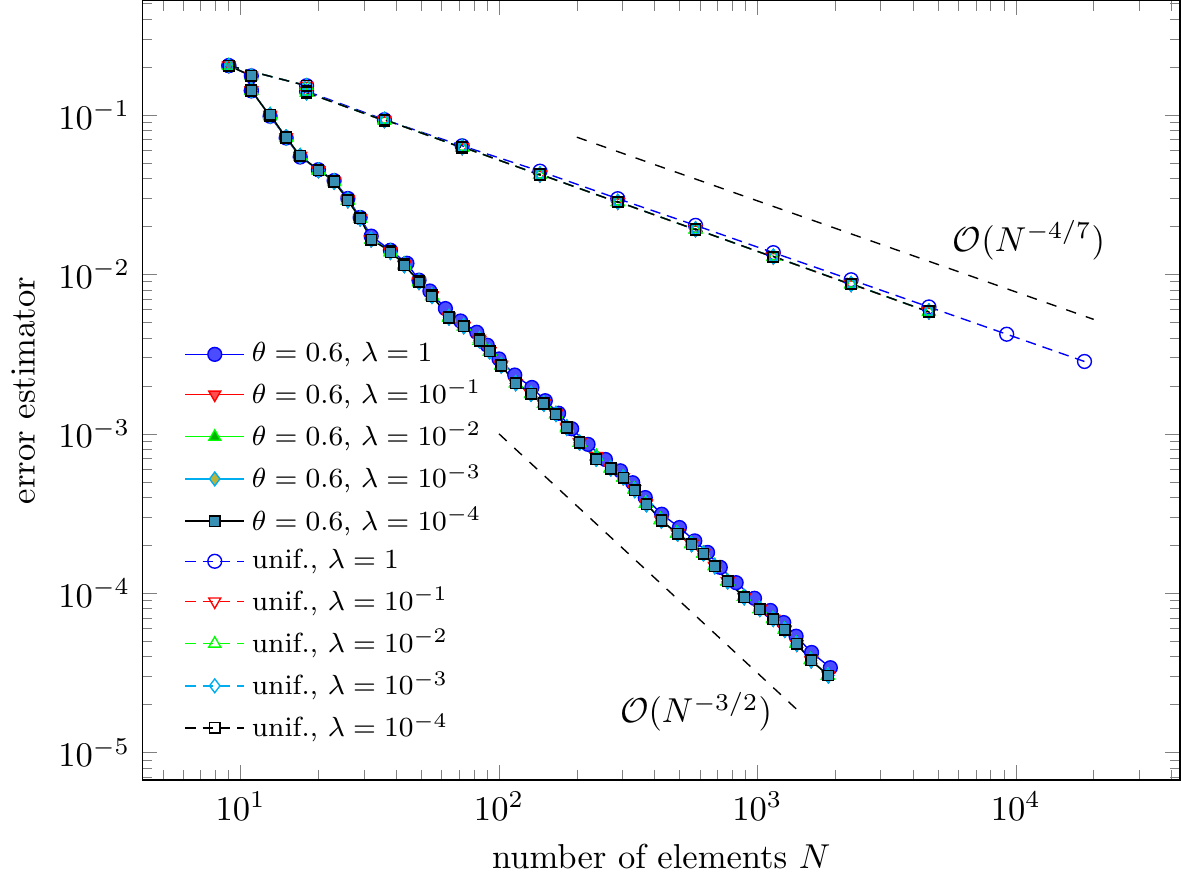}}
	\caption{Example~\ref{subsection:zshape}: Estimator convergence for fixed values of $\lambda$ (left: $\lambda=1$, right: $\lambda=10^{-3}$) and $\theta\in\{0.2,0.4,0.6,0.8\}$ (top) and for fixed values of $\theta$ (left: $\theta=0.4$, right: $\theta=0.6$) and $\lambda\in\{1,10^{-1},\ldots,10^{-4}\}$ (bottom).}
\label{fig:z_shape_conv}
\end{figure}
\begin{figure}
  	\raisebox{-0.5\height}{\includegraphics[width=0.48\textwidth]{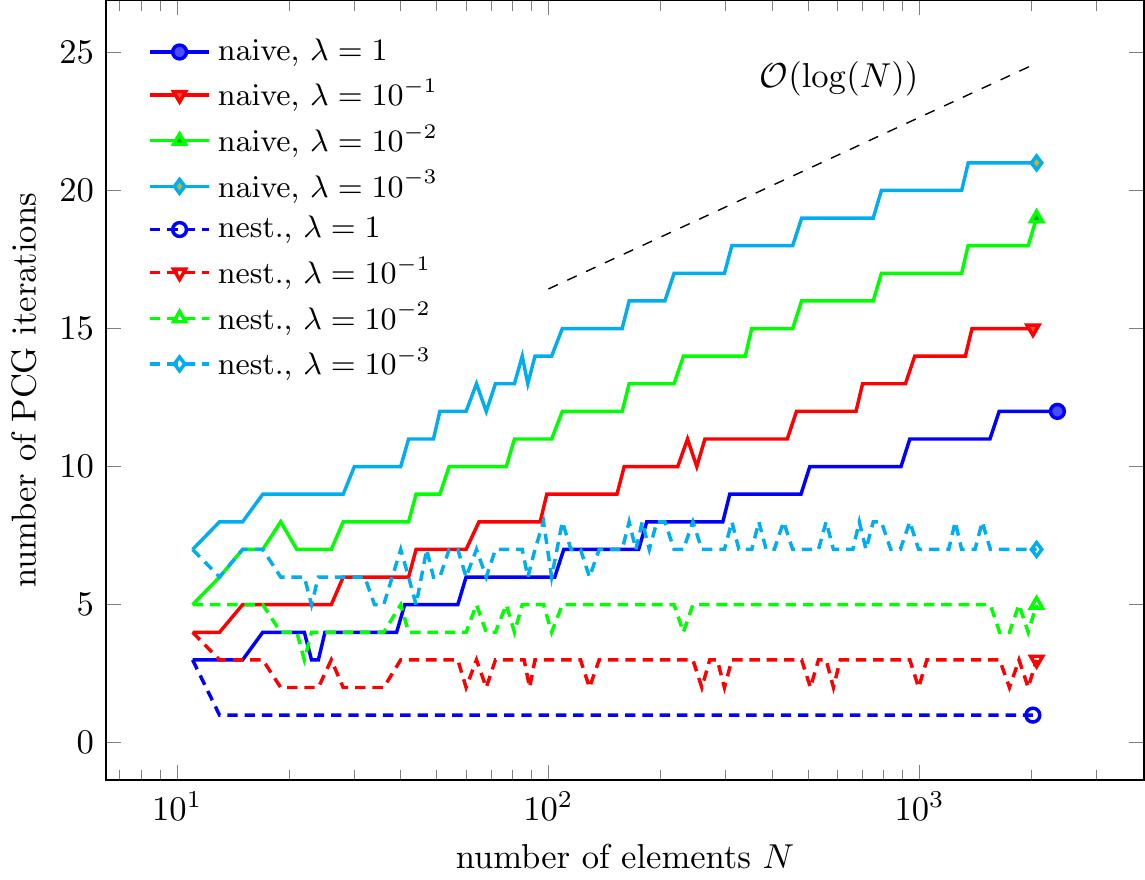}}
	\hfill
  	\raisebox{-0.5\height}{\includegraphics[width=0.48\textwidth]{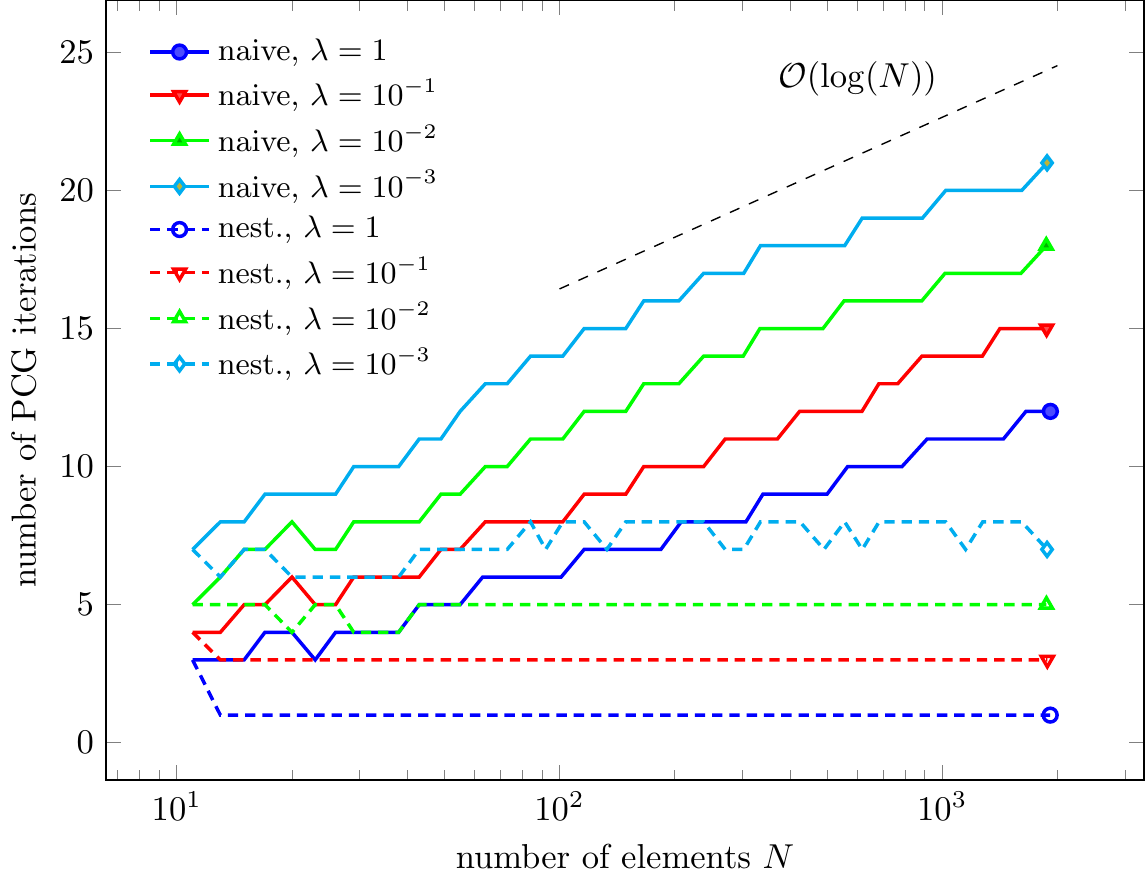}}
  	\caption{Example~\ref{subsection:zshape}: Number of PCG iterations in Algorithm~\ref{algorithm} for nested iteration (dashed lines), i.e., $\pphi_{(j+1)0}:=\pphi_{j\k}$ in Step~(vi), and naive initial guess (solid lines), i.e., $\pphi_{(j+1)0}:=0$. We compare fixed values of $\theta$ (left: $\theta=0.4$, right: $\theta=0.6$) and $\lambda=\in\{1,10^{-1},\ldots,10^{-3}\}$.}
  	\label{fig:z_shape_not_nested}
\end{figure}
%
\subsection{Z-shaped domain in 2D}
\label{subsection:zshape}
\noindent Let $\Gamma:=\partial\Omega$ be the boundary of the Z-shaped domain with reentrant corner at the origin $(0,0)$, cf.\ Figure~\ref{fig:z_shape_cond_number}.
The right-hand side is given by $f=(K+1/2)g$ with the double-layer operator $K\colon H^{1/2}(\Gamma)\to H^{1/2}(\Gamma)$. We note that the weakly-singular integral equation~\eqref{eq:slp} is then equivalent to the Dirichlet problem
\begin{align}\label{eq:dbp}
\begin{split}
-\Delta u &= 0\quad\textrm{ in }\Omega
\qquad\text{subject to}\qquad
u = g\quad\textrm{ on }\Gamma.
\end{split}
\end{align}
We prescribe the exact solution in 2D polar coordinates as
\begin{align}
u(x)=r^{4/7}\cos(4\,\xi/7)\quad\text{with }x=r(\cos \xi,\sin \xi).
\end{align}
Then, $u$ admits a generic singularity at the reentrant corner. The exact solution $\pphi^\exact$ of~\eqref{eq:slp} is just the normal derivative of the solution $u$.

We expect a convergence order of $\OO(N^{-4/7})$ for uniform mesh-refinement, and the optimal rate $\OO(N^{-3/2})$ for the adaptive strategy, which is seen in Figure~\ref{fig:z_shape_conv} for different values of $\theta$ and $\lambda$.
A naive initial guess in Step~(vi) of Algorithm~\ref{algorithm} (i.e., if $\pphi_{(j+1)0}:=0$) leads to a logarithmical growth of the number of PCG iterations, whereas for nested iteration $\pphi_{(j+1)0}:=\pphi_{j\k}$ the number of PCG iterations stays uniformly bounded, cf. Figure~\ref{fig:z_shape_not_nested}.
Figure~\ref{fig:z_shape_cond_number} shows the condition numbers for an artificial refinement towards the reentrant corner  as well as the condition numbers for Algorithm~\ref{algorithm} with $\lambda = 10^{-3}$ and $\theta=0.5$.

\begin{figure}
	\centering
	\begin{minipage}{.3\textwidth}
	\centering
	\raisebox{-0.5\height}{\includegraphics[width=\textwidth]{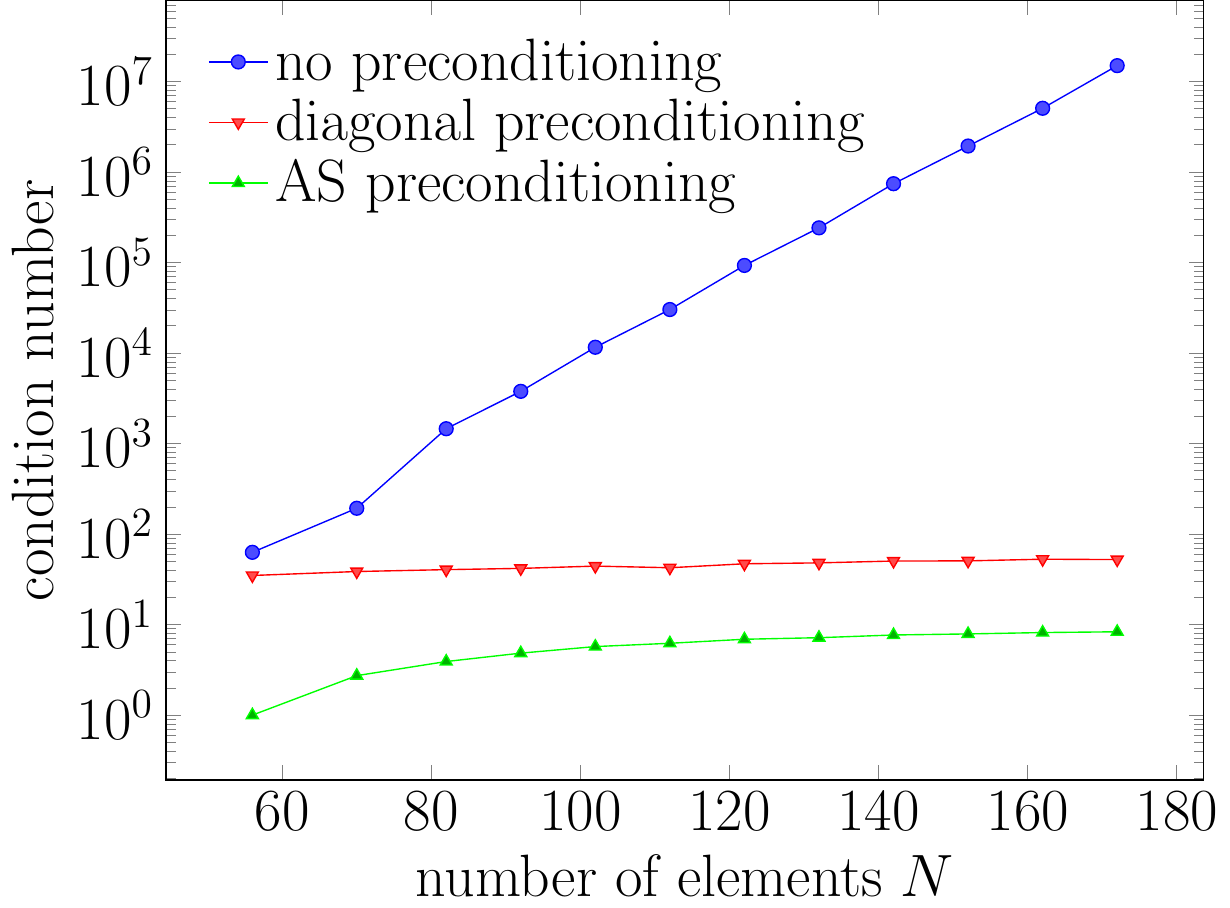}}\\%
	\raisebox{-0.5\height}{\adjincludegraphics[width=.9\textwidth,Clip={.15\width} {.10\height} {0.15\width} {.1\height}]{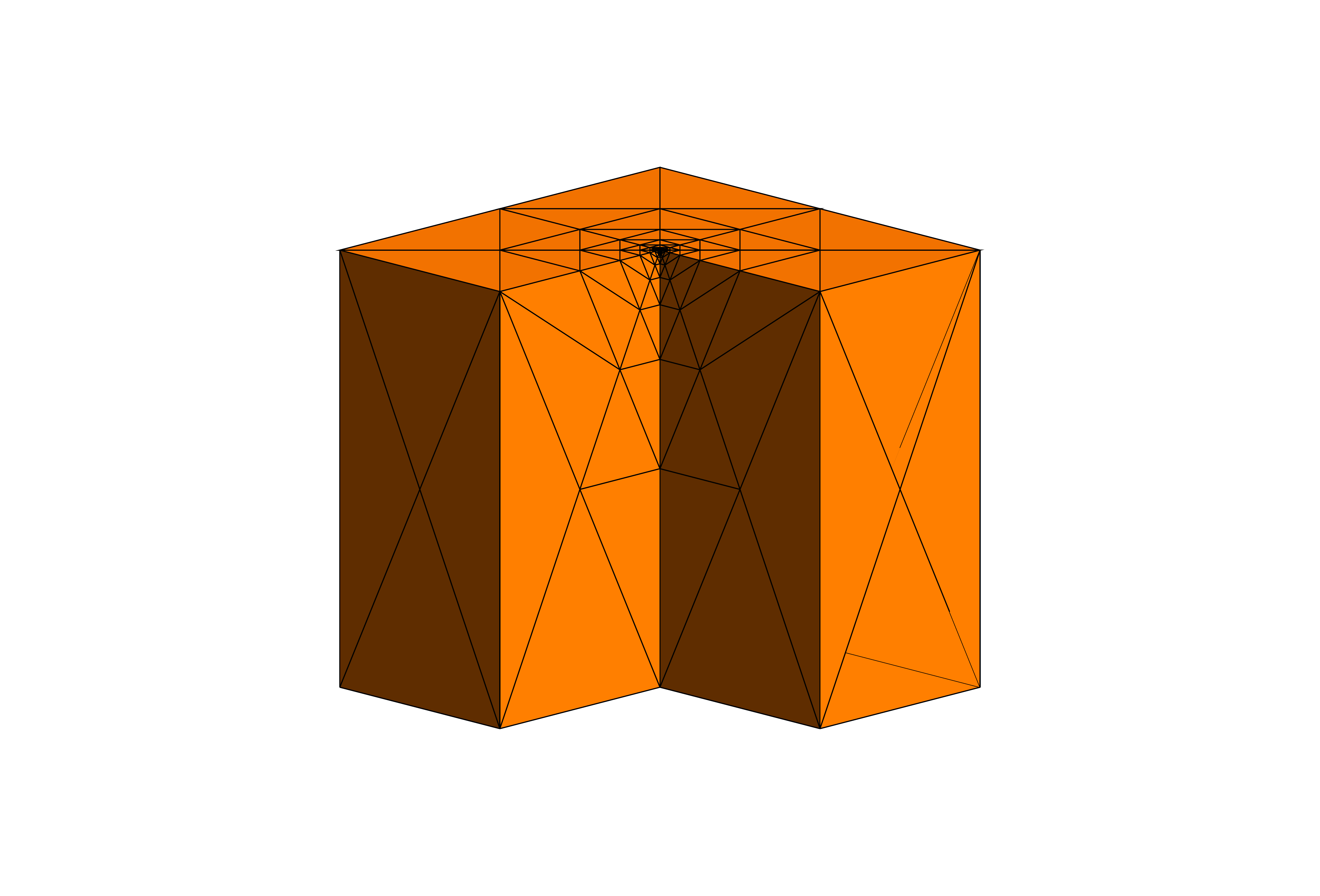}}\hspace*{-5mm}%
	\end{minipage}%
	\hfill%
	\begin{minipage}{.3\textwidth}
	\centering
	\raisebox{-0.5\height}{\includegraphics[width=\textwidth]{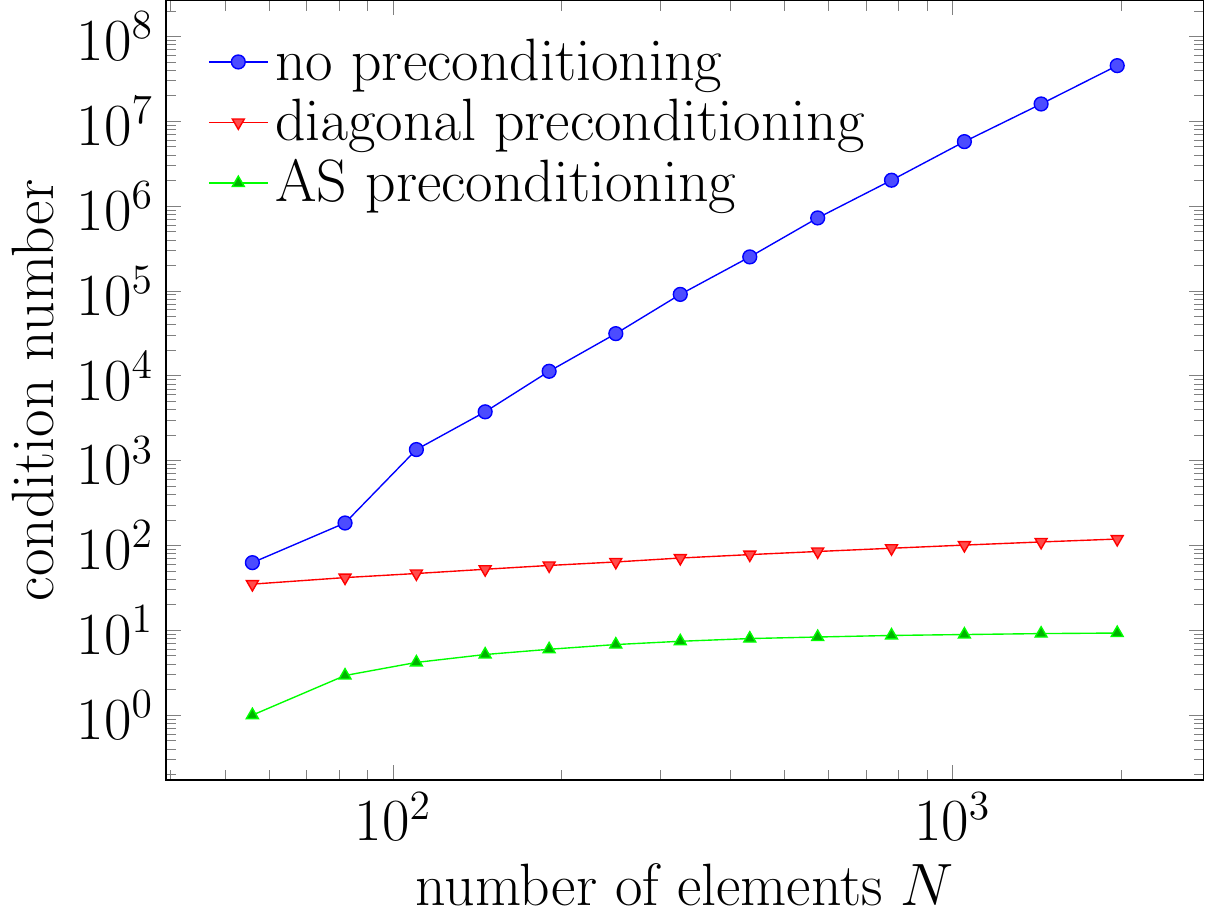}}\\%
	\raisebox{-0.5\height}{\adjincludegraphics[width=.9\textwidth,Clip={.15\width} {.10\height} {0.15\width} {.1\height}]{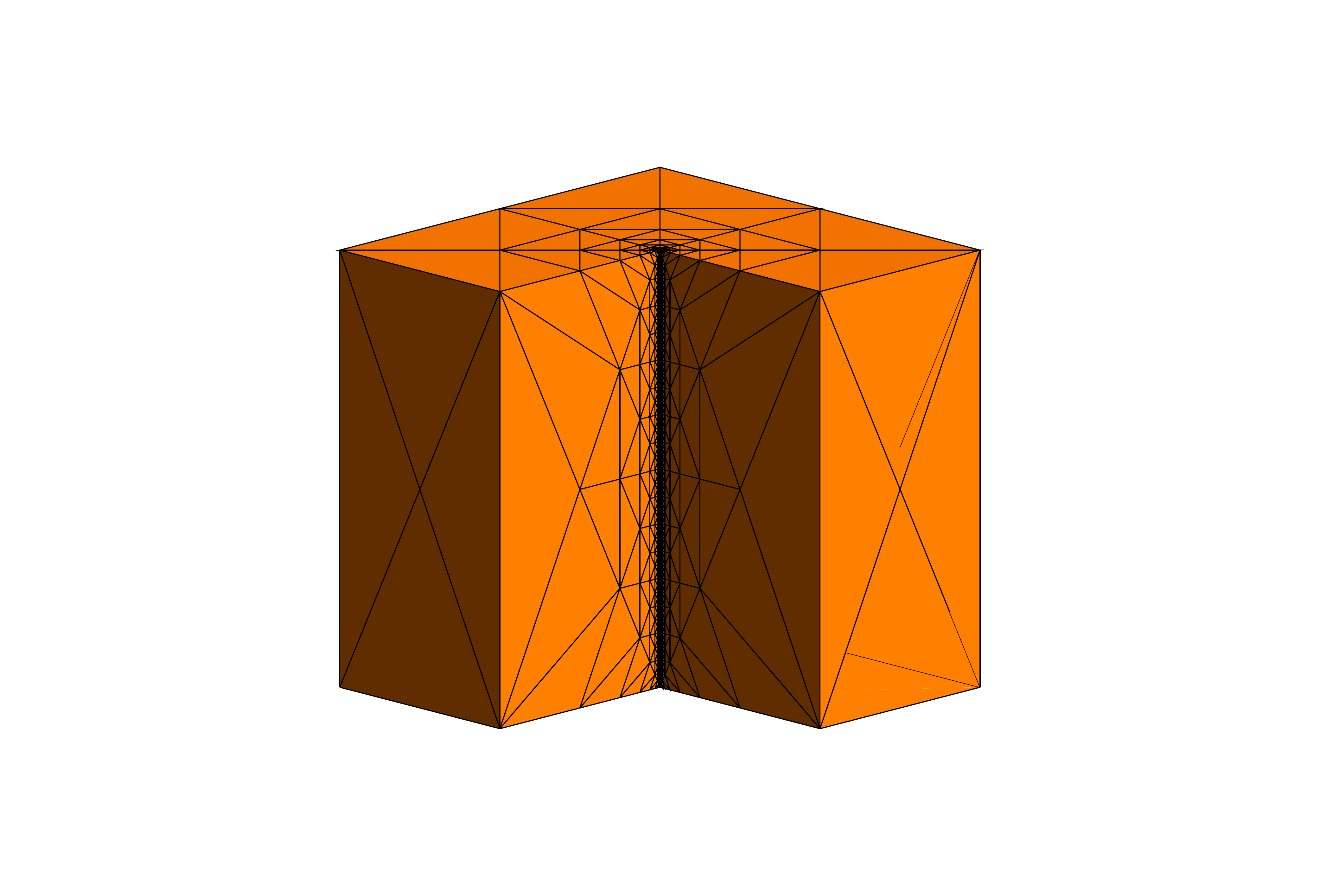}}\hspace*{-5mm}
	\end{minipage}%
	\hfill%
	\begin{minipage}{.3\textwidth}
	\centering
	\raisebox{-0.5\height}{\includegraphics[width=\textwidth]{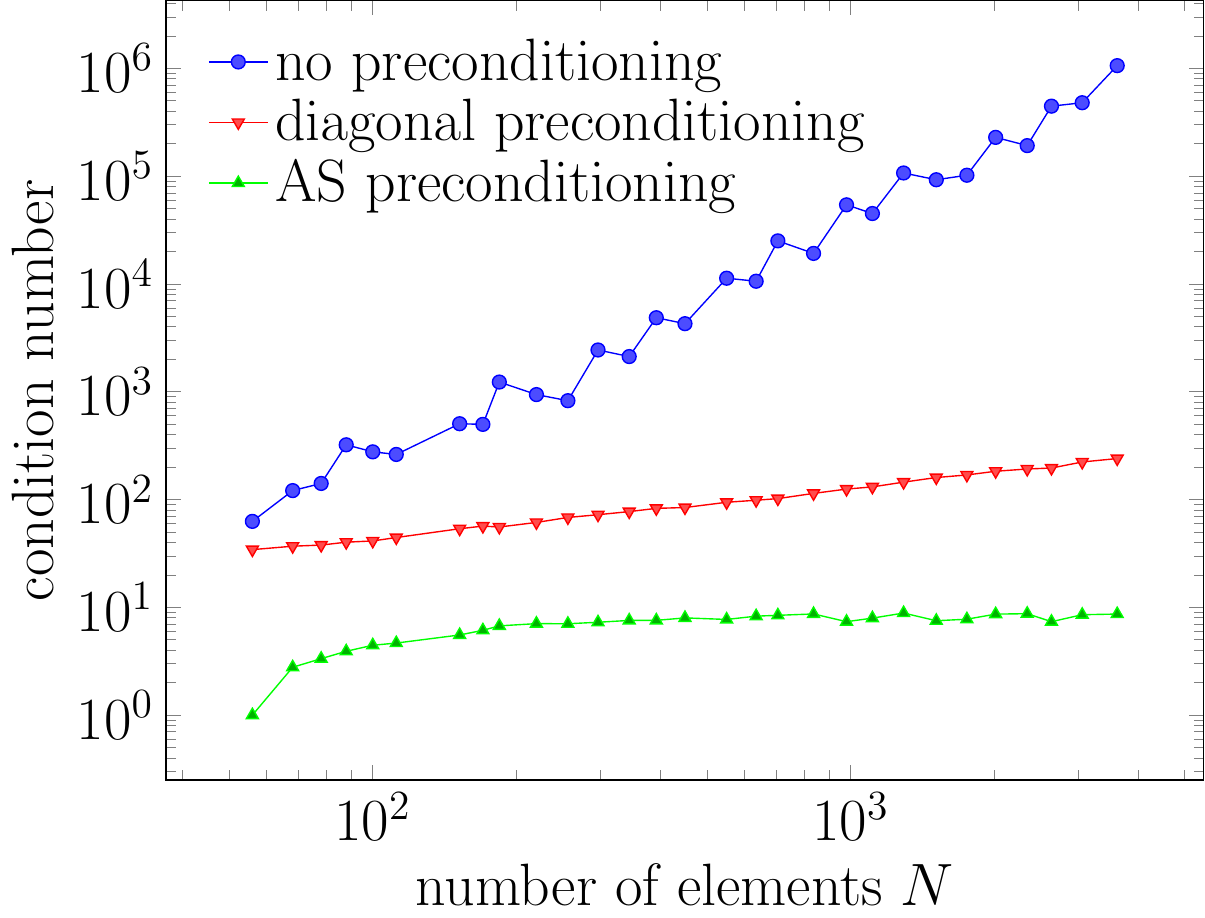}}\\%
	\raisebox{-0.5\height}{\adjincludegraphics[width=.9\textwidth,Clip={.15\width} {.10\height} {0.15\width} {.1\height}]{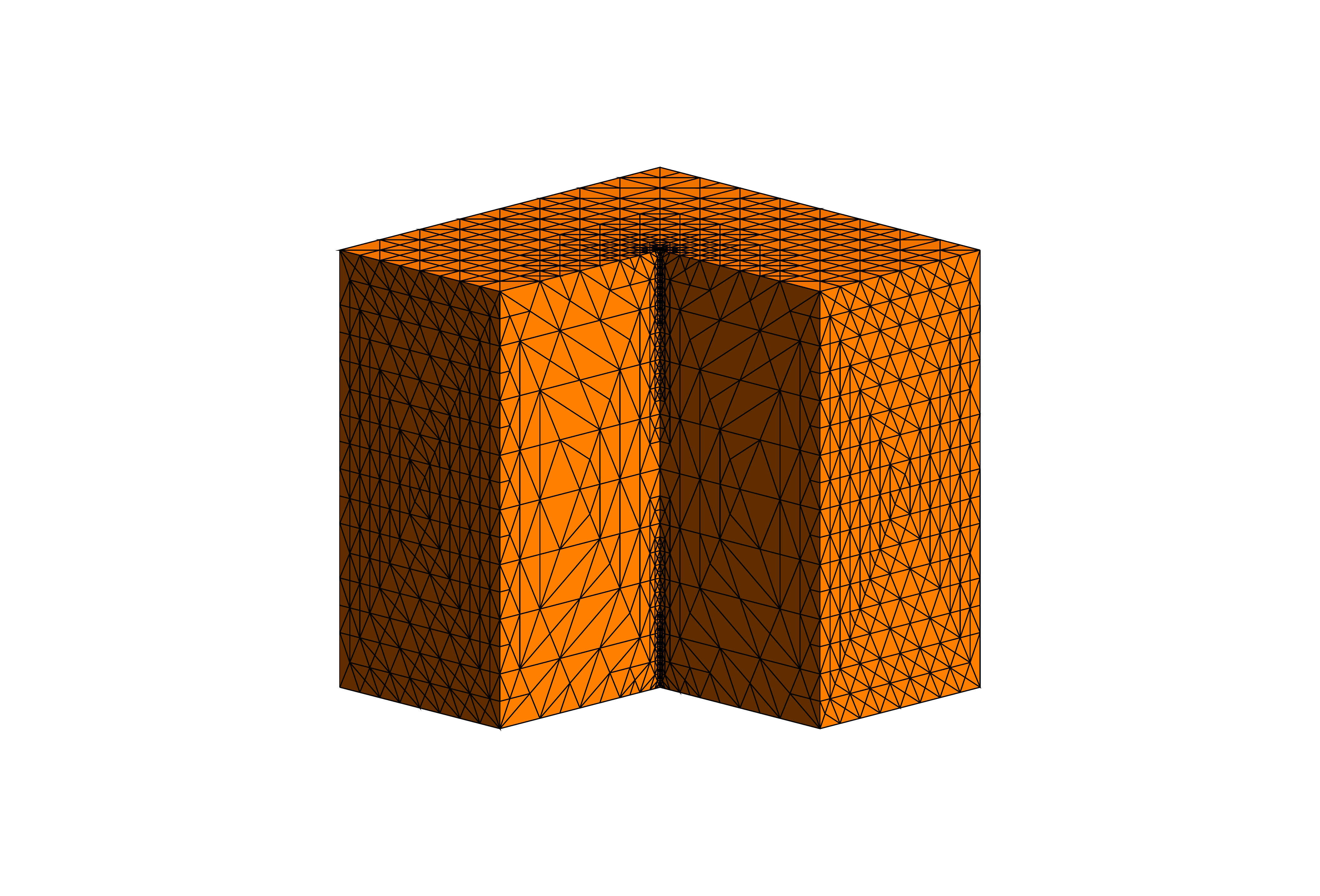}}\hspace*{-5mm}
	\end{minipage}%
	\caption{Example~\ref{subsection:lshape}: Condition numbers of the preconditioned and non-preconditioned Galerkin matrix for an artificial refinement towards one reentrant corner (left) or edge (middle), and for Algorithm~\ref{algorithm} (right).}
\label{fig:lshape_cond_number}
\end{figure}

\begin{figure}
	\centering
	\raisebox{-0.5\height}{\includegraphics[width=0.48\textwidth]{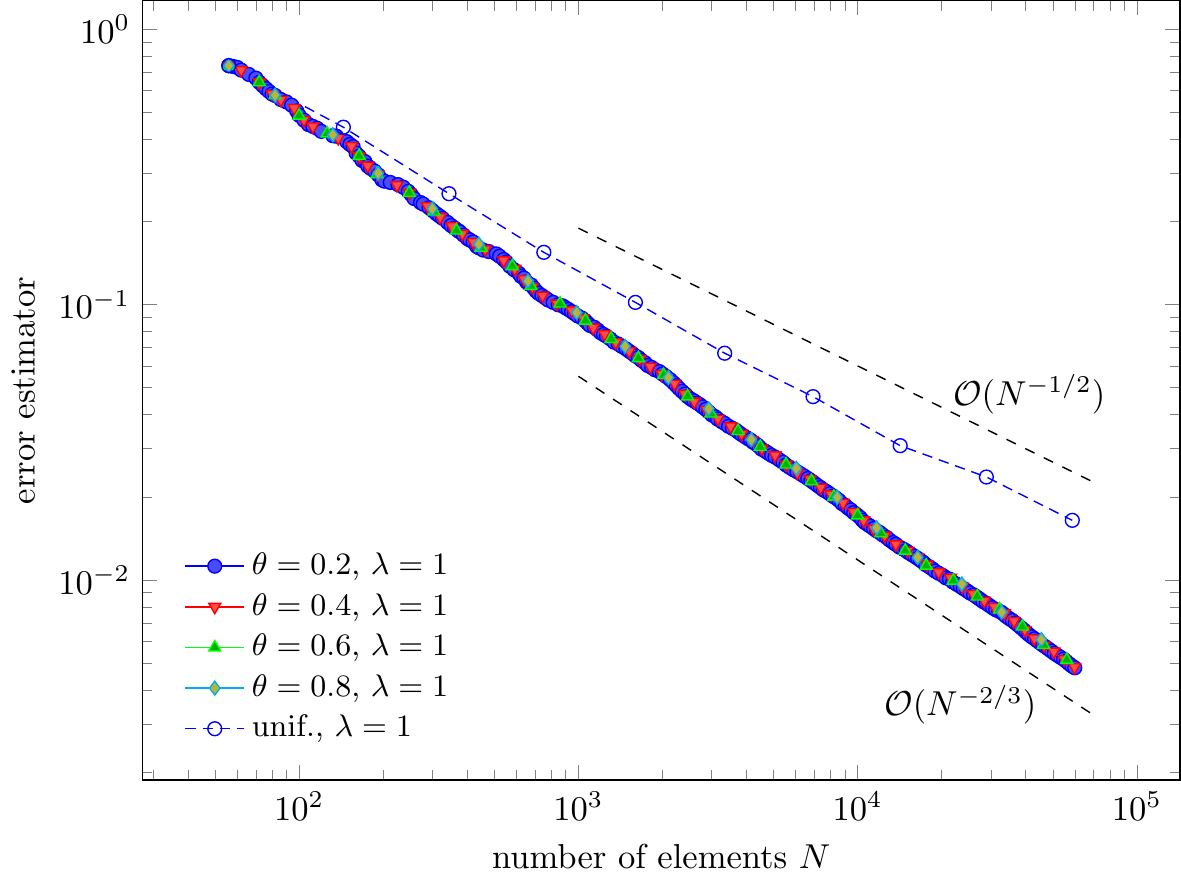}}
	\hfill
	\raisebox{-0.5\height}{\includegraphics[width=0.48\textwidth]{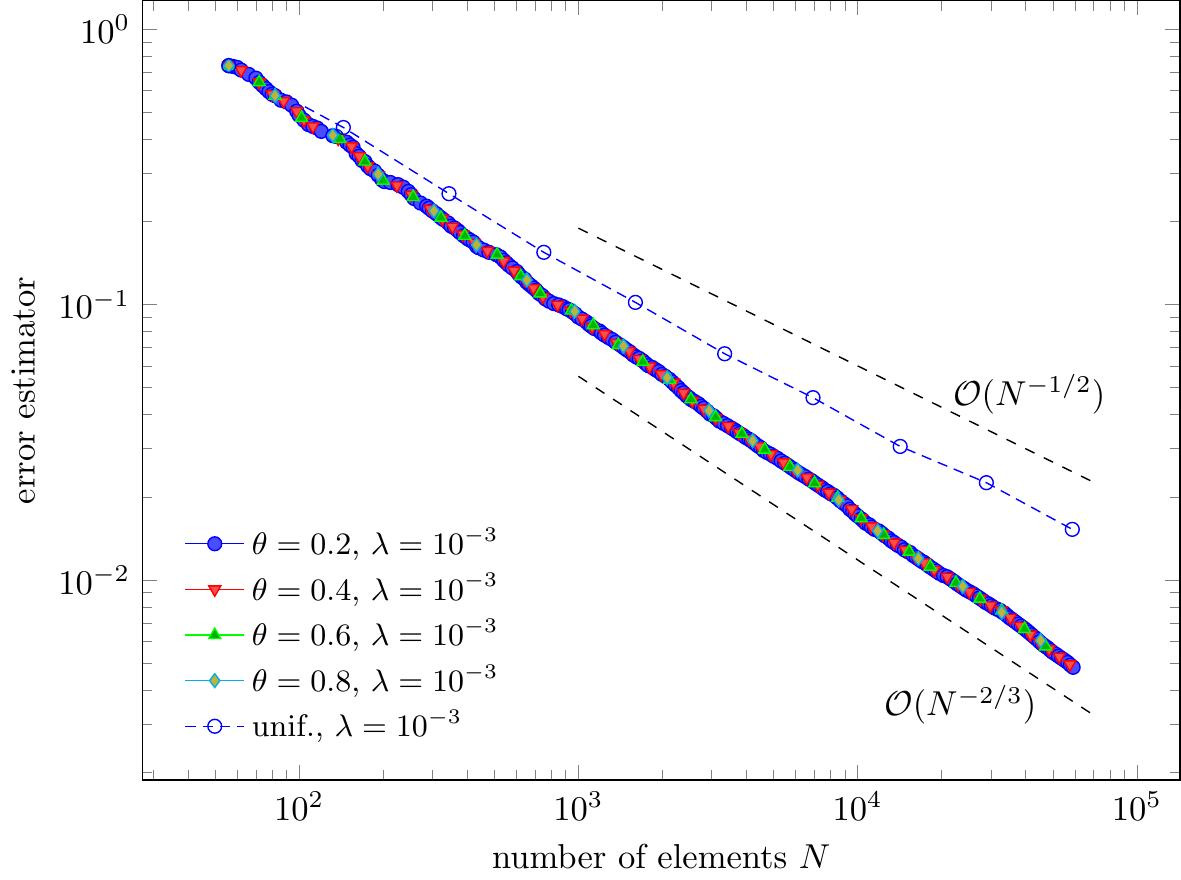}}\vspace{0.5cm}
	\raisebox{-0.5\height}{\includegraphics[width=0.48\textwidth]{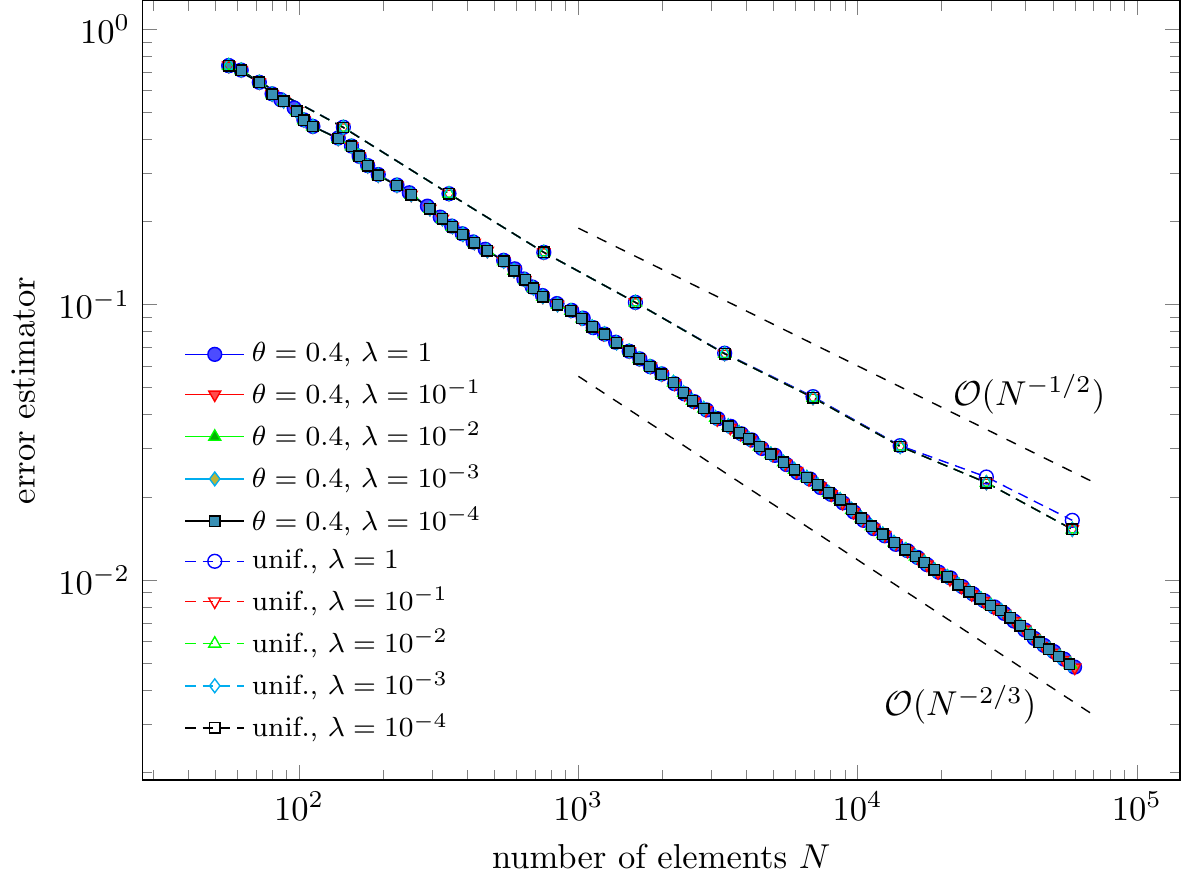}}
	\hfill
	\raisebox{-0.5\height}{\includegraphics[width=0.48\textwidth]{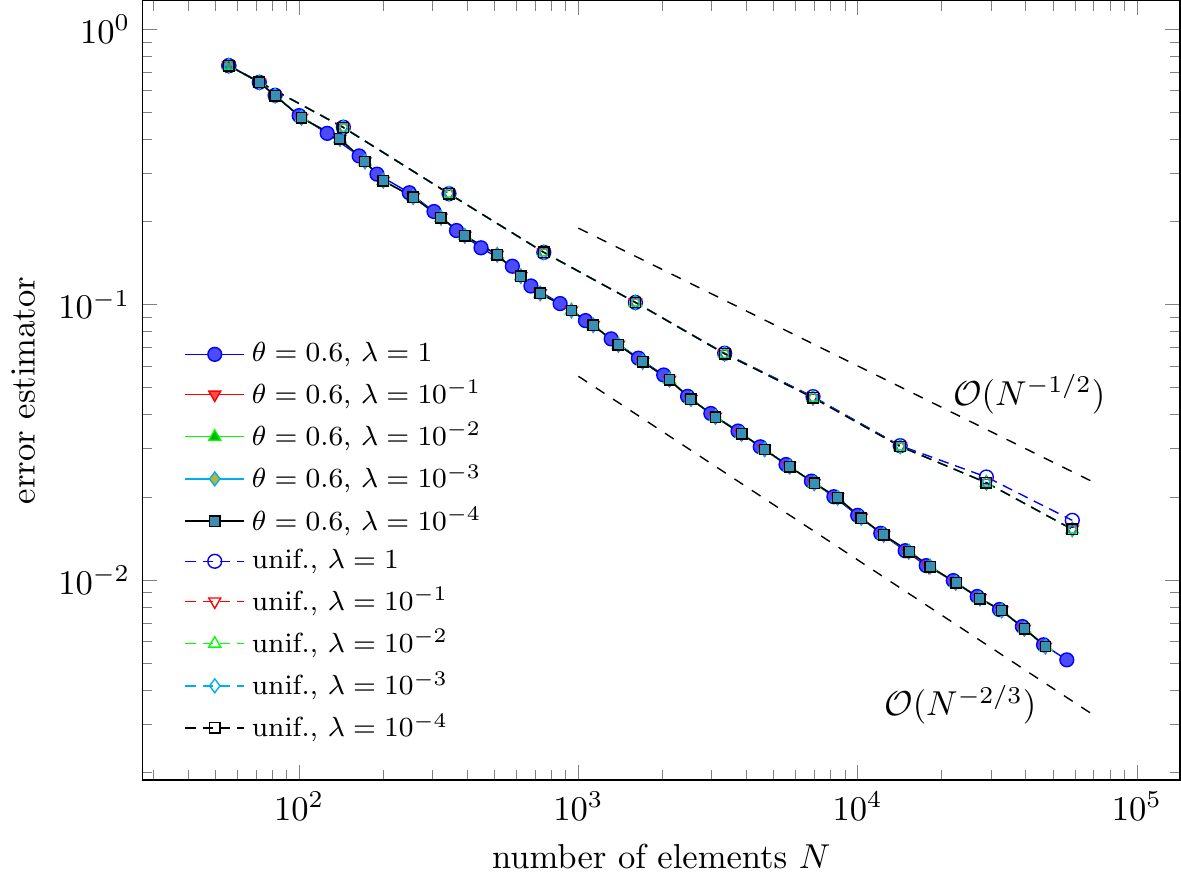}}
	\caption{Example~\ref{subsection:lshape}: Estimator convergence for fixed values of $\lambda$ (left: $\lambda=1$, right: $\lambda=10^{-3}$) and $\theta\in\{0.2,0.4,0.6,0.8\}$ (top) and for fixed values of $\theta$ (left: $\theta=0.4$, right: $\theta=0.6$) and $\lambda\in\{1,10^{-1},\ldots,10^{-4}\}$ (bottom).}
\label{fig:l_shape_conv}
\end{figure}
%
\subsection{L-shaped domain in 3D}
\label{subsection:lshape}
\noindent 
Let $\Gamma:=\partial\Omega$ be the boundary of the L-shaped domain 
$\Omega = (-1,1)^3 \backslash ([-1,0] \times [0,1] \times[-1,1])$, cf.\ Figure~\ref{fig:lshape_cond_number}.
The right-hand side is given by $f=(K+1/2)g$ with the double-layer operator $K\colon H^{1/2}(\Gamma)\to H^{1/2}(\Gamma)$. Again, the weakly-singular integral equation~\eqref{eq:slp} is then equivalent to the Dirichlet problem~\eqref{eq:dbp}.
We prescribe the exact solution in 3D cylindrical coordinates as
\begin{align}
u(x)=z\,r^{2/3}\cos(2/3 \, (\xi-\pi/4) )\quad\text{with }x=(r\cos \xi,r\sin \xi,z).
\end{align}
Note that $u$ admits a singularity along the reentrant edge. The exact solution $\pphi^\exact$ of \eqref{eq:slp} is just the normal derivative of the exact solution $u$.

In Figure~\ref{fig:l_shape_conv}, we compare Algorithm~\ref{algorithm} with different values for $\theta$ and $\lambda$ to uniform mesh-refinement.
Uniform mesh-refinement leads only to a reduced rate of $\OO(N^{-1/2})$, while adaptivity, independently of $\theta$ and $\lambda$, leads to the improved rate of approximately $\OO(N^{-2/3})$. While one would expect $\OO(N^{-3/4})$ for smooth exact solutions $\phi^\exact$, this would require anisotropic elements along the reentrant edge for the present solution $\phi^\exact=\partial_n u$. Since NVB guarantees uniform $\gamma$-shape regularity of the meshes, the latter is not possible and hence leads to a reduced optimal rate.
Finally, Figure~\ref{fig:lshape_cond_number} shows the condition numbers for (diagonal or additive Schwarz) preconditioning and no preconditioning for artificial refinements towards one reentrant corner or the reentrant edge as well as the condition numbers of the matrices arising from Algorithm~\ref{algorithm} with $\lambda = 10^{-3}$ and $\theta=0.5$.

\begin{figure}
	\centering
	\begin{minipage}{.3\textwidth}
	\centering
	\raisebox{-0.5\height}{\includegraphics[width=\textwidth]{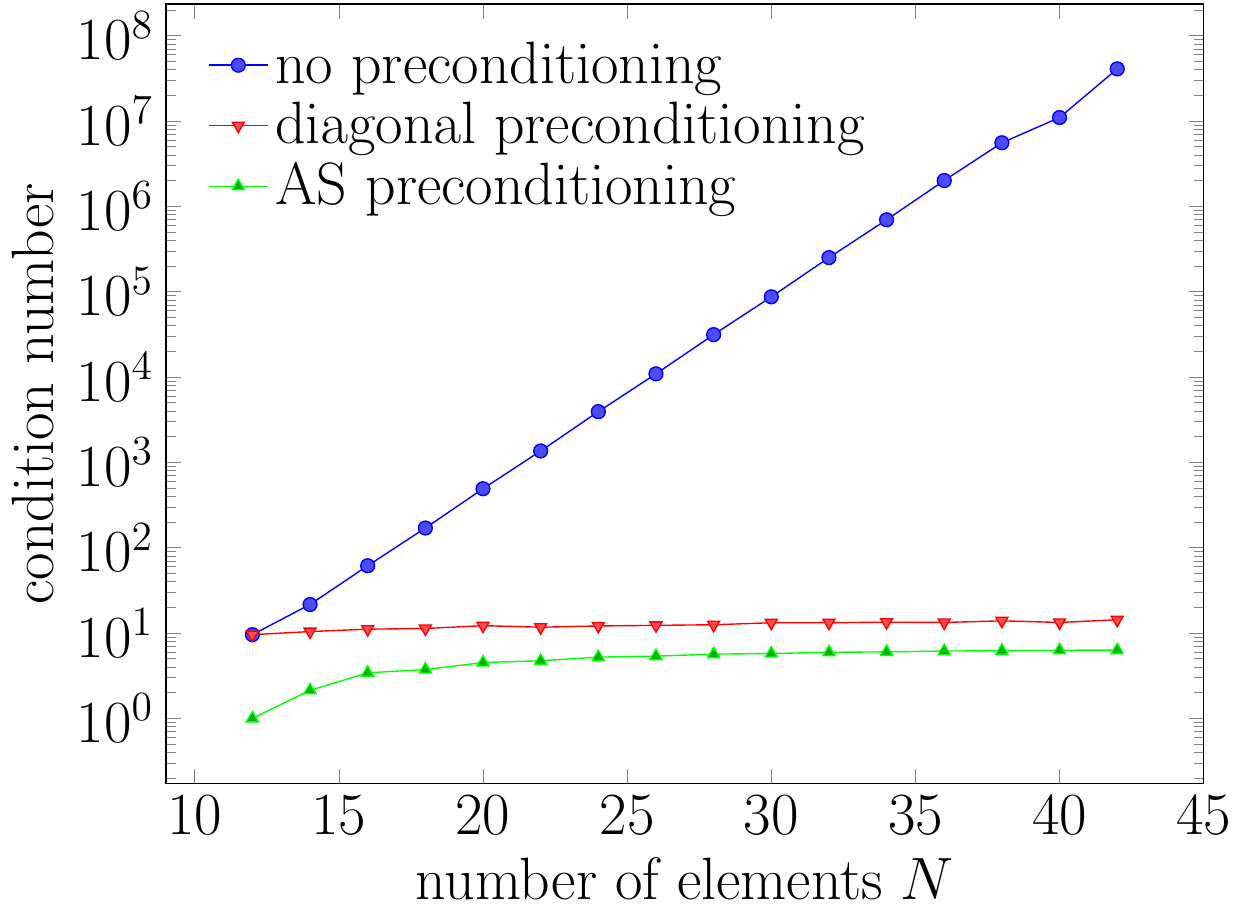}}\\%
	\raisebox{-0.5\height}{\adjincludegraphics[width=.9\textwidth,Clip={.25\width} {.0\height} {0.25\width} {.0\height}]{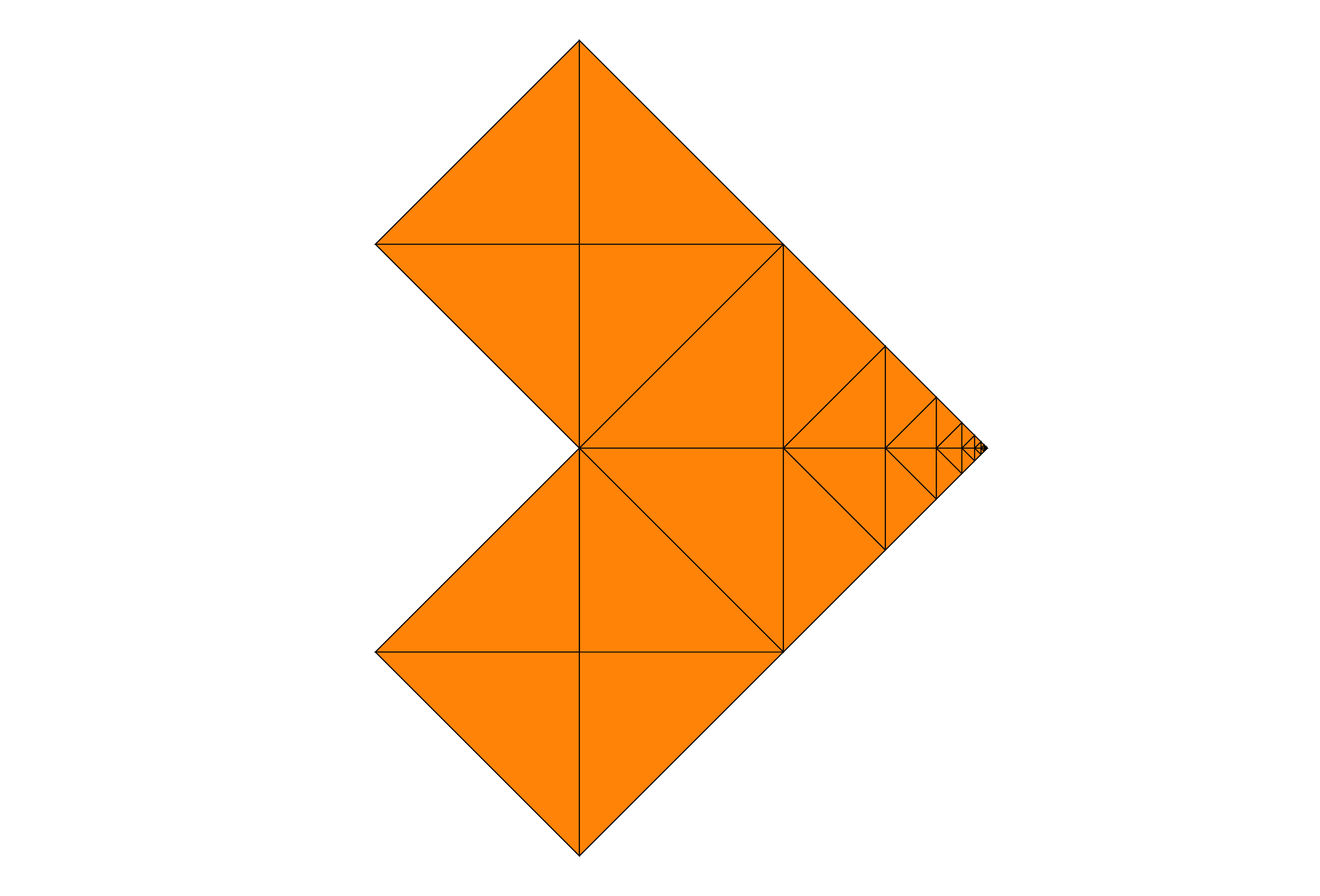}}\hspace*{-5mm}%
	\end{minipage}%
	\hfill%
	\begin{minipage}{.3\textwidth}
	\centering
	\raisebox{-0.5\height}{\includegraphics[width=\textwidth]{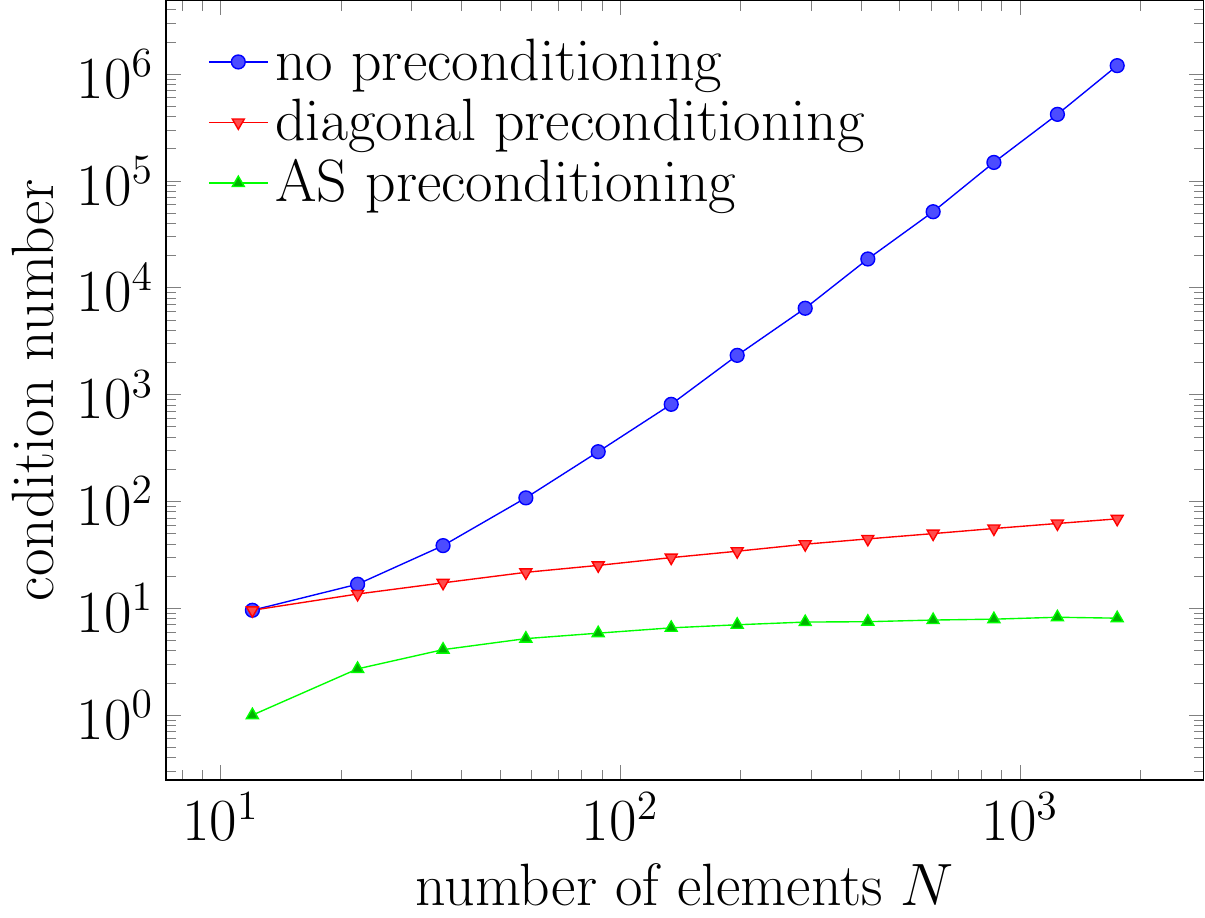}}\\%
	\raisebox{-0.5\height}{\adjincludegraphics[width=.9\textwidth,Clip={.25\width} {.0\height} {0.25\width} {.0\height}]{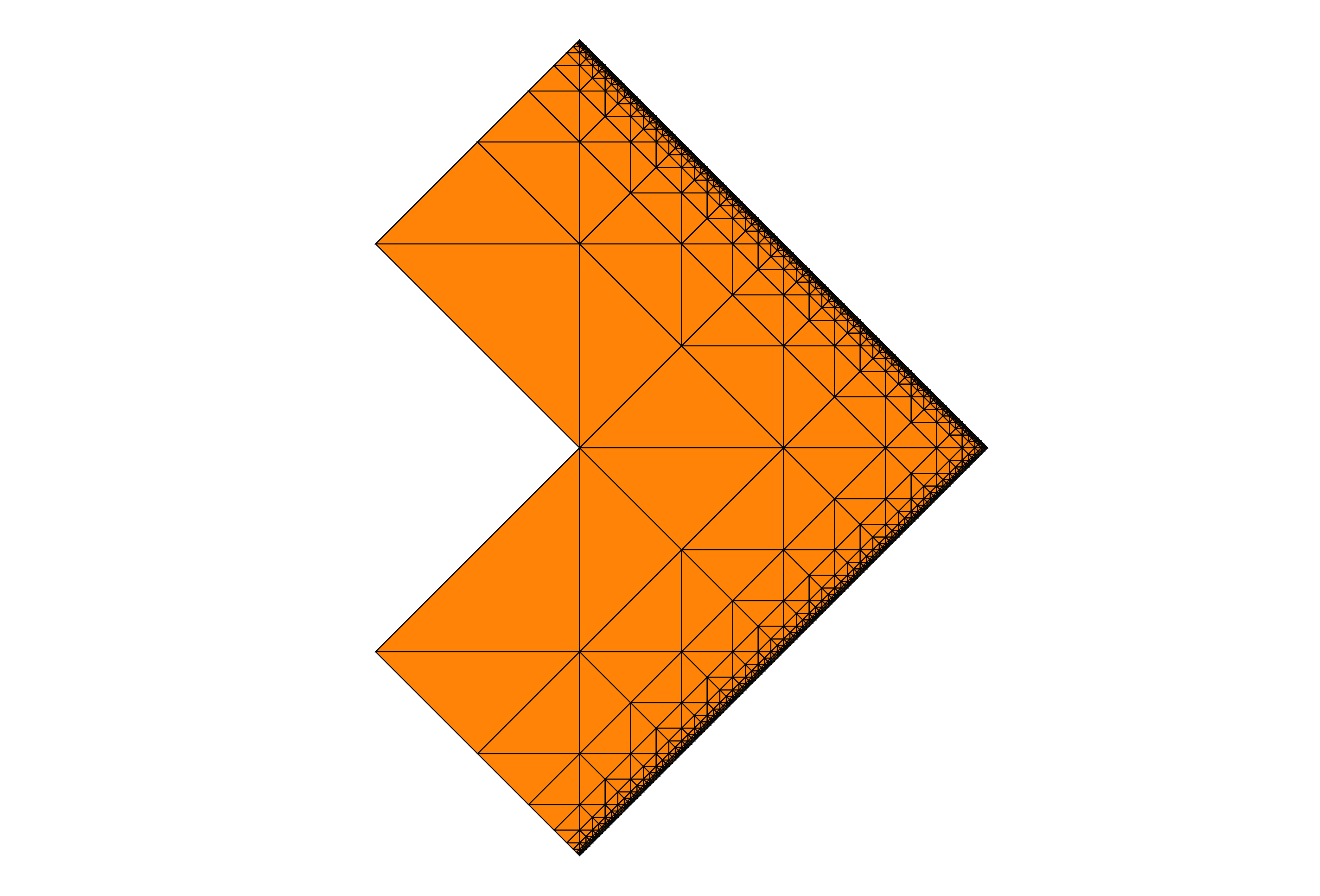}}\hspace*{-5mm}
	\end{minipage}%
	\hfill%
	\begin{minipage}{.3\textwidth}
	\centering
	\raisebox{-0.5\height}{\includegraphics[width=\textwidth]{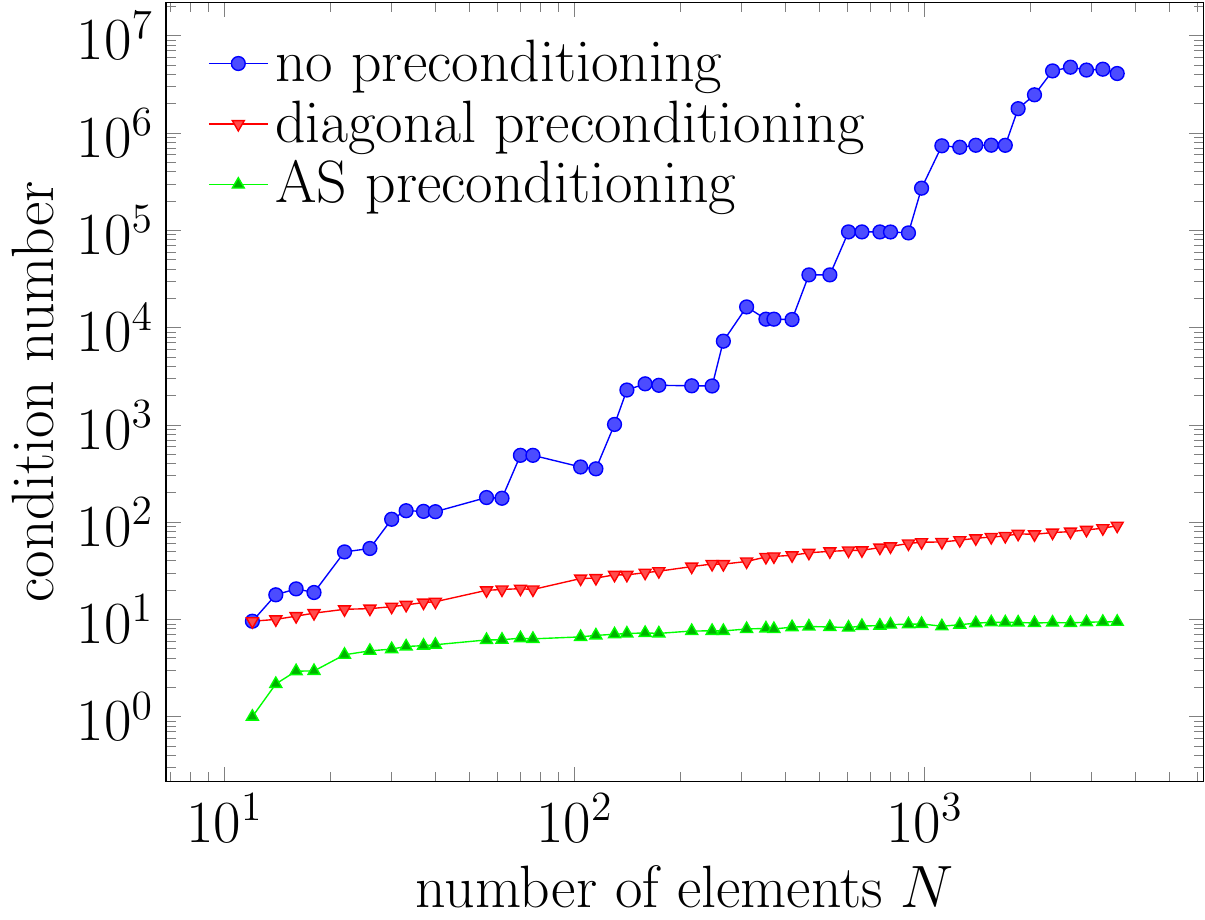}}\\%
	\raisebox{-0.5\height}{\adjincludegraphics[width=.9\textwidth,Clip={.25\width} {.0\height} {0.25\width} {.0\height}]{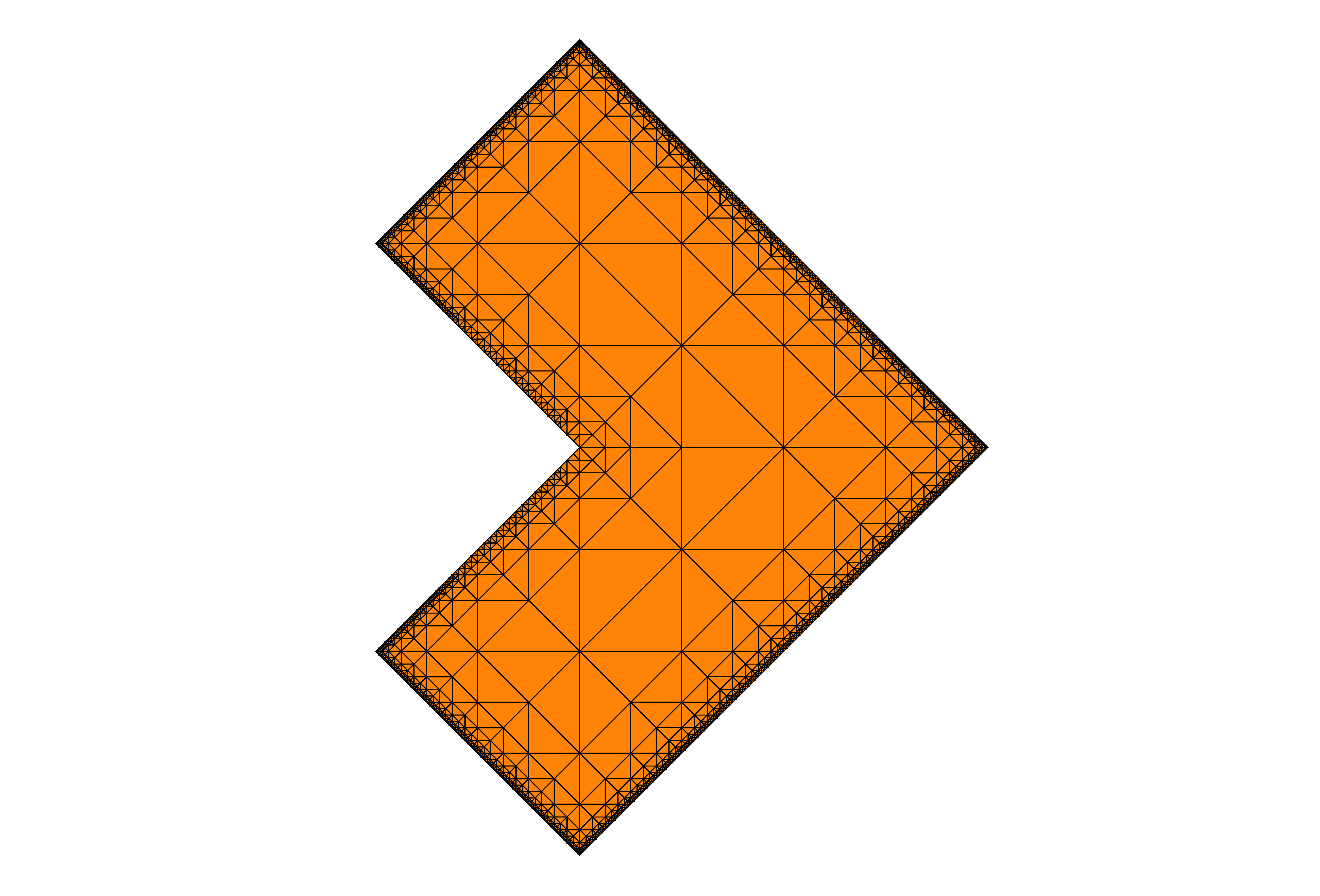}}\hspace*{-5mm}
	\end{minipage}%
	\caption{Example~\ref{subsection:screen}: Condition numbers of the preconditioned and non-preconditioned Galerkin matrix for an artificial refinement towards the right corner (left), the right edges (middle), and for Algorithm~\ref{algorithm} (right).}
\label{fig:screen_cond_number}
\end{figure}

\begin{figure}
	\centering
	\raisebox{-0.5\height}{\includegraphics[width=0.48\textwidth]{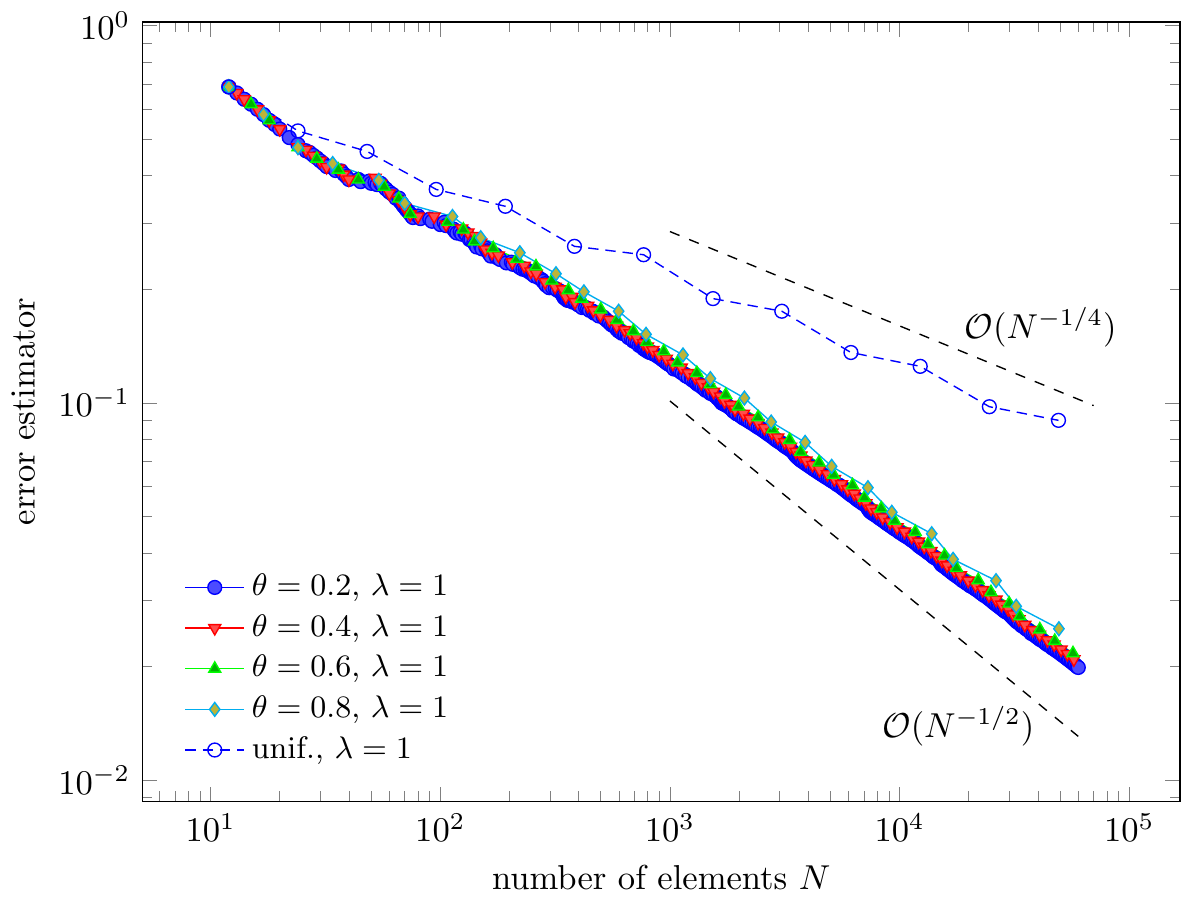}}
	\hfill
	\raisebox{-0.5\height}{\includegraphics[width=0.48\textwidth]{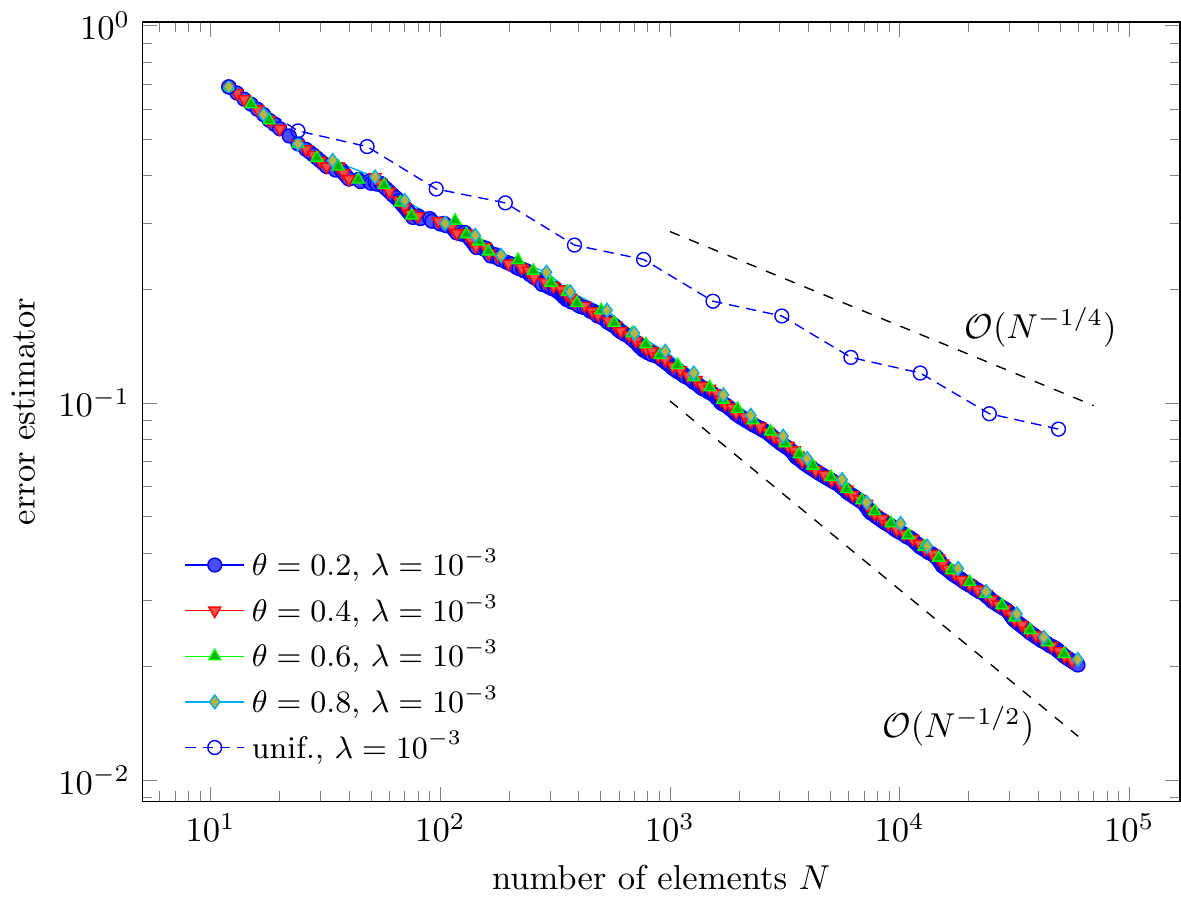}}\vspace{0.5cm}
	\raisebox{-0.5\height}{\includegraphics[width=0.48\textwidth]{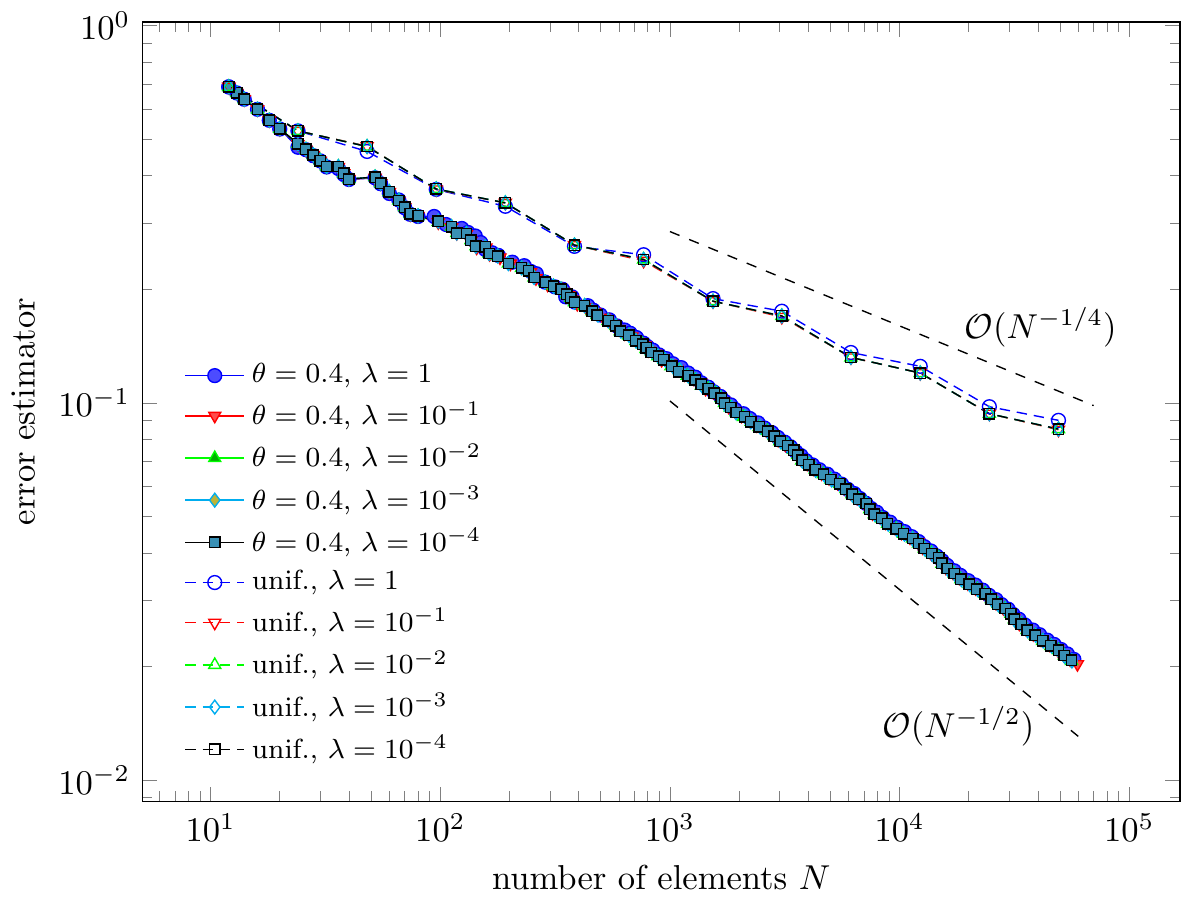}}
	\hfill
	\raisebox{-0.5\height}{\includegraphics[width=0.48\textwidth]{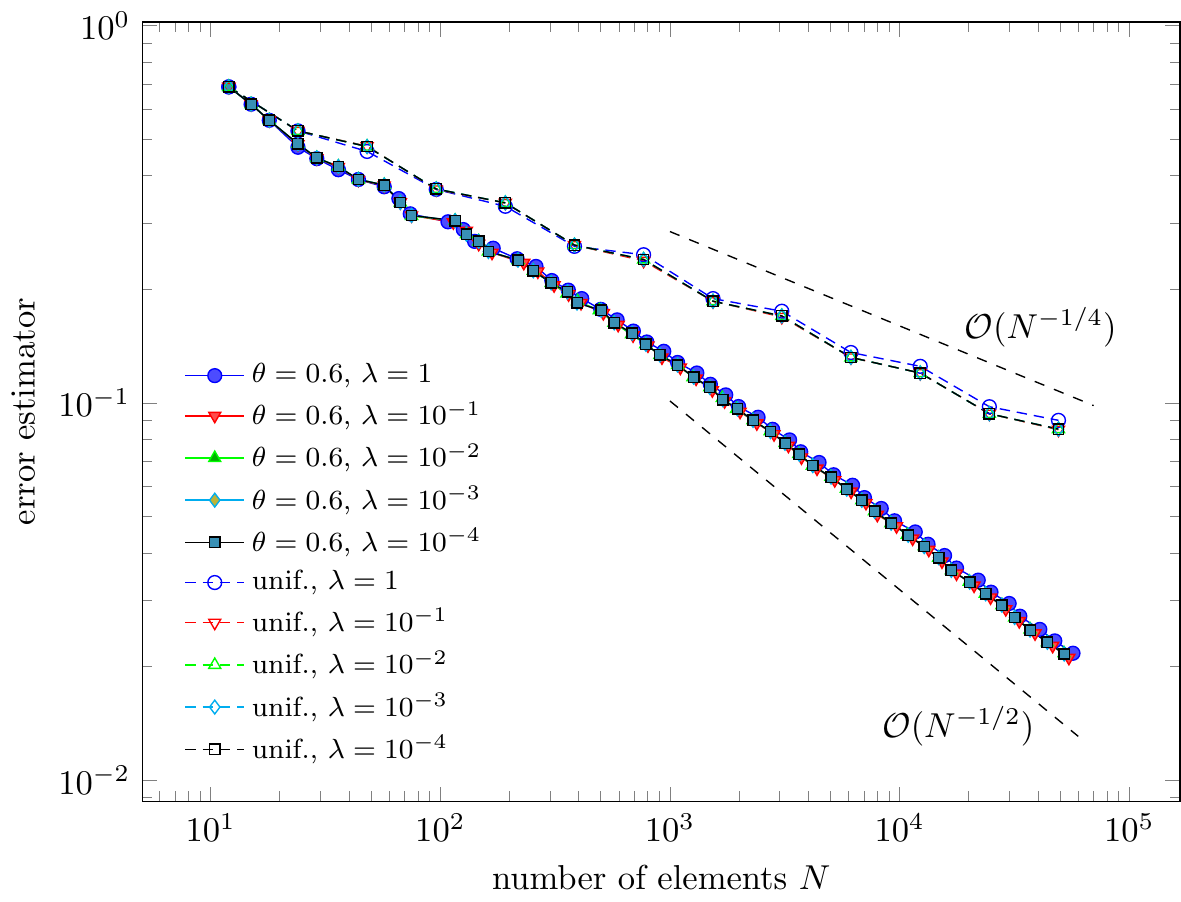}}
	\caption{Example~\ref{subsection:screen}: Estimator convergence for fixed values of $\lambda$ (left: $\lambda=1$, right: $\lambda=10^{-3}$) and $\theta\in\{0.2,0.4,0.6,0.8\}$ (top) and for fixed values of $\theta$ (left: $\theta=0.4$, right: $\theta=0.6$) and $\lambda\in\{1,10^{-1},\ldots,10^{-4}\}$ (bottom).}
\label{fig:screen_conv}
\end{figure}
%
\subsection{Screen problem in 3D}
\label{subsection:screen}
Let $\Gamma:=((-1,1)^2 \setminus [0,1]) \times\{0\}$, rotated by $3\pi/4$, cf. Figure~\ref{fig:screen_cond_number}.
We consider the weakly-singular integral equation $V \phi^\star=1$ on $\Gamma$. The exact solution $\pphi^\star \in \H^{-1/2}(\Gamma)$ is unknown. 

For the numerical solution of the Galerkin system, we employ PCG with the additive Schwarz preconditioner from Section~\ref{section:oasp}. We note that Theorem~\ref{thm:precond} does not cover this setting.
In particular, we note that the proposed additive Schwarz preconditioner from Section~\ref{section:oasp} appears to be optimal, while the mathematical optimality proof still remains open for screens, cf. Figure~\ref{fig:screen_cond_number}.

In Figure~\ref{fig:screen_conv}, we compare Algorithm~\ref{algorithm} with different values for $\theta$ and $\lambda$ to uniform mesh-refinement. 
We see that uniform mesh-refinement leads only to a reduced rate of $\OO(N^{-1/4})$, while adaptivity, independently of $\theta$ and $\lambda$, leads to the improved rate of approximately $\OO(N^{-1/2})$.

\begin{figure}
	\centering
	\raisebox{-0.5\height}{\includegraphics[width=0.48\textwidth]{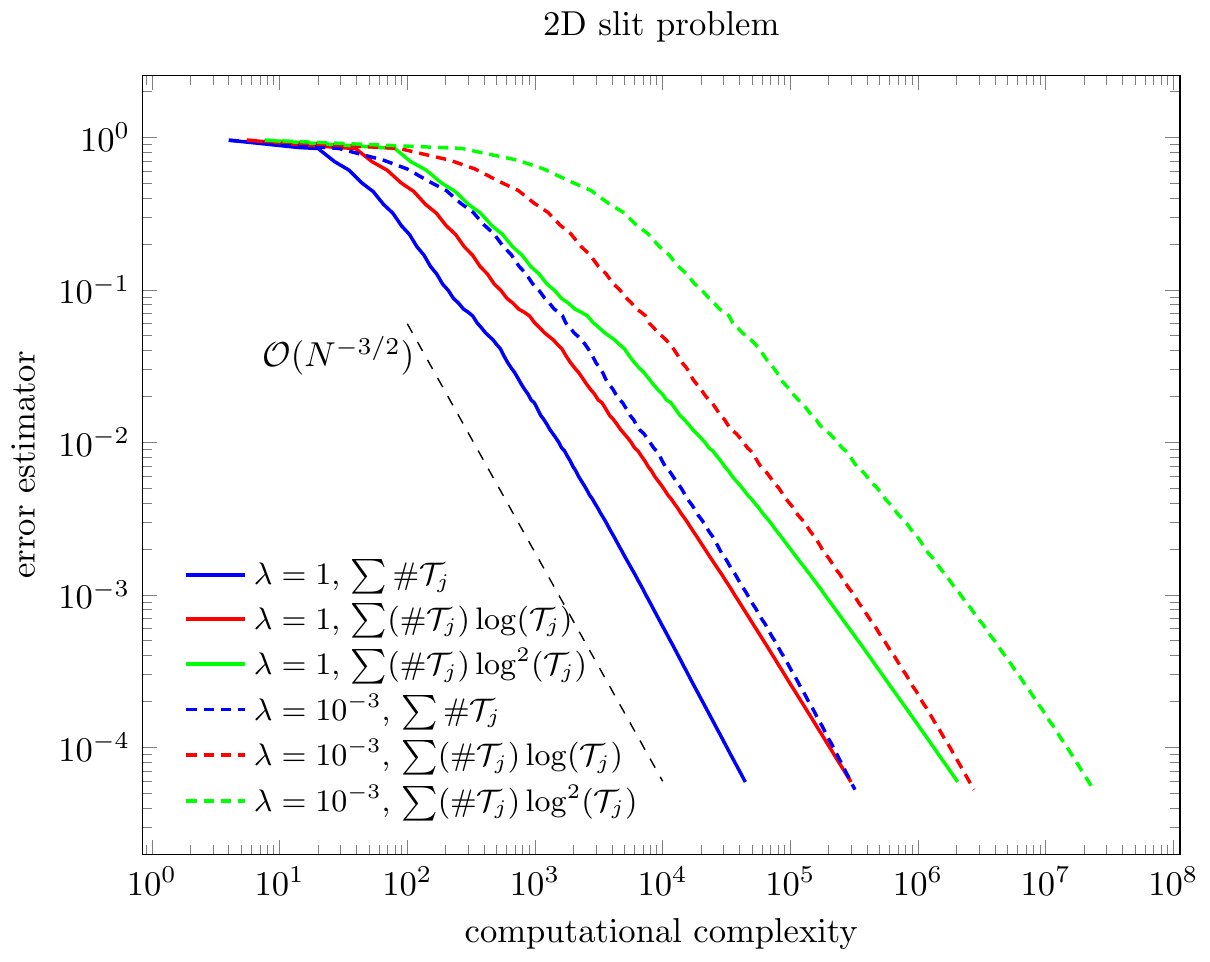}}
	\hfill
	\raisebox{-0.5\height}{\includegraphics[width=0.48\textwidth]{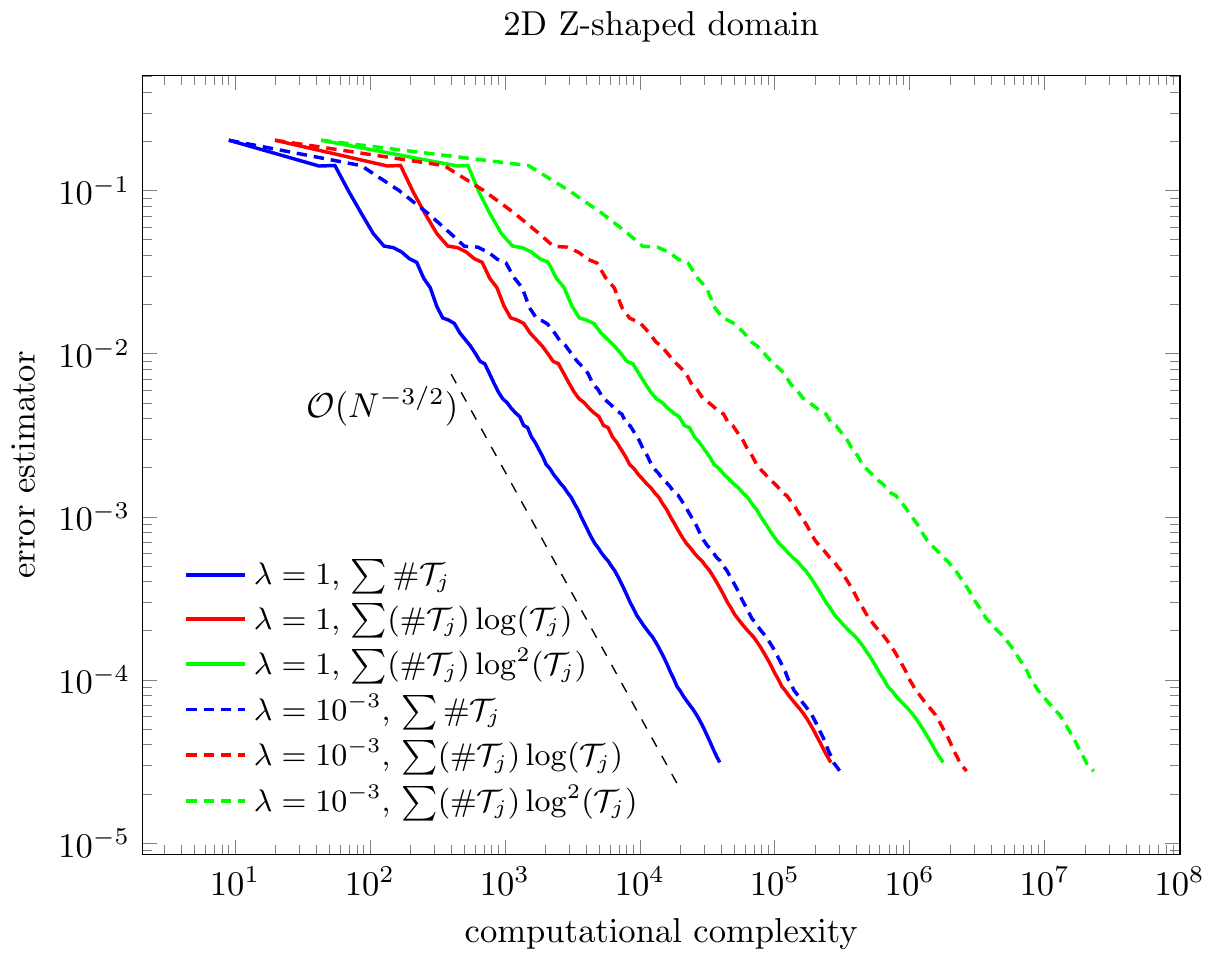}}\vspace{0.5cm}
	\raisebox{-0.5\height}{\includegraphics[width=0.48\textwidth]{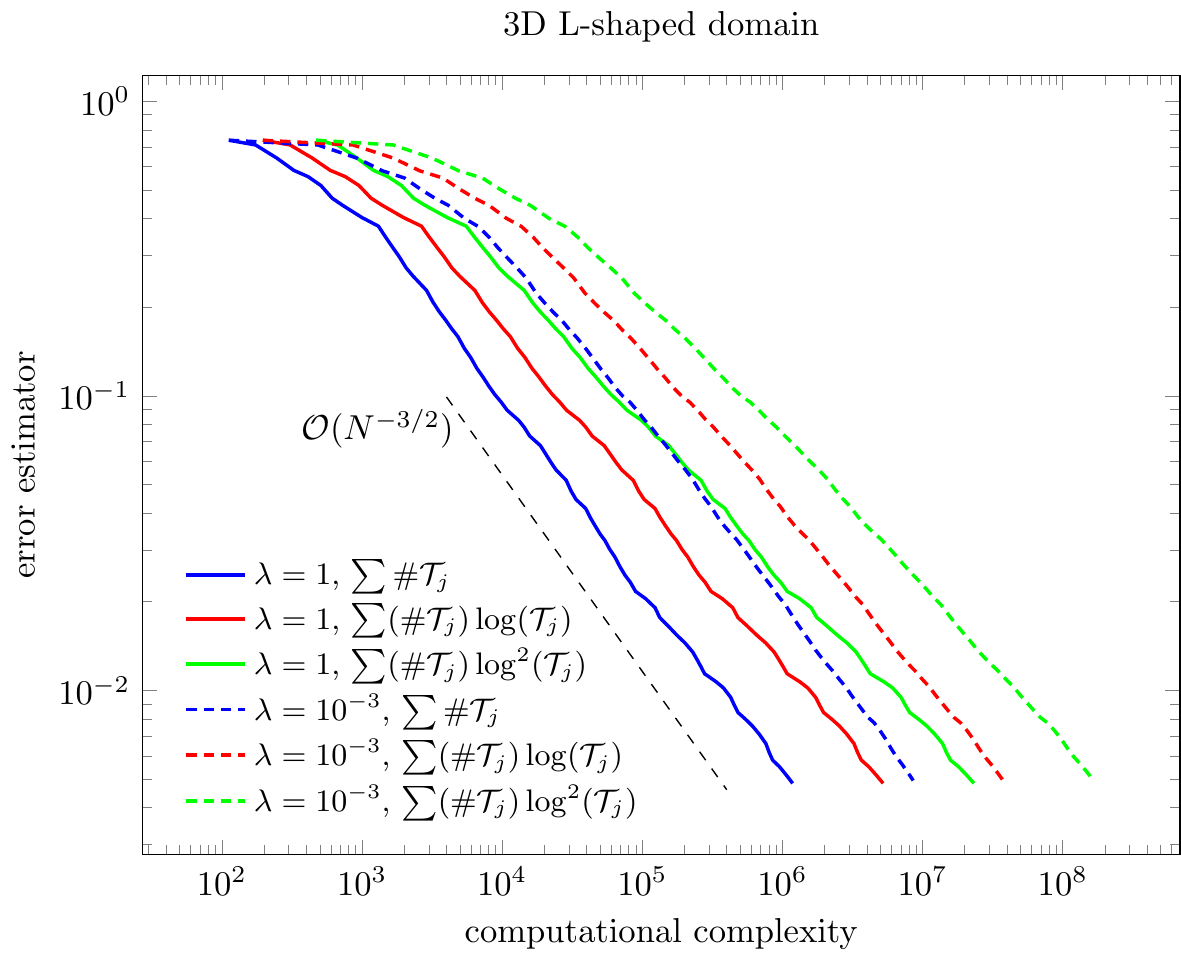}}
	\hfill
	\raisebox{-0.5\height}{\includegraphics[width=0.48\textwidth]{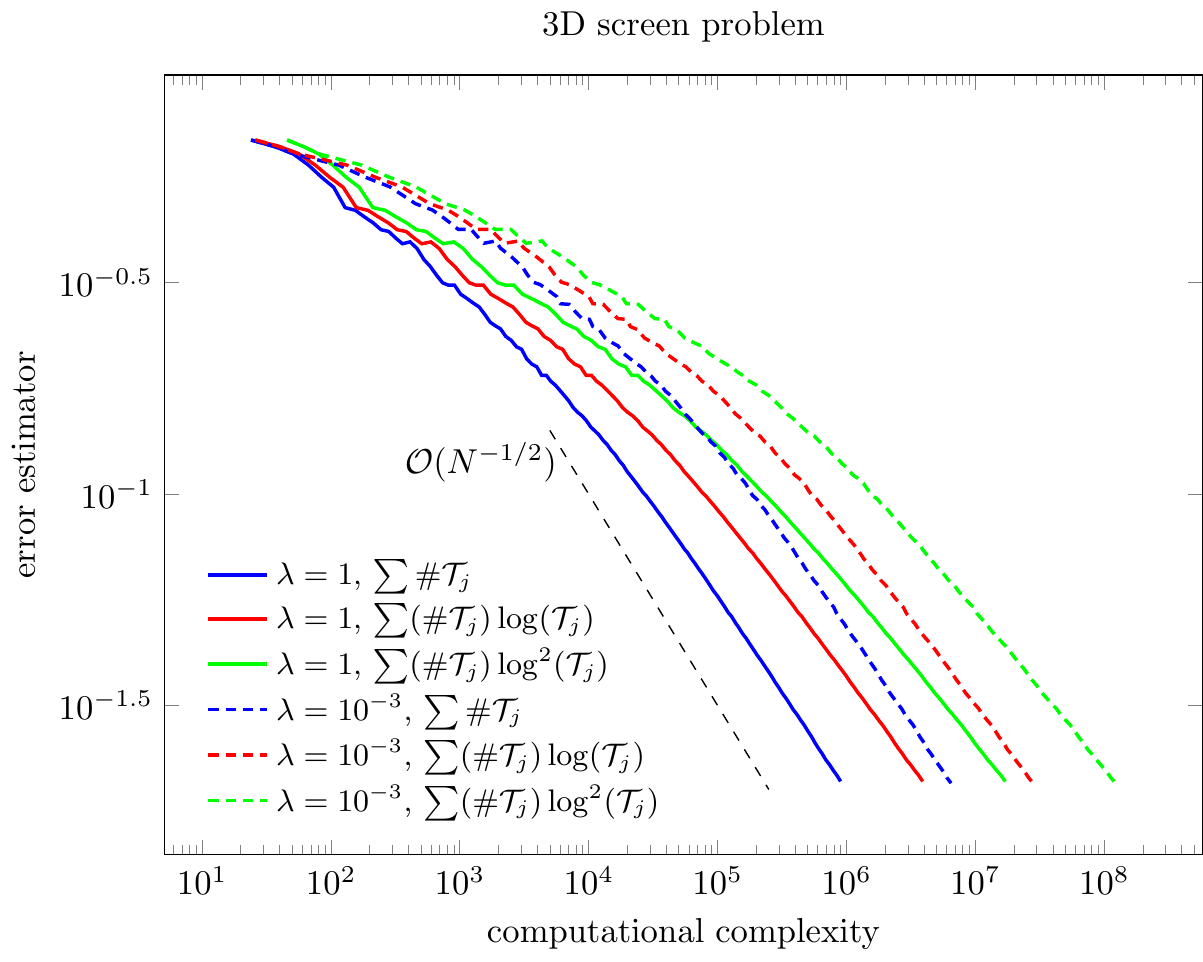}}
	\caption{To underline the quasi-optimal computational complexity of Algorithm~\ref{algorithm}, we plot the error estimator $\eta_j(\phi_{j\k})$ in the different experiments over the cumulative quantities $\sum_{(j,k) \le (j',k')} \#\TT_j$, $\sum_{(j,k) \le (j',k')} (\#\TT_j)\log(\#\TT_j)$ and $\sum_{(j,k) \le (j',k')} (\#\TT_j)\log^2(\#\TT_j)$ for $\theta = 0.4$ and $\lambda \in \{1,10^{-3}\}$.}
\label{fig:complexity}
\end{figure}
\def\cost{\rm cost}%
\subsection{Computational complexity}
With Figure~\ref{fig:complexity}, we aim to underpin the almost optimal computational complexity of Algorithm~\ref{algorithm} (see Corollary~\ref{corollary:algorithm}). To this end, we plot the error estimator $\eta_j(\phi_{jk})$ over the cumulative sums 
\begin{align*}
 \sum_{(j,k) \le (j',k')} \cost(\TT_j)
 \quad \text{with} \quad
 \cost(\TT_j) \in \big\{ \#\TT_j \,,\, (\#\TT_j)\, \log(\#\TT_j) \,,\, (\#\TT_j)\, \log^2(\#\TT_j) \big\}
\end{align*}
for $\theta = 0.4$ and $\lambda \in \{1,10^{-3}\}$. The negative impact of the logarithmic terms on the (preasymptotic) convergence rate is clearly visible.

\section{Proof of Theorem~\ref{thm:precond} (Optimal Multilevel Preconditioner)}\label{proof:precond}
\newcommand\ND{\boldsymbol{\mathcal{N}}\!\boldsymbol{\mathcal{D}}}
\newcommand\RT{\boldsymbol{\mathcal{R}}\!\boldsymbol{\mathcal{T}}}
\newcommand\curl{\boldsymbol{\operatorname{curl}}\,}
\newcommand\divergence{\boldsymbol{\operatorname{div}}\,}
\newcommand\Hcurl[1]{\boldsymbol{H}(\curl;\,#1)}
\newcommand\Hdiv[1]{\boldsymbol{H}(\divergence;\,#1)}
\newcommand\extND{\boldsymbol{E}}
\newcommand\normal{{\boldsymbol{n}}}
\newcommand\ssigma{{\boldsymbol{\sigma}}}
\newcommand\tangential{{\boldsymbol{t}}}
\newcommand\vv{{\boldsymbol{v}}}
\newcommand\uu{{\boldsymbol{u}}}
\newcommand\NN{{\mathcal{N}}}
\newcommand{\SY}{\ensuremath{\mathcal{Y}}}
\newcommand{\PAS}{\ensuremath{\mathcal{S}}}
\newcommand{\Smat}{\ensuremath{\boldsymbol{S}}}
\newcommand{\Amat}{\ensuremath{\boldsymbol{A}}}

For $d = 2$, we refer to~\cite{MR3634453,dissFuehrer} and thus focus only on $d = 3$ and $\Gamma = \partial \Omega$. Due to our additional assumption, $\TT_0 = \widehat\TT_0|_\Gamma$ is the restriction of a conforming simplicial triangulation $\widehat\TT_0$ of $\Omega$ to the boundary $\Gamma$. Moreover, 2D NVB refinement of $\TT_0$ (on the boundary $\Gamma$) is a special case of 3D NVB refinement of $\widehat\TT_0$ (in the volume $\Omega$) plus restriction to the boundary; see, e.g.,~\cite{stevenson}. Hence, each mesh $\TT_\coarse \in \T = \refine(\TT_0)$ is the restriction of a conforming NVB refinement $\widehat \TT_\coarse \in \widehat\T := \refine(\widehat\TT_0)$, i.e., $\TT_\coarse = \widehat\TT_\coarse|_\Gamma$. Throughout, let $\widehat\TT_\coarse \in \widehat\T$ be the coarsest extension of $\TT_\coarse \in \T$. Recall that NVB is a binary refinement rule. Therefore, $\TT_\fine \in \refine(\TT_\coarse)$ also implies that $\widehat\TT_\fine \in \refine(\widehat\TT_\coarse)$. Finally, we note that all triangulations $\widehat\TT_\coarse \in \widehat\T$ are uniformly $\gamma$-shape regular, i.e.,
\begin{align*}
 \max_{\widehat T \in \widehat \TT_\coarse} \frac{\diam(\widehat T)}{|\widehat{T}|^{1/3}} \le \gamma < \infty.
\end{align*}%
where $\gamma$ depends only on $\widehat\TT_0$.

Our argument adapts ideas from~\cite{hiptmairmao12}, where a subspace decomposition for 
the lowest-order N\'ed\'elec space $\ND^1(\widehat\TT_\coarse)$ (see, e.g., \cite{hiptmairzheng09}) in  $\Hcurl\Omega$ implies a decomposition of the corresponding discrete trace space. While the original idea dates back to~\cite{oswald99}, a nice summary of the argument is found in~\cite[Section~2]{hiptmairmao12}.

\begin{remark}
{\rm(i)} 
Our proof is based on the construction of an extension operator from $\PP_*^0(\TT_\bullet)$ to
$\ND^1(\widehat\TT_\bullet)$, see Lemma~\ref{prop:discreteExtension} below.
  It is not clear if such an operator can be constructed for the case
  $\Gamma\subsetneqq\partial\Omega$.
  
{\rm(ii)}
In~\cite{MR3378840}, a subspace decomposition of the lowest-order Raviart--Thomas space
$\RT^0(\widehat\TT_\coarse)$ (see, e.g., \cite{xcn09}) in $\Hdiv\Omega$ implies a decomposition of the corresponding normal trace space $\PP^0(\TT_\coarse)$.
Due to different scaling properties of the Raviart--Thomas basis functions (in the $\Hdiv\Omega$ norm) and their normal trace (in the $H^{-1/2}(\Gamma)$ norm), this argument does not apply in our case.
\end{remark}

\subsection{Discrete spaces and extensions}
Let $\widehat\EE_\coarse$ (resp.\ $\widehat\NN_\coarse$) denote the set of all edges (resp.\ all nodes) of $\widehat\TT_\coarse \in \widehat\T$.
For each node $\x\in\widehat\NN_\coarse$, let $\eta_{\coarse,\x} \in \mathcal{S}^1(\widehat\TT_\coarse)$ be the corresponding hat function, i.e., $\eta_{\coarse,\x}$ is $\widehat\TT_\coarse$-piecewise affine and globally continuous with $\eta_{\coarse,\x}(\y) = \delta_{\x\y}$ for all $\x, \y \in \widehat\NN_\coarse$.
For $E \in \widehat\EE_\coarse$, let $\uu_{\coarse,E} \in\ND^1(\widehat\TT_\coarse)$ denote the corresponding N\'ed\'elec basis function, i.e., for $K \in \widehat\TT_\coarse$ with $E = \conv\{\x,\y\} \subset \partial K$, it holds that
\begin{align}\label{eq:ND:basis}
  \uu_{\coarse,E}|_K = C(\eta_{\coarse,\x}\nabla\eta_{\coarse,\y}-\eta_{\coarse,\y}\nabla\eta_{\coarse,\x}),
\end{align}
where $C > 0$ is chosen such that $\int_{E'} \uu_{\coarse,E} \,ds = |E| \, \delta_{E E'}$ for all
$E,E'\in \widehat\EE_\coarse$. 
Scaling arguments yield the next lemma.
The proof follows the lines of~\cite[Lemma~5.7]{hiptmairmao12}.

\pagebreak
\vspace*{-6mm}
\begin{lemma}\label{lem:localext}
For $E\in\EE_\coarse$, recall the Haar function $\varphi_{\coarse,E} \in \PP^0(\TT_\coarse)$ from~\eqref{eq:psi}. Let $\uu_{\coarse,E} \in\ND^1(\widehat\TT_\coarse)$ denote the corresponding N\'ed\'elec basis function; see~\eqref{eq:ND:basis}.
Then, 
\begin{align}
 \varphi_{\coarse,E} = \curl \uu_{\coarse,E} \cdot \normal|_\Gamma
\, \text{ and } \, \,
 \norm{\varphi_{\coarse,E}}{H^{-1/2}(\Gamma)} 
 \leq \norm{\uu_{\coarse,E}}{\Hcurl\Omega}
  \leq C \, \norm{\varphi_{\coarse,E}}{H^{-1/2}(\Gamma)},
\end{align}
where $C>0$ depends only on $\Omega$ and the $\gamma$-shape regularity of $\widehat\TT_\bullet$. \qed
\end{lemma}

The following lemma holds for (simply) connected Lipschitz domains $\Omega$ and follows essentially from~\cite{MR3522965}.
Recall $\PP^0_*(\TT_\coarse)$ from~\eqref{eq:P0*}.

\begin{lemma}\label{prop:discreteExtension}
There exists a linear operator $\extND_\bullet : \PP_*^0(\TT_\coarse) \to\ND^1(\widehat\TT_\coarse)$ such that
\begin{align}\label{eq:prop:discreteExtension}
 \curl(\extND_\bullet\psi_\bullet) \cdot \normal|_\Gamma = \psi_\bullet
 \text{ and }
 \norm{\extND_\bullet\psi_\bullet}{\Hcurl\Omega}
 \leq C \, \norm{\psi_\bullet}{H^{-1/2}(\Gamma)}
 \text{ for all } \psi_\bullet \in \PP_*^0(\TT_\coarse).
\end{align}
The constant $C > 0$ depends only on $\gamma$-shape regularity of $\widehat\TT_\bullet$.
\end{lemma}

\begin{proof}
Let $\psi_\bullet \in \PP_*^0(\TT_\coarse)$.
First, \cite[Theorem~2.1]{MR3522965} provides $\ssigma_\coarse \in \RT^0(\widehat\TT_\coarse)$ with 
\begin{align*}
 \ssigma_\coarse\cdot\normal|_\Gamma = \psi_\bullet, 
 \quad
 \divergence\ssigma_\coarse = 0,
 \quad \text{and} \quad
 \norm{\ssigma_\coarse}{\Hdiv\Omega} \lesssim \norm{\psi_\bullet}{H^{-1/2}(\partial\Omega)}.
\end{align*}
Then,~\cite[Lemma~4.3]{MR3522965} provides $\extND_\bullet \psi_\bullet := \vv_\coarse\in\ND^1(\widehat\TT_\coarse)$ such that 
\begin{align*}
 \curl\vv_\coarse = \ssigma_\coarse
 \quad \text{and} \quad
 \norm{\vv_\coarse}{\Hcurl\Omega}\lesssim \norm{\ssigma_\coarse}{\Hdiv\Omega}.
\end{align*}
Combining these results, we conclude the proof.
\end{proof}

\subsection{Abstract additive Schwarz preconditioners}
Let $\SX$ denote some finite dimensional Hilbert space with norm $\norm{\cdot}\SX$ and subspace decomposition
\begin{align*}
  \SX = \sum_{i\in\mathcal I} \SX_i,
\end{align*}
where $\mathcal{I}$ is a finite index set.  
The additive Schwarz operator is given by $\PAS = \sum_{i\in\mathcal I} \PAS_i$, where $\PAS_i$
is the $\SX$-orthogonal projection onto $\SX_i$, i.e., 
\begin{align*}
  \dual{\PAS_i x}{x_i}_{\SX} = \dual{x}{x_i}_{\SX} \quad\text{for all } x_i\in \SX_i \text{ and all }x\in\SX,
\end{align*}
where $\dual\cdot\cdot_{\SX}$ denotes the scalar product on $\SX$.
Then, the operator $\PAS$ is positive definite and symmetric
(with respect to $\dual\cdot\cdot_{\SX}$).
Define the multilevel norm
\begin{align}\label{eq:multilevelnorm}
  \enorm{x}_\SX^2 := \inf\set[\Big]{\sum_{i\in\mathcal I} \norm{x_i}{\SX}^2}{x=\sum_{i\in\mathcal{I}} x_i
  \quad\text{with } x_i\in \SX_i \text{ for all } i \in \mathcal{I} }.
\end{align}
It is proved, e.g., in~\cite[Theorem~16]{oswald94}
that $\dual{\PAS^{-1}x}{x}_{\SX} = \enorm{x}_\SX^2$. Let $C > 0$. If
\begin{align*}
  C^{-1} \norm{x}\SX \leq \enorm{x}_\SX \leq C \, \norm{x}\SX
  \quad \text{for all } x \in \SX,
\end{align*}
then the extreme eigenvalues of $\PAS^{-1}$ (and hence those of $\PAS$) are bounded (from above and below).
In particular, the additive Schwarz operator $\PAS$ 
is optimal in the sense that its
condition number (ratio of largest and smallest eigenvalues) depends only on $C>0$.

Let $\Smat$ denote the matrix representation of $\PAS$. Then, the norm equivalence from above and the
latter observations imply that the condition number of $\Smat$ is bounded.
The abstract theory on additive Schwarz operators given in~\cite[Chapter~2]{toswid} shows that $\Smat$ has the form
$\Smat = \P^{-1}\Amat$, where $\Amat$ is the Galerkin matrix of $\dual\cdot\cdot_\SX$.
Therefore, boundedness of the condition number of $\Smat$ implies \emph{optimality} of the preconditioner $\P^{-1}$.
We shortly discuss the matrix representation~\eqref{eq:defPrec} of the additive Schwarz preconditioner $\P^{-1}$.
Following~\cite[Chapter~2]{toswid}, let $\Amat_i$ denote the Galerkin matrix of $\dual{\cdot}{\cdot}_\SX$ restricted to
$\SX_i$, and let $\embed_i$ denote the matrix that realizes the embedding from $\SX_i\to\SX$.
We consider the matrix representation of $\PAS_i : \SX\to\SX_i\subset \SX$.
Let $x\in\SX$ with coordinate vector $\x$, and let $x_i\in\SX_i$ be arbitrary with coordinate vector $\x_i$.
The defining relation
\begin{align*}
  \dual{\PAS_i x}{x_i}_{\SX} = \dual{x}{x_i}_{\SX} \quad\text{for all } x_i\in \SX_i
\end{align*}
of $\PAS_i$ then reads in matrix-vector form (with $\Smat_i$ being the matrix representation of $\PAS_i$) as
\begin{align*}
  \x_i\cdot(\Amat_i \Smat_i\x) = (\embed_i\x_i)\cdot(\Amat\x) \quad\text{for all coefficient vectors } \x_i,
\end{align*}
or equivalently
\begin{align*}
  \Amat_i \Smat_i\x = \embed_i^{T}\Amat\x.
\end{align*}
Since $\Amat_i$ is invertible, we have that
\begin{align*}
  \Smat_i = \Amat_i^{-1} \embed_i^T\Amat.
\end{align*}
Note that the range of the operator $\PAS_i$ is $\SX_i$ and correspondingly for the matrix representation $\Smat_i$.
We therefore apply the embedding $\embed_i$ and obtain the representation 
\begin{align*}
  \Smat = \P^{-1}\Amat, 
  \quad \text{where} \quad
  \P^{-1} = \sum_{i\in\mathcal I} \embed_i \Amat_i^{-1} \embed_i^T.
\end{align*}
To finally prove~\eqref{eq:defPrec},
note that for one-dimensional subspaces $\SX_i$, $\Amat_i$
reduces to the diagonal entry of the matrix $\Amat$. Overall, we thus derive the matrix representation~\eqref{eq:defPrec}.

\subsection{Subspace decomposition of $\boldsymbol{\ND^1(\widehat\TT_\coarse)}$ in $\boldsymbol{\Hcurl\Omega}$}
The following result is taken from~\cite[Theorem~4.1]{hiptwuzheng2012}; see also the references therein.
In particular, we note that their proof requires the assumption that $\Omega$ is simply connected.

\begin{proposition}\label{thm:decomp:nedelec}
Let $\SY_\bullet := \ND^1(\widehat\TT_\bullet)$, $\SY_{\bullet,E} := \linhull\{\uu_{\coarse,E}\}$,
$\SY_{\coarse,\x} := \linhull\{\nabla \eta_{\coarse,\x}\}$, and
\begin{align*}
 \widehat\EE_\ell^\star 
 &:= (\widehat\EE_\ell \setminus \widehat\EE_{\ell-1})
 \cup \set{E\in\widehat\EE_\ell}{\supp\uu_{\ell,E}\subsetneqq\supp\uu_{\ell-1,E}},
 \\
 \widehat\NN_\ell^\star &:= (\widehat\NN_\ell\setminus\widehat\NN_{\ell-1}) 
 \cup \set{\x\in\widehat\NN_\ell}{\supp\eta_{\ell,\x}\subsetneqq\supp\eta_{\ell-1,\x}}.
\end{align*}
Then, it holds that
\begin{align}\label{eq:decomp:nedelec}
  \SY_L = \SY_0 + \sum_{\ell=1}^L \bigg(\sum_{E\in\widehat\EE_\ell^\star} \SY_{\ell,E}
  + \sum_{\x\in\widehat\NN_\ell^\star} \SY_{\ell,\x}\bigg).
\end{align}
Moreover, it holds that
\begin{align}\label{eq:Hcurl:prop}
 C^{-1}\norm{\vv}{\Hcurl\Omega} \leq \enorm{\vv}_{\SY_L}
 \leq C \, \norm{\vv}{\Hcurl\Omega} 
 \quad\text{for all } \vv \in\SY_L,
\end{align}
where $C > 0$ depends only on $\Omega$ and $\widehat\TT_0$.\qed
\end{proposition}

\subsection{Subspace decomposition of $\boldsymbol{\PP^0(\TT_\bullet)}$ in $\boldsymbol{H^{-1/2}(\Gamma)}$}
It remains to prove the following proposition to conclude the proof of Theorem~\ref{thm:precond}.

\begin{proposition}
  The multilevel norm $\enorm\cdot_{\SX_L}$ associated with  the decompomposition~\eqref{eq:decomp}
  satisfies the equivalence
  \begin{align}\label{eq:precond:prop}
    C^{-1}\norm{\psi}{H^{-1/2}(\Gamma)} \leq \enorm{\psi}_{\SX_L}
    \leq C \, \norm{\psi}{H^{-1/2}(\Gamma)} \quad\text{for all }
    \psi \in\PP^0(\TT_L),
  \end{align}
  where $C>0$ depends only on $\Omega$ and $\widehat\TT_0$.
\end{proposition}

\begin{proof}[Proof of lower estimate in~\eqref{eq:precond:prop}]
Let $\psi\in\PP^0(\TT_L)$ with arbitrary decomposition
\begin{align}\label{eq:as:lower}
 \psi = \psi_0 + \psi_*, \quad \psi_*=\sum_{\ell=1}^L \sum_{E\in\EE_\ell^\star} \psi_{\ell,E}
 \quad \text{with} \quad \psi_0\in \SX_0 \text{ and } \psi_{\ell,E} \in \SX_{\ell,E}.
\end{align}
Note that $\SX_{\ell,E} \subset \PP^0_*(\TT_\ell)$.
Recall the extension operator $\extND_\ell$ from Lemma~\ref{prop:discreteExtension}.
Define $\vv_* := \sum_{\ell=1}^L \sum_{E\in\EE_\ell^\star} \extND_\ell\psi_{\ell,E} \in \SY_L$.
Then, $\curl\vv_* \cdot \normal|_\Gamma = \psi_*$ and hence 
\begin{align*}
 \norm{\psi_*}{H^{-1/2}(\Gamma)}
 \lesssim \norm{\curl\vv_*}{\Hdiv\Omega} = \norm{\curl\vv_*}{L^2(\Omega)}
 \reff{eq:Hcurl:prop}\lesssim \enorm{\vv_*}_{\SY_L}^2
\end{align*}
from the continuity of the trace operator in $\Hdiv\Omega$.
Moreover, the triangle inequality, the lower bound from Proposition~\ref{thm:decomp:nedelec}, and Lemma~\ref{prop:discreteExtension} show that
\begin{align*}
 &\norm{\psi}{H^{-1/2}(\Gamma)}^2 
 \lesssim \norm{\psi_0}{H^{-1/2}(\Gamma)}^2 + \norm{\psi_*}{H^{-1/2}(\Gamma)}^2
\lesssim \norm{\psi_0}{H^{-1/2}(\Gamma)}^2 + \enorm{\vv_*}_{\SY_L}^2
 \\&\quad
 \leq \norm{\psi_0}{H^{-1/2}(\Gamma)}^2 
 + \sum_{\ell=1}^L \sum_{E\in\EE_\ell^\star} \norm{\extND_\ell\psi_{\ell,E}}{\Hcurl\Omega}^2 
 \reff{eq:prop:discreteExtension}\lesssim 
 \norm{\psi_0}{H^{-1/2}(\Gamma)}^2 + \sum_{\ell=1}^L \sum_{E\in\EE_\ell^\star}
 \norm{\psi_{\ell,E}}{H^{-1/2}(\Gamma)}^2.
\end{align*}%
Taking the infimum over all possible decompositions~\eqref{eq:as:lower}, we derive the lower estimate in~\eqref{eq:precond:prop} by definition~\eqref{eq:multilevelnorm} of the multilevel norm.
\end{proof}

\begin{proof}[Proof of upper estimate in~\eqref{eq:precond:prop}]
Let $\psi \in \PP^0(\TT_L)$. Define $\psi_{00} := \dual{\psi}1_\Gamma/|\Gamma|$ and $\psi_* := \psi - \psi_{00} \in \PP^0_*(\TT_L)$. Note that
\begin{align}\label{eq:precond:dp1}
 \norm{\psi_*}{H^{-1/2}(\Gamma)} 
 \le \norm{\psi}{H^{-1/2}(\Gamma)} + \norm{\psi_{00}}{H^{-1/2}(\Gamma)}
 \le \big( 1 + \norm{1/|\Gamma|}{H^{1/2}(\Gamma)} \big) \, \norm{\psi}{H^{-1/2}(\Gamma)}.
\end{align}%
With Lemma~\ref{prop:discreteExtension}, choose $\vv = \extND_L \psi_* \in \SY_L = \ND^1(\widehat\TT_L)$. Note that $\psi_* = \curl \vv \cdot \normal|_\Gamma$ and  $\norm{\vv}{\Hcurl\Omega} \lesssim \norm{\psi_*}{H^{-1/2}(\Gamma)}$. The upper bound in Proposition~\ref{thm:decomp:nedelec} further provides $\vv_0 \in \SY_0$, $\vv_{\ell,E} \in \SY_{\ell,E}$, and $\vv_{\ell,\x} \in \SY_{\ell,\x}$ such that
\begin{align*}
 \vv = \vv_0 
 + \sum_{\ell=1}^L\bigg(\sum_{E\in\widehat\EE_\ell^\star} \vv_{\ell,E} 
 + \sum_{\x\in\widehat\NN_\ell^\star} \vv_{\ell,\x}\bigg)
\end{align*}
as well as
\begin{align}\label{eq:precond:dp2}
 \norm{\vv_0}{\Hcurl\Omega}^2 
 + \sum_{\ell=1}^L \bigg( \sum_{E\in\widehat\EE_\ell^\star} \norm{\vv_{\ell,E}}{\Hcurl\Omega}^2 
 +\! \sum_{\x\in\widehat\NN_\ell^\star} \norm{\vv_{\ell,\x}}{\Hcurl\Omega}^2 \bigg)
 \reff{eq:Hcurl:prop}\lesssim \norm{\vv}{\Hcurl\Omega}^2.
\end{align}
Observe that $\curl \vv_{\ell,\x} = 0$, since $\vv_{\ell,\x} \in \SY_{\ell,\x} = \linhull\{\nabla \eta_{\ell,\x}\}$. Thus, we see that
\begin{align*}
 \psi 
 = \psi_{00} + \psi_*
 = \psi_{00} + \curl \vv \cdot \normal|_\Gamma
 &= \psi_{00} + \curl \vv_0 \cdot \normal|_\Gamma
 + \sum_{\ell=1}^L \sum_{E\in\widehat\EE_\ell^\star} \curl \vv_{\ell,E} \cdot \normal|_\Gamma.
 \\&
 = \psi_{00} + \curl \vv_0 \cdot \normal|_\Gamma
 + \sum_{\ell=1}^L \sum_{E\in\EE_\ell^\star} \curl \vv_{\ell,E} \cdot \normal|_\Gamma.
\end{align*}
Note that $\psi_{*0} := \curl \vv_0 \cdot \normal|_\Gamma \in \SX_0 = \PP^0(\TT_0)$ and hence $\psi_{00} + \psi_{*0} \in \SX_0$. Note that
\begin{align}\label{eq:precond:dp3}
 \begin{split}
 &\norm{\psi_{00} + \psi_{*0}}{H^{-1/2}(\Gamma)}
 \le \norm{\psi_{00}}{H^{-1/2}(\Gamma)} + \norm{\curl \vv_0 \cdot \normal}{H^{-1/2}(\Gamma)}
 \\&\quad
 \lesssim \norm{\psi}{H^{-1/2}(\Gamma)} + \norm{\curl \vv_0}{\Hdiv\Omega}
 = \norm{\psi}{H^{-1/2}(\Gamma)} + \norm{\vv_0}{L^2(\Omega)}.
 \end{split}
\end{align}
Due to Lemma~\ref{lem:localext} and $\vv_{\ell,E} \in \SY_{\ell,E} = \linhull\{\uu_{\ell,E}\}$, it holds that $\psi_{\ell,E} := \curl \vv_{\ell,E} \cdot \normal|_\Gamma \in \SX_{\ell,E} = \linhull\{\varphi_{\ell,E}\}$ with $\norm{\psi_{\ell,E}}{H^{-1/2}(\Gamma)} \simeq \norm{\vv_{\ell,E}}{\Hcurl\Omega}$.
We hence see that
\begin{align*}
 \psi = (\psi_{00} + \psi_{*0}) + \sum_{\ell=1}^L \sum_{E\in\EE_\ell^\star} \psi_{\ell,E}
\end{align*}
with
\begin{align*}
  \enorm{\psi}_{\PP^0(\TT_L)}^2 
 &\reff{eq:multilevelnorm}\le \norm{\psi_{00} + \psi_{*0}}{H^{-1/2}(\Gamma)}^2
 + \sum_{\ell=1}^L \sum_{E\in\EE_\ell^\star} \norm{\psi_{\ell,E}}{H^{-1/2}(\Gamma)}^2
 \\&
 \reff{eq:precond:dp3}\lesssim \norm{\psi}{H^{-1/2}(\Gamma)}^2
 + \norm{\vv_0}{L^2(\Omega)}^2 + \sum_{\ell=1}^L \sum_{E\in\EE_\ell^\star} \norm{\vv_{\ell,E}}{\Hcurl\Omega}^2
 \\&
 \reff{eq:precond:dp2}\lesssim \norm{\psi}{H^{-1/2}(\Gamma)}^2 + \norm{\vv}{\Hcurl\Omega}^2
 \reff{eq:prop:discreteExtension}\lesssim \norm{\psi}{H^{-1/2}(\Gamma)}^2 + \norm{\psi_*}{H^{-1/2}(\Gamma)}^2
 \reff{eq:precond:dp1}\lesssim \norm{\psi}{H^{-1/2}(\Gamma)}^2.
\end{align*}
This concludes the proof.
\end{proof}

\newcommand{\XX}{\mathcal{X}}

\section{Proof of Theorem~\ref{theorem:algorithm} (Rate Optimality of Adaptive Algorithm)}
\label{proof:algorithm}%

In the spirit of~\cite{axioms}, we give an abstract analysis, where the precise problem and discretization (i.e., Galerkin BEM with piecewise constants for the weakly-singular integral equation for the 2D and 3D Laplacian) enter only through certain properties of the error estimator. These properties are explicitly stated in Section~\ref{section:axioms}, before Section~\ref{section:pcg} provides general
PCG estimates. 
The remaining sections (Section~\ref{section:proof:thm:a}--\ref{section:proof:cor}) then only exploit these abstract framework to prove Theorem~\ref{theorem:algorithm} and Corollary~\ref{corollary:algorithm}.

\subsection{Axioms of adaptivity}
\label{section:axioms}%
In this section, we recall some structural properties of the residual error estimator \eqref{eq:estimator} which have been identified in~\cite{axioms} to be important and sufficient for the numerical analysis of Algorithm \ref{algorithm}. For the proof, we refer to~\cite{fkmp13,partOne}. We only note that~\eqref{axiom:discrete_reliability} already implies~\eqref{axiom:reliability} with $\Crel \le \Cdrl$ in general; see~\cite[Section~3.3]{axioms}.

For ease of notation, let $\TT_0$ be the fixed initial mesh of Algorithm~\ref{algorithm}. Let $\mathbb{T}:=\refine(\TT_0)$ be the set of all possible meshes that can be obtained by successively refining $\TT_0$.

\begin{proposition}\label{prop:axioms}
There exist constants $\Cstb,\Cred,\Crel>0$ and $0<\qred<1$ which depend only on $\Gamma$ and the $\gamma$-shape regularity, such that the following properties \eqref{axiom:stability}--\eqref{axiom:discrete_reliability} hold:
\begin{enumerate}
\renewcommand{\theenumi}{A\arabic{enumi}}
\bf
\item\label{axiom:stability}
\rm
\textbf{stability on non-refined element domains:} For each mesh $\TT_\coarse\in\mathbb{T}$, all refinements $\TT_\circ\in\operatorname{refine}(\TT_\coarse)$, arbitrary discrete functions $v_\coarse\in\PP^0(\TT_\coarse)$ and $v_\circ\in\PP^0(\TT_\circ)$, and an arbitrary set $\mathcal{U}_\coarse\subseteq\mathcal{T}_\coarse\cap\mathcal{T}_\circ$ of non-refined elements, it holds that
$$|\eta_\circ(\mathcal{U}_\coarse,v_\circ)-\eta_\coarse(\mathcal{U}_\coarse,v_\coarse)| \leq C_{\rm stb}\,\enorm{v_\circ-v_\coarse}.$$
\bf
\item\label{axiom:reduction}
\rm
\textbf{reduction on refined element domains:} For each mesh $\TT_\coarse\in\mathbb{T}$, all refinements $\TT_\circ\in\operatorname{refine}(\TT_\coarse)$, and arbitrary 
$v_\coarse\in\PP^0(\TT_\coarse)$ and $v_\circ\in\PP^0(\TT_\circ)$, it holds that
$$\eta_\circ(\TT_\circ\backslash\TT_\coarse,v_\circ)^2 \le \qred\,\eta_\coarse(\TT_\coarse\backslash\TT_\circ,v_\coarse)^2 + \Cred\,\enorm{v_\circ-v_\coarse}^2.$$
\bf
\item\label{axiom:reliability}
\rm
\textbf{reliability:}
For each mesh $\TT_\coarse\in\mathbb{T}$, the error of the exact discrete solution $\PPhi_\coarse^\exact \in \PP^0(\TT_\coarse)$ of \eqref{eq:galerkin} is controlled by
\begin{align*}
	\enorm{\PPhi^\exact-\PPhi_\coarse^\exact} \le \Crel\,\eta_\coarse(\PPhi_\coarse^\exact).
\end{align*}

\bf
\item\label{axiom:discrete_reliability}
\rm
\textbf{discrete reliability:}
For each mesh $\TT_\coarse\in\mathbb{T}$ and all refinements $\TT_\circ\in\operatorname{refine}(\TT_\coarse)$, there exists a set $\RR_{\coarse,\circ}\subseteq\TT_\coarse$ 
with $\TT_\coarse\backslash\TT_\circ \subseteq \RR_{\coarse,\circ}$ as well as $\#\RR_{\coarse,\circ} \le \Cdrl\,\#(\TT_\coarse\backslash\TT_\circ)$ such that the difference of $\PPhi_\coarse^\exact \in \PP^0(\TT_\coarse)$ and $\PPhi_\circ^\exact \in \PP^0(\TT_\circ)$ is controlled by
\begin{align*}
	\enorm{\PPhi_\circ^\exact-\PPhi_\coarse^\exact} \le \Cdrl\,\eta_\coarse(\RR_{\coarse,\circ},\PPhi_\coarse^\exact).
	\qquad\qed\hspace*{-20mm}
\end{align*}
\end{enumerate}
\end{proposition}

\subsection{Energy estimates for the PCG solver}
\label{section:pcg}%
This section collects some auxiliary results which rely on the use of PCG and, in particular,
PCG with an optimal preconditioner. We first note the following Pythagoras identity.

\begin{lemma}
Let $\A_\coarse, \P_\coarse \in \R^{N \times N}$ be symmetric and positive definite, $\b_\coarse\in \R^N$, $\x_\coarse^\exact := \A_\coarse^{-1} \b_\coarse$, $\x_{\coarse 0} \in \R^N$ and $\x_{\coarse k}$ the iterates of the PCG algorithm.\\
There holds the Pythagoras identity
\begin{align}\label{eq:gal_orth}
\enorm{\PPhi_\coarse^\exact-\PPhi_{\coarse (k+1)}}^2+\enorm{\PPhi_{\coarse (k+1)}-\PPhi_{\coarse k}}^2=\enorm{\PPhi_\coarse^\exact-\PPhi_{\coarse k}}^2\quad\textrm{for all }k\in\N_0.
\end{align}
\end{lemma}

\begin{proof}
\def\residual{\widetilde{\boldsymbol{r}}}%
According to the definition of PCG (and CG), it holds that
\begin{align*}
 \norm{\widetilde\x_\coarse^\exact-\widetilde\x_{\coarse k}}{\widetilde \A_\coarse}
 = \min_{\widetilde \y_\coarse\in\mathcal{K}_k(\widetilde \A_\coarse,\widetilde \b_\coarse,\widetilde\x_{\coarse0})}\norm{\widetilde \x_\coarse^\exact-\widetilde \y_\coarse}{\widetilde \A_\coarse},
\end{align*}
where $\mathcal{K}_k(\widetilde \A_\coarse,\widetilde \b_\coarse,\widetilde x_{\coarse0}):={\rm span}\{\residual_{\bullet0},\widetilde \A_\coarse \residual_{\bullet0},\ldots,\widetilde \A_\coarse^{k-1} \residual_{\bullet0}\}$
with $\residual_{\bullet0}:=\widetilde \b_\coarse-\widetilde\A_\coarse\widetilde\x_{\coarse0}$.
According to Linear Algebra, $\widetilde \x_{\coarse k}$ is the orthogonal projection of $\widetilde \x_\coarse^\exact$ in $\mathcal{K}_k(\widetilde \A_\coarse,\widetilde \b_\coarse,\widetilde\x_{\coarse0})$ with respect to the matrix norm $\norm{\cdot}{\widetilde \A_\coarse}$. From nestedness $\mathcal{K}_k(\widetilde \A_\coarse,\widetilde \b_\coarse,\widetilde\x_{\coarse0})\subseteq\mathcal{K}_{k+1}(\widetilde \A_\coarse,\widetilde \b_\coarse,\widetilde\x_{\coarse0})$, it thus follows that
\begin{align*}
\norm{\widetilde\x_\coarse^\exact-\widetilde\x_{\coarse k}}{\widetilde\A_\coarse}^2=\norm{\widetilde\x_\coarse^\exact-\widetilde\x_{\coarse (k+1)}}{\widetilde\A_\coarse}^2+\norm{\widetilde\x_{\coarse(k+1)}-\widetilde\x_{\coarse k}}{\widetilde\A_\coarse}^2.
\end{align*}
Together with \eqref{eq:normequi1}	and \eqref{eq:normequi2}, this proves \eqref{eq:gal_orth}.
\end{proof}%

The following lemma collects some estimates which follow from the contraction property~\eqref{eq2:pcg} of PCG.

\begin{lemma}\label{lemma2:pcg}
Algorithm~\ref{algorithm} guarantees the following estimates for all $(j,k) \in \QQ$ with $k \ge 1$:
\begin{itemize}
\item[\rm(i)] $\enorm{\phi_j^\exact - \PPhi_{jk}} \le \qpcg \, \enorm{\phi_j^\exact - \PPhi_{j(k-1)}}$
\item[\rm(ii)]
$\enorm{\PPhi_{jk}-\PPhi_{j(k-1)}} \leq (1+\qpcg) \, \enorm{\PPhi_j^\exact-\PPhi_{j(k-1)}}$
\item[\rm(iii)] 
$\enorm{\PPhi_j^\exact-\PPhi_{j(k-1)}}
 \leq (1-\qpcg)^{-1} \, \enorm{\PPhi_{jk}-\PPhi_{j(k-1)}}$
\item[\rm(iv)]
$\enorm{\PPhi_j^\exact-\PPhi_{jk}}
 \leq \qpcg (1-\qpcg)^{-1} \, \enorm{\PPhi_{jk}-\PPhi_{j(k-1)}}$
\end{itemize}
\end{lemma}

\begin{proof}
According to~\eqref{eq:normequi2}, estimate~\eqref{eq2:pcg} proves~(i).
The estimates~(ii)--(iv) follow from~(i) and the triangle inequality.
\end{proof}

\subsection{Proof of Theorem~\ref{theorem:algorithm}(a)}
\label{section:proof:thm:a}%
With reliability~\eqref{axiom:reliability} and stability~\eqref{axiom:stability}, we see that
\begin{align*}
 \enorm{\phi^\exact - \PPhi_{jk}}
 \le \enorm{\phi^\exact - \PPhi_j^\exact}
 + \enorm{\PPhi_j^\exact - \PPhi_{jk}}
 &\reff{axiom:reliability}\lesssim \eta_j(\PPhi_j^\exact) + \enorm{\PPhi_j^\exact - \PPhi_{jk}}
 \\&
 \reff{axiom:stability}\lesssim \eta_j(\PPhi_{jk}) + \enorm{\PPhi_j^\exact - \PPhi_{jk}}
 \quad \text{for all } (j,k) \in \QQ.
\end{align*}
With Lemma~\ref{lemma2:pcg}(iv), we hence prove the reliability estimate~\eqref{eq:algorithm:reliability}.

According to~\cite{invest}, it holds that
\begin{align*}
 \eta_j(\PPhi_{jk})
 &\lesssim \norm{h_j^{1/2}(\phi^\exact - \PPhi_{jk})}{L^2(\Gamma)}
 + \enorm{\phi^\exact - \PPhi_{jk}}
 \\&
 \le \norm{h_j^{1/2}(\phi^\exact - \PPhi_{jk})}{L^2(\Gamma)}
 + \enorm{\phi^\exact - \PPhi_j^\exact}
 + \enorm{\PPhi_j^\exact - \PPhi_{jk}}
\end{align*}
Let $\mathbb{G}_j: \H^{-1/2}(\Gamma) \to \PP^0(\TT_j)$ be the Galerkin projection. Let $\Pi_j: L^2(\Gamma) \to \PP^0(\TT_j)$ be the $L^2$-orthogonal projection. With the C\'{e}a lemma and a duality argument (see, e.g.,~\cite[Theorem~4.1]{MR2199747}), we see that

\begin{align*}
\enorm{(1-\mathbb{G}_j)\psi}\leq\enorm{(1-\Pi_j)\psi}\lesssim\norm{h_j^{1/2}\,\psi}{L^2(\Gamma)}\quad\textrm{for all~}\psi\in L^2(\Gamma).
\end{align*}
Hence, for $\psi=\pphi^\exact-\PPhi_{jk}$, it follows that
\begin{align*}
\enorm{\pphi^\exact-\PPhi_j^\exact}
= \enorm{(1-\mathbb{G}_j)\pphi^\exact}
= \enorm{(1-\mathbb{G}_j)(\pphi^\exact-\PPhi_{jk})}
\lesssim \norm{h_j^{1/2}(\pphi^\exact-\PPhi_{jk})}{L^2(\Gamma)}.
\end{align*}
Combining the latter estimates, we see that
\begin{align*}
 \eta_j(\PPhi_{jk})
 &\lesssim \norm{h_j^{1/2}(\phi^\exact - \PPhi_{jk})}{L^2(\Gamma)}
 + \enorm{\PPhi_j^\exact - \PPhi_{jk}}.
\end{align*}
With Lemma~\ref{lemma2:pcg}(iv), we hence prove the efficiency estimate~\eqref{eq:algorithm:efficiency}.\qed

\subsection{Proof of Theorem~\ref{theorem:algorithm}(b)}
\label{section:linear_convergence}
The following lemma is the heart of the proof of Theorem~\ref{theorem:algorithm}(b).

\begin{lemma}\label{lem:contr}
Consider Algorithm~\ref{algorithm} for arbitrary parameters $0 < \theta \le 1$ and $\lambda > 0$. 
There exist constants $0<\mu,\qctr<1$ such that
\begin{align*}
\Delta_{jk}:=\mu\,\eta_j(\PPhi_{jk})^2+\enorm{\phi^\exact-\PPhi_{jk}}^2
\quad \text{for } (j, k) \in \QQ
\end{align*}
satisfies, for all $j \in \N_0$, that
\begin{align}\label{eq1:lem:contr}
\Delta_{j(k+1)}\leq\qctr\,\Delta_{jk}\quad\textrm{for all }0\leq k<k+1<\k
\end{align}
as well as 
\begin{align}\label{eq2:lem:contr}
\Delta_{(j+1)0}\leq\qctr\,\Delta_{j(\k-1)}\quad\textrm{for }k=0.
\end{align}
Moreover, for all $(j',k'),(j,k)\in\QQ$, it holds that
\begin{align}\label{eq:contr3}
\Delta_{j'k'}\leq\qctr^{|(j',k')|-|(j,k)|}\,\Delta_{jk}\quad\textrm{provided that $(j',k')>(j,k)$, $k'<\k(j')$, and $k<\k(j)$.}
\end{align}
The constants $0 < \mu, \qctr < 1$ depend only on $\lambda$, $\theta$, $\qpcg$, and the constants in~\eqref{axiom:stability}--\eqref{axiom:reliability}.
\end{lemma}

\begin{proof}
The proof is split into five steps.

{\bf Step~1.}\quad
We fix some constants, which are needed below. 
We note that all these constants depend on $0 < \theta \le 1$ and $\lambda > 0$, but do not require any additional constraint.
First, define 
\begin{align}\label{eq:def:qest}
 0 < \qest := 1-(1-\qred)\theta^2 < 1.
\end{align}
Second, choose $\gamma>0$ such that
\begin{align}
(1+\gamma)\,\qest&\reff{eq:def:qest}<1.\label{pr_contr:aux1}
\end{align}
Third, choose $\mu>0$ such that
\begin{align}\label{eq:0601a}
\mu\,(1+\gamma^{-1})\,\qest\,\Cstb^2(1+\qpcg)^2 ~<~ \frac{1-\qpcg^2}{2}
\qquad \text{and} \quad
\mu\,\lambda^{-2} ~\leq~ \frac{1}{2}.
\end{align}%
Fourth, choose $\varepsilon>0$ such that
\begin{align}\label{eq:0601b}
\varepsilon\,(1-\qpcg)^{-2}+2\,\varepsilon\,\Crel^2\,\Cstb^2\,(1-\qpcg)^{-2} ~\leq~ \frac{1}{2}
\qquad \text{and} \quad
2\,\varepsilon\,\Crel^2 ~\leq~ (1-\varepsilon)\mu.
\end{align}%
Fifth, choose $\kappa>0$ such that
\begin{align}\label{eq:0601c}
2\,\kappa\,\Crel^2 \stackrel{\eqref{pr_contr:aux1}}{<} \big(1-(1+\gamma)\,\qest\big)\,\mu
\qquad \text{and} \quad
2\,\kappa\,\Crel^2\Cstb^2~<~\frac{1-\qpcg^2}{2}.
\end{align}%
With~\eqref{eq:0601a}--\eqref{eq:0601c}, we finally define
\begin{align}
\begin{split}
 0<\qctr:=\max\Big\{&1-\varepsilon \,,\, \big(\mu\,(1+\gamma)\,\qest+2\,\kappa\,\Crel^2\big)\,\mu^{-1}\,,\,1-\kappa,\\
&\big(\mu\,(1+\gamma^{-1})\,\qest\,\Cstb^2\,(1+\qpcg)^2+\qpcg^2+2\,\kappa\,\Crel^2\Cstb^2\big)\Big\}<1.\label{pr_contr:qctr}
\end{split}
\end{align}

{\bf Step~2.}\quad Due to reliability \eqref{axiom:reliability}, stability \eqref{axiom:stability}, and Lemma~\ref{lemma2:pcg}(iii), it follows that
\begin{align*}
&\enorm{\PPhi^\exact-\PPhi_j^\exact}^2=(1-\varepsilon)\enorm{\PPhi^\exact-\PPhi_j^\exact}^2+\varepsilon\,\enorm{\PPhi^\exact-\PPhi_j^\exact}^2
\\& \quad
\stackrel{\eqref{axiom:reliability}}{\leq}(1-\varepsilon)\enorm{\PPhi^\exact-\PPhi_j^\exact}^2+\varepsilon\,\Crel^2\,\eta_j(\PPhi_j^\exact)^2
\\& \quad
\stackrel{\eqref{axiom:stability}}{\leq}(1-\varepsilon)\enorm{\PPhi^\exact-\PPhi_j^\exact}^2+2\,\varepsilon\,\Crel^2\,\big(\eta_j(\PPhi_{jk})^2+\Cstb^2\,\enorm{\PPhi_j^\exact-\PPhi_{jk}}^2\big)
\\& \quad
\!\stackrel{\ref{lemma2:pcg}\rm(iii)}{\leq}\!
(1-\varepsilon)\enorm{\PPhi^\exact-\PPhi_j^\exact}^2+2\,\varepsilon\,\Crel^2\,\eta_j(\PPhi_{jk})^2+2\,\varepsilon\,\Crel^2\,\Cstb^2\,(1-\qpcg)^{-2}\,\enorm{\PPhi_{j(k+1)}-\PPhi_{jk}}^2.
\end{align*}

{\bf Step~3.}\quad
We consider the case $k+1<\k(j)$. 
Step (iv) of Algorithm \ref{algorithm} yields that
\begin{align}\label{pr_contr:aux9}
\eta_j(\PPhi_{j(k+1)})^2<\lambda^{-2}\enorm{\PPhi_{j(k+1)}-\PPhi_{jk}}^2.
\end{align}
Moreover, the Pythagoras identity \eqref{eq:gal_orth} implies that
\begin{align}\label{pr_contr:aux10}\begin{split}
\enorm{\PPhi_j^\exact-\PPhi_{j(k+1)}}^2
&=(1-\varepsilon) \, \enorm{\PPhi_j^\exact-\PPhi_{jk}}^2+\varepsilon\,\enorm{\PPhi_j^\exact-\PPhi_{jk}}^2-\enorm{\PPhi_{j(k+1)}-\PPhi_{jk}}^2.
\end{split}
\end{align}
Further, we note the Pythagoras identity
\begin{align}\label{eq:dp:aux}
 \enorm{\phi^\exact - \PPhi_{j}^\exact}^2 
 + \enorm{\PPhi_{j}^\exact - \PPsi_{j}}^2 
 = \enorm{\phi^\exact - \PPsi_{j}}^2 
 \quad \text{for all } \PPsi_{j} \in \PP^0(\TT_{j}).
\end{align}%
Combining \eqref{pr_contr:aux9}--\eqref{eq:dp:aux} and applying Lemma~\ref{lemma2:pcg}(iii), we see that
\begin{align*}
&\Delta_{j(k+1)}
=\mu\,\eta_j(\PPhi_{j(k+1)})^2+\enorm{\PPhi_j^\exact-\PPhi_{j(k+1)}}^2+\enorm{\PPhi^\exact-\PPhi_j^\exact}^2
\\& \quad
<(1-\varepsilon)\enorm{\PPhi_j^\exact-\PPhi_{jk}}^2+\varepsilon\,\enorm{\PPhi_j^\exact-\PPhi_{jk}}^2
+(\mu\,\lambda^{-2}-1)\enorm{\PPhi_{j(k+1)}-\PPhi_{jk}}^2+\enorm{\PPhi^\exact-\PPhi_j^\exact}^2
\\& \quad
\!\stackrel{\ref{lemma2:pcg}\rm(iii)}{\leq}\!
(1-\varepsilon)\,\enorm{\PPhi_j^\exact-\PPhi_{jk}}^2+\big(\varepsilon\,(1-\qpcg)^{-2}+\mu\,\lambda^{-2}-1\big)\enorm{\PPhi_{j(k+1)}-\PPhi_{jk}}^2+\enorm{\PPhi^\exact-\PPhi_j^\exact}^2.
\intertext{Step 2 yields that}
& \quad
\leq(1-\varepsilon)\big(\enorm{\PPhi_j^\exact-\PPhi_{jk}}^2+\enorm{\PPhi^\exact-\PPhi_j^\exact}^2\big)+2
\,\varepsilon\,\Crel^2\,\eta_j(\PPhi_{jk})^2
\\& \qquad
+\big(\varepsilon\,(1-\qpcg)^{-2}+\mu\,\lambda^{-2}-1+2\,\varepsilon\,\Crel^2\,\Cstb^2\,(1-\qpcg)^{-2}\big)\enorm{\PPhi_{j(k+1)}-\PPhi_{jk}}^2.
\end{align*}
Using~\eqref{eq:0601a}--\eqref{eq:0601b} and~\eqref{eq:dp:aux}, we thus see that
\begin{align*}
\Delta_{j(k+1)}
\leq (1-\varepsilon)\big(\mu\,\eta_j(\PPhi_{jk})^2 +
\enorm{\PPhi^\exact-\PPhi_{jk}}^2
\big)
\reff{pr_contr:qctr}\le \qctr \, \Delta_{jk}
\quad \text{if $k+1 < \k(j)$}.
\end{align*}
This concludes the proof of~\eqref{eq1:lem:contr}.

{\bf Step~4.}\quad
We use the definition $\PPhi_{(j+1)0}:=\PPhi_{j\k}$ from Step (vi) of Algorithm \ref{algorithm} to see that
\begin{align}\label{pr_contr:aux12} \begin{split}
\Delta_{(j+1)0}
&= \mu\,\eta_{j+1}(\PPhi_{(j+1)0})^2
+ \enorm{\PPhi^\exact-\PPhi_{(j+1)0}}^2
=\mu\,\eta_{j+1}(\PPhi_{j\k})^2
+ \enorm{\PPhi^\exact-\PPhi_{j\k}}^2.
\end{split}
\end{align}
For the first summand of \eqref{pr_contr:aux12}, we use stability~\eqref{axiom:stability} and reduction~\eqref{axiom:reduction}. Together with the D\"orfler marking strategy in Step~(v) of Algorithm~\ref{algorithm} and $\MM_j \subseteq \TT_j \backslash \TT_{j+1}$, we see that
\begin{align}\label{eq:dp:estimator_reduction}
\begin{split}
 \eta_{j+1}(\PPhi_{j\k})^2
 &= \eta_{j+1}(\TT_{j+1} \backslash \TT_j, \PPhi_{j\k})^2
 + \eta_{j+1}(\TT_{j+1} \cap \TT_j, \PPhi_{j\k})^2
 \\
 &\le \qred \, \eta_{j}(\TT_j \backslash \TT_{j+1}, \PPhi_{j\k})^2
 + \eta_{j}(\TT_{j+1} \cap \TT_j, \PPhi_{j\k})^2
 \\
 &= \eta_{j}(\PPhi_{j\k})^2
 - (1-\qred) \, \eta_{j}(\TT_j \backslash \TT_{j+1}, \PPhi_{j\k})^2
 \\
 &\le \eta_{j}(\PPhi_{j\k})^2
 - (1-\qred)\theta^2 \, \eta_{j}(\PPhi_{j\k})^2
 \reff{eq:def:qest}= \qest \, \eta_{j}(\PPhi_{j\k})^2.
\end{split}
\end{align}
With this and stability \eqref{axiom:stability}, the Young inequality and 
Lemma~\ref{lemma2:pcg}(ii) yield that
\begin{align}\label{pr_contr:aux13}
\notag
\eta_{j+1}(\PPhi_{j\k})^2
&\reff{eq:dp:estimator_reduction}{\leq}\qest\,\eta_j(\PPhi_{j\k})^2
\stackrel{\eqref{axiom:stability}}{\leq}(1+\gamma)\,\qest\,\eta_j(\PPhi_{j(\k-1)})^2+(1+\gamma^{-1})\,\qest\,\Cstb^2\enorm{\PPhi_{j\k}-\PPhi_{j(\k-1)}}^2\\
&\hspace*{-7mm}
\stackrel{\rm(ii)}{\leq} (1+\gamma)\,\qest\,\eta_j(\PPhi_{j(\k-1)})^2+(1+\gamma^{-1})\,\qest\,\Cstb^2\,(1+\qpcg)^2\enorm{\PPhi_j^\exact-\PPhi_{j(\k-1)}}^2.
\end{align}
For the second summand of~\eqref{pr_contr:aux12}, we apply the Pythagoras identity~\eqref{eq:dp:aux} 
together with Lemma~\ref{lemma2:pcg}\rm(i)
and obtain that\hspace*{-2mm}
\begin{align}\label{pr_contr:aux14}
 \enorm{\PPhi^\exact-\PPhi_{j\k}}^2
 \reff{eq:dp:aux} = \enorm{\PPhi^\exact-\PPhi_j^\exact}^2
 + \enorm{\PPhi_j^\exact-\PPhi_{j\k}}^2
 \stackrel{\ref{lemma2:pcg}\rm(i)}\le \enorm{\PPhi^\exact-\PPhi_j^\exact}^2
 + \qpcg^2 \, \enorm{\PPhi_j^\exact-\PPhi_{j(\k-1)}}^2.
\end{align}
Combining \eqref{pr_contr:aux12}--\eqref{pr_contr:aux14}, we end up with
\begin{align*}
\Delta_{(j+1)0}&\leq\mu\,(1+\gamma)\,\qest\,\eta_j(\PPhi_{j(\k-1)})^2\\
&\quad+\big(\mu\,(1+\gamma^{-1})\,\qest\,\Cstb^2\,(1+\qpcg)^2+\qpcg^2\big)\enorm{\PPhi_j^\exact-\PPhi_{j(\k-1)}}^2
+\enorm{\PPhi^\exact-\PPhi_j^\exact}^2.
\end{align*}
Using the same arguments as in Step 2, we get that
\begin{align*}
\Delta_{(j+1)0}&\leq\mu\,(1+\gamma)\,\qest\,\eta_j(\PPhi_{j(\k-1)})^2\\
&\quad+\big(\mu\,(1+\gamma^{-1})\,\qest\,\Cstb^2\,(1+\qpcg)^2+\qpcg^2\big)\enorm{\PPhi_j^\exact-\PPhi_{j(\k-1)}}^2\\
&\quad+(1-\kappa)\enorm{\PPhi^\exact-\PPhi_j^\exact}^2+2\,\kappa\,\Crel^2\,\eta_j(\PPhi_{j(\k-1)})^2+2\,\kappa\,\Crel^2\,\Cstb^2\,\enorm{\PPhi_j^\exact-\PPhi_{j(\k-1)}}^2\\
&=\big(\mu\,(1+\gamma)\,\qest+2\,\kappa\,\Crel^2\big)\,\eta_j(\PPhi_{j(\k-1)})^2+(1-\kappa)\enorm{\PPhi^\exact-\PPhi_j^\exact}^2\\
&\quad+\big(\mu\,(1+\gamma^{-1})\,\qest\,\Cstb^2\,(1+\qpcg)^2+\qpcg^2+2\,\kappa\,\Crel^2\Cstb^2\big)\enorm{\PPhi_j^\exact-\PPhi_{j(\k-1)}}^2
\\
&\stackrel{\eqref{pr_contr:qctr}}{\leq}\qctr\,\mu\,\eta_j(\PPhi_{j(\k-1)})^2+\qctr\,\enorm{\PPhi^\exact-\PPhi_j^\exact}^2+\qctr\,\enorm{\PPhi_j^\exact-\PPhi_{j(\k-1)}}^2
\reff{eq:dp:aux}= \qctr\,\Delta_{j(\k-1)}.
\end{align*}
This concludes the proof of~\eqref{eq2:lem:contr}.

{\bf Step~5.}\quad
Inequality \eqref{eq:contr3} follows by induction.
This concludes the proof.
\end{proof}

\begin{proof}[{\bfseries Proof of Theorem~\ref{theorem:algorithm}{\bf(b)}.}]
The proof is split into three steps.

{\bf Step~1.}\quad 
Let $j \in \N$. Recall the Pythagoras identity~\eqref{eq:dp:aux}.
We use stability \eqref{axiom:stability} and Step (iv) of Algorithm~\ref{algorithm} to see that
\begin{align*}
\Delta_{j(\k-1)}&\reff{eq:dp:aux}=\mu\,\eta_j(\PPhi_{j(\k-1)})^2+\enorm{\PPhi_j^\exact-\PPhi_{j(\k-1)}}^2+\enorm{\phi^\exact-\PPhi_j^\exact}^2\\
&\reff{axiom:stability}\lesssim\eta_j(\PPhi_{j\k})^2+\enorm{\PPhi_{j\k}-\PPhi_{j(\k-1)}}^2+\enorm{\PPhi_j^\exact-\PPhi_{j\k}}^2+\enorm{\phi^\exact-\PPhi_j^\exact}^2\\
&\!\stackrel{\ref{algorithm}\rm(iv)}{\lesssim}\!
\eta_j(\PPhi_{j\k})^2+\enorm{\PPhi_j^\exact-\PPhi_{j\k}}^2+\enorm{\phi^\exact-\PPhi_j^\exact}^2
\reff{eq:dp:aux}\simeq\Delta_{j\k}.
\end{align*}
With the Pythagoras identity~\eqref{eq:gal_orth}, we may argue similarly to obtain that
\begin{align*}
\Delta_{j\k}
&\reff{eq:dp:aux}=\mu\,\eta_j(\PPhi_{j\k})^2+\enorm{\PPhi_j^\exact-\PPhi_{j\k}}^2+\enorm{\phi^\exact-\PPhi_j^\exact}^2\\
&\reff{axiom:stability}\lesssim\eta_j(\PPhi_{j(\k-1)})^2+\enorm{\PPhi_{j\k}-\PPhi_{j(\k-1)}}^2+\enorm{\PPhi_j^\exact-\PPhi_{j\k}}^2+\enorm{\phi^\exact-\PPhi_j^\exact}^2\\
&\reff{eq:gal_orth}= \eta_j(\PPhi_{j(\k-1)})^2+\enorm{\PPhi_j^\exact-\PPhi_{j(\k-1)}}^2+\enorm{\phi^\exact-\PPhi_j^\exact}^2
\reff{eq:dp:aux}\simeq\Delta_{j(\k-1)}.
\end{align*}
Hence, it follows that $\Delta_{j\k}\simeq\Delta_{j(\k-1)}$.

{\bf Step~2.}\quad For $0\leq j\leq j^\prime$, define $\widehat k(j) := \widehat k \in\N_0$ by
\begin{align*}
\widehat k:=\begin{cases}
\k(j)\quad&\textrm{if } j<j^\prime,\\
k^\prime\quad&\textrm{if }j=j^\prime.
\end{cases}
\end{align*}
From Step 1, Lemma \ref{lem:contr}, and the geometric series (for the sum over $k$), it follows that
\begin{align*}
 \sum_{j=0}^{j^\prime} \sum_{k=0}^{\widehat{k}(j)} \Delta_{jk}^{-1}
 \lesssim \Delta_{j'k'}^{-1} +  \sum_{j=0}^{j^\prime}\sum_{k=0}^{\widehat{k}(j)-1}\Delta_{jk}^{-1}
 &\stackrel{\eqref{eq:contr3}}{\leq} \Delta_{j'k'}^{-1} +  \sum_{j=0}^{j^\prime}
 \sum_{k=0}^{\widehat{k}(j)-1}\qctr^{|(j,\widehat k-1)| - |(j,k)|}\,\Delta_{j(\widehat k-1)}^{-1}
 \\
 &\lesssim \Delta_{j'k'}^{-1} +  \sum_{j=0}^{j^\prime}\Delta_{j(\widehat{k}-1)}^{-1}.
\end{align*}
For $k^\prime<\k(j^\prime)$, inequality \eqref{eq:contr3} and the geometric series (for the sum over $j$)  yield that
\begin{align*}
 \sum_{j=0}^{j^\prime}\Delta_{j(\widehat{k}-1)}^{-1}
 \reff{eq:contr3}\lesssim \sum_{j=0}^{j^\prime}\qctr^{|(j^\prime,k^\prime)|-|(j,\widehat{k}-1)|}\,\Delta_{j^\prime k^\prime}^{-1}
 \lesssim\Delta_{j^\prime k^\prime}^{-1}.
\end{align*}
For $k^\prime=\k(j^\prime)$, inequality \eqref{eq:contr3}, the geometric series, and Step 1 yield that
\begin{align*}
\sum_{j=0}^{j^\prime}\Delta_{j(\widehat{k}-1)}^{-1}
= \Delta_{j^\prime (\k-1)}^{-1} + \sum_{j=0}^{j^\prime-1}\Delta_{j(\k-1)}^{-1}
&\reff{eq:contr3}\lesssim \Big(1 + \sum_{j=0}^{j^\prime-1}\qctr^{|(j^\prime,\k-1)|-|(j,\k-1)|} \Big) \, \Delta_{j^\prime (\k-1)}^{-1}
\\&
\lesssim\Delta_{j^\prime(\k-1)}^{-1}
\simeq\Delta_{j^\prime \k}^{-1}
= \Delta_{j^\prime k'}^{-1}.
\end{align*}
Overall, it follows that
\begin{align}\label{eq:dp:Delta}
 \sum_{j=0}^{j^\prime}\sum_{k=0}^{\widehat{k}(j)}\Delta_{jk}^{-1}
 &\lesssim\Delta_{j^\prime k^\prime}^{-1}
 \quad\textrm{for all $(j^\prime,k^\prime)\in\QQ$.}
\end{align}

{\bf Step~3.}\quad
According to the proof of \cite[Lemma 4.9]{axioms}, estimate~\eqref{eq:dp:Delta} guarantees (and is even equivalent to) the existence of $0<\qlin<1$ such that
\begin{align}\label{eq2:dp:Delta}
\Delta_{j^\prime k^\prime}^{1/2} \lesssim \qlin^{|(j^\prime,k^\prime)|-|(j,k)|} \, \Delta_{jk}^{1/2}\quad\textrm{for all }(j,k),(j^\prime,k^\prime)\in\QQ\textrm{ with }(j',k') \ge (j,k).
\end{align}
Clearly, it holds that $\quasierror_{jk} \simeq \Delta_{jk}^{1/2}$ for all $(j,k) \in \QQ$.
This and~\eqref{eq2:dp:Delta} conclude the proof.
\end{proof}

\subsection{Proof of Theorem~\ref{theorem:algorithm}(c)}
\label{section:proof:thm:c}%
The proof of optimal convergence rates requires the following additional properties of the mesh-refinement strategy. For 3D BEM (and 2D NVB from Section~\ref{section:refined3D}) these properties are verified in~\cite{bdd,stevenson07,stevenson}, and any assumption on $\TT_0$ is removed in~\cite{kpp}. For 2D BEM (and the extended 1D bisection from Section~\ref{section:refined2D}), these properties are verified in~\cite{cmam}.
\begin{enumerate}
\renewcommand{\theenumi}{R\arabic{enumi}}
\bf
\item\label{axiom:split}
\rm
\textbf{splitting property:} Each refined element is split in at least 2 and at most in 
$C_{\rm son}\ge2$ many sons, i.e., for all $\TT_\coarse \in \mathbb{T}$ and 
all $\mathcal{M}_\coarse \subseteq \TT_\coarse$, the refined mesh
$\TT_\circ = \operatorname{refine}(\TT_\coarse, \mathcal{M}_\coarse)$ satisfies that
\begin{align*}
	\# (\TT_\coarse \setminus \TT_\circ) + \# \TT_\coarse \leq \# \TT_\circ
	\leq C_{\rm son} \, \# (\TT_\coarse \setminus \TT_\circ) + \# (\TT_\coarse \cap \TT_\circ).
\end{align*}

\bf
\item\label{axiom:overlay} 
\rm
\textbf{overlay estimate:} For all meshes $\TT \in \mathbb{T}$ and $\TT_\coarse,\TT_\circ \in \operatorname{refine}(\TT)$
there exists a common refinement $\TT_\coarse \oplus \TT_\circ \in \operatorname{refine}(\TT_\coarse) \cap \operatorname{refine}(\TT_\circ) \subseteq \operatorname{refine}(\TT)$
with
\begin{align*}
	\# (\TT_\coarse \oplus \TT_\circ) \leq \# \TT_\coarse + \# \TT_\circ - \# \TT.
\end{align*}

\bf
\item\label{axiom:mesh_closure} 
\rm
\textbf{mesh-closure estimate:}
 There exists $C_{\rm mesh}>0$ such that the sequence $\TT_j$ with corresponding 
 $\mathcal{M}_j \subseteq \TT_j$, which is generated by Algorithm~\ref{algorithm}, satisfies that	
\begin{align*}
	\# \TT_j  - \# \TT_0 \leq C_{\rm mesh} \sum_{\ell=0}^{j -1} \# \mathcal{M}_\ell.
\end{align*} 
\end{enumerate}

Recall the constants $\Cstab > 0$ from~\eqref{axiom:stability} and $\Cdrel>0$ from~\eqref{axiom:discrete_reliability}.
Suppose that $0<\theta\leq1$ and $\lambda > 0 $ are sufficiently small such that
\begin{align}\label{eq:ass_theta}
0 < \theta'':=\frac{\theta+\lambda/\lambda_{\rm opt}}{1-\lambda/\lambda_{\rm opt}}<\theta_{\rm opt}:=\big(1+\Cstb^2\,\Cdrel^2\big)^{-1/2},
\text{~where~~}
\lambda_{\rm opt} := \Big(\Cstab \,\frac{\qpcg}{ 1-\qpcg}\Big)^{-1}.\hspace*{-3mm}
\end{align}
In particular, it holds that $0 < \theta < \theta_{\rm opt}$ and $0 < \lambda < \lambda_{\rm opt}$.
We need the following comparison lemma which is found in \cite[Lemma~4.14]{axioms}.

\begin{lemma}\label{lemma:doerfler2}
Suppose \eqref{axiom:overlay}, \eqref{axiom:stability}, \eqref{axiom:reduction}, and \eqref{axiom:discrete_reliability}. Recall the assumption~\eqref{eq:ass_theta}.
There exist constants $C_1, C_2 > 0$ such that for all $s>0$ with $\norm{\PPhi^\exact}{\mathbb{A}_s} < \infty$ and all $j \in \N_0$, there exists $\mathcal{R}_j \subseteq \mathcal{T}_j $ which satisfies
\begin{align}\label{eq:doerfler2_Rbound}
\# \mathcal{R}_j \leq C_1 C_2^{-1/s} \, \norm{\PPhi^\exact}{\mathbb{A}_s}^{1/s} \eta_j(\PPhi_j^\exact)^{-1/s},
\end{align}
as well as the D\"orfler marking criterion
\begin{align}\label{eq:doerfler2_doerfler}
\theta'' \eta_j(\PPhi_j^\exact) \leq \eta_j(\mathcal{R}_j, \PPhi_j^\exact).
\end{align}
The constants $C_1, C_2$ depend only on the constants of~\eqref{axiom:stability},~\eqref{axiom:reduction}, and~\eqref{axiom:discrete_reliability}. \qed
\end{lemma}

Another lemma, which we need for the proof of Theorem~\ref{theorem:algorithm}(c), shows that the iterates $\PPhi_{\coarse k} $ of Algorithm~\ref{algorithm} are close to the exact Galerkin approximation $\PPhi_\coarse^\exact \in \PP^0(\TT_\coarse)$.

\pagebreak
\begin{lemma}\label{lemma:est_equivalence}
Let $0<\lambda<\lambda_{\rm opt}$. For all $j \in \N_0$, it holds that
\begin{align}\label{eq1:lemma:equivalence}
 \enorm{\PPhi^\exact_j-\PPhi_{j\k}} \le \lambda\,\frac{\qpcg}{ 1-\qpcg}\,\min\Big\{\eta_j(\PPhi_{j\k})\,,\,\frac{1}{1-\lambda/\lambda_{\rm opt}}\,\eta_j(\PPhi^\exact_j)\Big\}.
\end{align}
Moreover, there holds equivalence
\begin{align}\label{eq2:lemma:equivalence}
 (1-\lambda/\lambda_{\rm opt})\,\eta_j(\PPhi_{j\k}) \le \eta_j(\PPhi^\exact_j) \le (1+\lambda/\lambda_{\rm opt})\,\eta_j(\PPhi_{j\k}).
\end{align}
\end{lemma}

\begin{proof}
Stability
\eqref{axiom:stability} yields that $|\eta_j(\PPhi^\exact_j)-\eta_j(\PPhi_{j\k})|\le \Cstab\,\enorm{\PPhi^\exact_j - \PPhi_{j\k}} $.
 Therefore, Lemma~\ref{lemma2:pcg}(iv) and the assumption on the PCG iterate in Step~(iv) of Algorithm~\ref{algorithm} imply that
\begin{align*}
 \enorm{\PPhi^\exact_j -\PPhi_{j\k} }
 \stackrel{\ref{lemma2:pcg}\text{(iv)}}{\le}  \frac{\qpcg}{ 1-\qpcg}  \,\enorm{\PPhi_{j\k}-\PPhi_{j (\k-1)}}
  &\le\lambda \,\frac{\qpcg}{ 1-\qpcg} \,\eta_j(\PPhi_{j\k}) \\
 &\reff{axiom:stability}\le \lambda\, \frac{\qpcg}{ 1-\qpcg} \,\big(\eta_j(\PPhi^\exact_j) + \Cstab\,\enorm{\PPhi^\exact_j -\PPhi_{j\k} } \big).
\end{align*}
Since $0<\lambda<\lambda_{\rm opt}$ and hence $\lambda\Cstab\,\frac{\qpcg}{ 1-\qpcg} = \lambda/\lambda_{\rm opt}<1$, this yields that
\begin{align*}
 \enorm{\PPhi^\exact_j -\PPhi_{j\k} } \le\frac{ \lambda\,\frac{\qpcg}{ 1-\qpcg} }{ 1 - \lambda\Cstab\,\frac{\qpcg}{ 1-\qpcg} }\,\eta_j(\PPhi^\exact_j)
 = \lambda\,\frac{\qpcg}{ 1-\qpcg} \,\frac{1}{1-\lambda/\lambda_{\rm opt}}\,\eta_j(\PPhi^\exact_j).
\end{align*}
Altogether, this proves~\eqref{eq1:lemma:equivalence}. Moreover, with stability~\eqref{axiom:stability},
we see that
\begin{align*}
\eta_j(\PPhi^\exact_j)
 \reff{axiom:stability} \le \eta_j(\PPhi_{j\k}) + \Cstab\,\enorm{\PPhi^\exact_j -\PPhi_{j\k} }
 \reff{eq1:lemma:equivalence}
 \le (1+\lambda/\lambda_{\rm opt})\,\eta_j(\PPhi_{j\k})
\end{align*}
as well as 
\begin{align*}
 \eta_j(\PPhi_{j\k})
 \reff{axiom:stability}\le \eta_j(\PPhi^\exact_j) + \Cstab\,\enorm{\PPhi^\exact_j -\PPhi_{j\k} }
 \reff{eq1:lemma:equivalence}
 \le \Big(1+ \,\frac{\lambda/\lambda_{\rm opt}}{1-\lambda/\lambda_{\rm opt}}\Big)\,\eta_j(\PPhi^\exact_j)
 = \frac{1}{1-\lambda/\lambda_{\rm opt}}\,\eta_j(\PPhi^\exact_j).
\end{align*}
This concludes the proof.
\end{proof}

Finally, we need the following lemma which immediately shows
``$\Longleftarrow$'' in~\eqref{eq:theorem:opt_rate}.

\begin{lemma}\label{lemma:class}
Suppose~\eqref{axiom:split}.
For $j \in \N_0$, let $\opt\TT_{j+1} = \refine(\opt\TT_j,\opt\MM_j)$ with arbitrary, but non-empty $\opt\MM_j \subseteq \opt\TT_j$ and $\opt\TT_0 = \TT_0$. Let $\opt\QQ \subseteq \N_0 \times \N_0$ be an index set and $\opt\PPhi_{jk} \in \PP^0(\opt\TT_j)$ for all $(j,k) \in \opt\QQ$.
Let $s > 0$ and suppose that the corresponding quasi-errors $\opt\quasierror_{jk}^2 := \enorm{\phi^\exact - \opt\PPhi_{jk}}^2 + \opt\eta_j(\opt\phi_j^\star)^2$ satisfy that
\begin{align}\label{eq:lemma:class}
 \sup_{(j,k) \in \opt\QQ} \big(\#\opt\TT_j - \#\TT_0 +1\big)^s \, \opt\quasierror_{jk} < \infty.
\end{align}
Then, it follows that $\norm{\phi^\exact}{\mathbb{A}_s} < \infty$.
\end{lemma}

\begin{proof}
Due to the Pythagoras identity~\eqref{eq:dp:aux} and stability~\eqref{axiom:stability}, it holds that
\begin{align}\begin{split}\label{eq:opt_aux1}
\opt\quasierror_{jk}^2&=\enorm{\phi^\star - \opt\PPhi_{jk}}^2 + \opt\eta_{j}(\opt\PPhi_{jk})^2
\reff{eq:dp:aux}=\enorm{\phi^\star - \opt\PPhi_j^\exact}^2+\enorm{\opt\PPhi_j^\exact-\opt\PPhi_{jk}}^2 + \opt\eta_{j}(\opt\PPhi_{jk})^2
\reff{axiom:stability}\gtrsim \opt\eta_j(\opt\PPhi_j^\exact)^2.
\end{split}
\end{align}
Additionally, \cite[Lemma~22]{bhp17} shows that
\begin{align}\label{eq:opt_aux2}
 \#\TT_\bullet - \#\TT_0 + 1
 \le \#\TT_\bullet \le 
 \#\TT_0 \, \big(\#\TT_\bullet - \#\TT_0 + 1\big)
 \quad \text{for all } \TT_\bullet \in \mathbb{T}.
\end{align}
Given $N\in\N_0$, there exists an index $j\in\N_0$ such that
\begin{align}\label{eq:opt_aux3} 
 \#\opt\TT_j - \#\TT_0
 \leq N 
 < N+1 
 \le \#\opt\TT_{j+1} - \#\TT_0 + 1
 \reff{eq:opt_aux2}\le \#\opt\TT_{j+1}
 \reff{axiom:split}\lesssim \#\opt\TT_j
 \reff{eq:opt_aux2}\lesssim \#\opt\TT_j- \#\TT_0 + 1.\hspace*{-1mm}
\end{align}
With~\eqref{eq:opt_aux1}--\eqref{eq:opt_aux3}, it follows that
\begin{align*}
(N+1)^s \! \min_{\substack{\TT_\coarse \in \refine(\TT_0) \\ \#\TT_\coarse - \#\TT_0 \le N}} \!\eta_\coarse(\PPhi_\coarse^\star) 
&\reff{eq:opt_aux3}\lesssim \big(\#\opt\TT_j- \#\TT_0 + 1\big)^s \, \opt\eta_j(\opt\PPhi_j^\exact)
\reff{eq:opt_aux1}\lesssim \!\! \sup_{(j,k) \in \opt\QQ} \! \big(\#\opt\TT_j- \#\TT_0 + 1\big)^s \, \opt\quasierror_{jk}
\reff{eq:lemma:class}< \infty.
\end{align*}%
Since the upper bound is finite and independent of $N$, this implies that $\norm{\phi^\exact}{\mathbb{A}_s} < \infty$.
\end{proof}

\begin{proof}[{\bfseries Proof of Theorem~\ref{theorem:algorithm}{\bf(c)}.}]
With Lemma~\ref{lemma:class}, it only remains to prove the implication ``$\Longrightarrow$'' in~\eqref{eq:theorem:opt_rate}.
The proof is split into three steps, where
we may suppose that $\norm{\phi^\star}{\mathbb{A}_s} < \infty$.

{\bf Step~1.}\quad By Assumption~\eqref{eq:ass_theta}, Lemma~\ref{lemma:doerfler2} provides a set $\mathcal{R}_j\subseteq\mathcal{T}_j$ with~\eqref{eq:doerfler2_Rbound}--\eqref{eq:doerfler2_doerfler}. 
Due to stability~\eqref{axiom:stability} 
and $\lambda_{\rm opt}^{-1} = C_{\rm stb}\,\frac{\qpcg}{1-\qpcg}$, it holds that
\begin{align*}
 \eta_j(\mathcal{R}_j,\PPhi_j^\exact) 
 \reff{axiom:stability}\le \eta_j(\mathcal{R}_j,\PPhi_{j\k}) + C_{\rm stb}\,\enorm{\PPhi_j^\star-\PPhi_{j\k}}
 \reff{eq1:lemma:equivalence}\le \eta_j(\mathcal{R}_j,\PPhi_{j\k}) + \lambda/\lambda_{\rm opt}\,\eta_j(\PPhi_{j\k}).
\end{align*}
Together with $\theta'' \eta_j(\PPhi_j^\exact) \le \eta_\ell(\mathcal{R}_\ell,\PPhi_j^\exact)$, this proves that
\begin{align*}
 (1-\lambda/\lambda_{\rm opt})\theta''\,\eta_j(\PPhi_{j\k})
 \reff{eq2:lemma:equivalence}\le \theta''\,\eta_j(\PPhi_j^\exact)
 \le\eta_j(\mathcal{R}_j,\PPhi_j^\exact)
 \le\eta_j(\mathcal{R}_j,\PPhi_{j\k}) + \lambda/\lambda_{\rm opt}\,\eta_j(\PPhi_{j\k})
\end{align*}
and results in
\begin{align}\label{eq37*}
 \theta\,\eta_j(\PPhi_{j\k}) \reff{eq:ass_theta}= \Big((1-\lambda/\lambda_{\rm opt})\theta'' - \lambda/\lambda_{\rm opt}\Big)\,\eta_j(\PPhi_{j\k}) \le \eta_j(\mathcal{R}_j,\PPhi_{j\k}).
\end{align}
Hence, $\mathcal{R}_j$ satisfies the D\"orfler marking for $\PPhi_{j\k}$ with parameter $\theta$. By choice of $\mathcal{M}_j$ in Step~(v) of Algorithm~\ref{algorithm}, we thus infer that
 \begin{align*}
 	\# \mathcal{M}_j \stackrel{\eqref{eq37*}}{\lesssim} \# \mathcal{R}_j 
 	\stackrel{\eqref{eq:doerfler2_Rbound}}{\lesssim} \eta_j(\PPhi_j^\exact)^{-1/s}
	\reff{eq2:lemma:equivalence}\simeq \eta_j(\PPhi_{j\k})^{-1/s}
	\quad\text{for all $j\in\mathbb{N}_0$}.
 \end{align*}
The mesh-closure estimate~\eqref{axiom:mesh_closure} guarantees that
\begin{align}\label{eq:theorem_optimal_rate_proof}
	\# \mathcal{T}_j - \# \mathcal{T}_0 +1 
	\reff{axiom:mesh_closure}\lesssim \sum_{\ell=0}^{j-1} \# \mathcal{M}_\ell 
	\lesssim \sum_{\ell=0}^{j-1} \eta_\ell(\PPhi_{\ell\k})^{-1/s}  \quad \text{for all } j>0.
\end{align}

{\bf Step~2.}\quad
For $j=0$ it holds that $1\lesssim\quasierror_{0\k}^{-1/s}$. For $j>0$, we proceed as follows: 
Remark~\ref{remark:quasierror} yields that $\eta_\ell(\phi_{\ell\k}) \simeq \Lambda_{\ell\k}$.
Theorem~\ref{theorem:algorithm}(b) and the geometric series prove that
\begin{align*}
 \sum_{\ell=0}^{j-1} \eta_\ell(\PPhi_{\ell\k})^{-1/s}
 \simeq \sum_{\ell=0}^{j-1} \quasierror_{\ell\k}^{-1/s}
 \reff{eq1:theorem:algorithm}\lesssim \sum_{\ell=0}^{j-1} (\qlin^{1/s})^{|(j,\k)-(\ell,\k)|} \, \quasierror_{j\k}^{-1/s}
 \lesssim \quasierror_{j\k}^{-1/s}.
\end{align*}
Combining this with \eqref{eq:theorem_optimal_rate_proof} and including the estimate for $j=0$, we derive that
\begin{align}\label{eq:opt_aux5}
 \# \mathcal{T}_j - \# \mathcal{T}_0 +1 \lesssim \quasierror_{j\k}^{-1/s}  \quad \text{for all } j \in \N_0.
\end{align}

{\bf Step~3.}\quad 
Arguing as in~\eqref{eq:opt_aux3} and employing Theorem~\ref{theorem:algorithm}(b), we see that
\begin{align*}
\#\TT_{j}-\#\TT_0+1
\reff{eq:opt_aux3}\simeq 
\#\TT_{j-1}-\#\TT_0+1
\reff{eq:opt_aux5}\lesssim \quasierror_{(j-1)\k}^{-1/s}
\reff{eq1:theorem:algorithm}\lesssim \quasierror_{jk}^{-1/s}
\text{ for all } (j,k) \in \QQ \text{ with } j > 0.
\end{align*}
Since $\k(0) \leq\#\TT_0< \infty$, we hence conclude that
$\displaystyle
 \sup_{(j,k) \in \QQ} (\#\TT_{j}-\#\TT_0+1)^s \, \quasierror_{jk} < \infty.
$
\end{proof}

\subsection{Proof of Corollary~\ref{corollary:algorithm}}
\label{section:proof:cor}%
For all $\delta > 0$, it holds that
\begin{align*}
\#\TT_\bullet - \#\TT_0 + 1 
\reff{eq:opt_aux2}\simeq \#\TT_\bullet \le (\#\TT_\bullet)\log^2(1+\#\TT_\bullet) 
\lesssim (\#\TT_\bullet)^{1+\delta} 
\quad \text{for all } \TT_\bullet \in \T,
\end{align*}
where the hidden constant depends only on $\delta$. 
From~\eqref{eq:cost:opt}, it thus follows that  
\begin{align*}
 \sup_{j \in \N_0} \big[ \#\opt\TT_j - \#\TT_0 + 1\big]^s \, \opt\eta_j(\opt\phi_j^\star) 
 \lesssim \sup_{j \in \N_0} \big[(\#\opt\TT_j) \, \log^2(1+\#\opt\TT_j)\big]^s \, \opt\eta_j(\opt\phi_j^\star) < \infty.
\end{align*}
From Lemma~\ref{lemma:class}, we derive that $\norm{\phi^\star}{\mathbb{A}_s} < \infty$.
Hence, Theorem~\ref{theorem:algorithm}(c) yields that
\begin{align}\label{eq:aux:complexity}
 \sup_{(j,k) \in \QQ} \big[\#\TT_j\big]^s \, \quasierror_{jk} 
 \simeq  \sup_{(j,k) \in \QQ} \big[ \#\TT_j - \#\TT_0 + 1 \big]^s \, \quasierror_{jk}  < \infty.
\end{align}
Let $0 < \eps < s$ and choose $\delta > 0$ such that 
\begin{align*}
 0 < s - \eps = \frac{s}{1+\delta} =: t.
\end{align*}
This leads to
\begin{align*}
 (\#\TT_j) \log^2(1+\#\TT_j) 
 \lesssim (\#\TT_j)^{1+\delta} 
 \reff{eq:aux:complexity}\lesssim \quasierror_{jk}^{-(1+\delta)/s} = \quasierror_{jk}^{-1/t}
 \quad \text{for all } (j,k) \in \QQ.
\end{align*}
From Theorem~\ref{theorem:algorithm}(b) and the geometric series, it follows that
\begin{align*}
 \sum_{(j,k) \le (j',k')} \quasierror_{jk}^{-1/t}
 \reff{eq1:theorem:algorithm}\lesssim \sum_{(j,k) \le  (j',k')} (\qlin^{1/t})^{|(j',k')|-|(j,k)|}  \quasierror_{j'k'}^{-1/t}
 \lesssim \quasierror_{j'k'}^{-1/t}
 \quad \text{for all } (j',k') \in \QQ.
\end{align*}
Combining the last two estimates, we see that
\begin{align*}
 \Big[\sum_{(j,k) \le (j',k')} (\#\TT_j)\log^2(1+\#\TT_j)\big)\Big]^{s-\eps}
 \lesssim \quasierror_{j'k'}^{-(s-\eps)/t}
 = \quasierror_{j'k'}^{-1}
 \text{ for all } (j',k') \in \QQ. 
\end{align*}
This concludes the proof.
\qed

\def\subspace{{\ast}}%
\def\SS{\mathcal{S}}%
\def\SSS{\widetilde{\mathcal{S}}}%
\section{Hyper-singular integral equation}
\label{section:hypsing}%

We only sketch the setting and refer to~\cite{mclean} for further details and proofs. 
Given $f: \Gamma \to \R$, the hyper-singular integral equation seeks $u^\exact:\Gamma\to\R$ such that
\begin{align}\label{eq:hypsing}
 (W u^\star)(x) := -\partial_{\normal(x)} \int_\Gamma \partial_{\normal(y)} G(x-y) u^\star(y) \d{y}
 = f(x) \quad \text{for all } x \in \Gamma,
\end{align}
where $\partial_{\normal}$ denotes the normal derivative with the outer unit normal vector $\normal(\cdot)$ on $\Gamma \subseteq \partial\Omega$. For $0 \le \alpha \le 1$, define $\H^\alpha(\Gamma) := \set{v \in H^\alpha(\Gamma)}{\supp(v) \subseteq \overline\Gamma}$ and let $H^{-\alpha}(\Gamma)$ be its dual space with respect to $\dual\cdot\cdot$. Note that $\H^{\pm\alpha}(\Gamma) = H^{\pm\alpha}(\Gamma)$ for $\Gamma = \partial\Omega$.
The hyper-singular integral operator $W : \H^{1/2+s}(\Gamma) \to H^{-1/2+s}(\Gamma)$ is a bounded linear operator for all $-1/2 \le s \le 1/2$ which is even an isomorphism for $-1/2 < s < 1/2$. For $s=0$, the operator $W$ is symmetric and (since $\Gamma$ is connected) positive semi-definite with kernel being the constant functions. For $\Gamma \subsetneqq \partial\Omega$, the operator $W : \H^{1/2}(\Gamma) \to H^{-1/2}(\Gamma)$ is hence an elliptic isomorphism. Moreover, for $\Gamma = \partial\Omega$ and $H^{\pm1/2}(\Gamma) := \set{\psi \in H^{\pm1/2}(\Gamma)}{\dual{\psi}{1}=0}$, $W : H^{1/2}_\subspace(\Gamma) \to H^{-1/2}_\subspace(\Gamma)$ is an elliptic isomorphism. Therefore, 
\begin{align*}
 \edual{u}{v} := \begin{cases}
  \dual{Wu}{v}, & \text{if } \Gamma \subsetneqq \partial \Omega,\\
  \dual{Wu}{v} + \dual{u}{v}, & \text{if } \Gamma = \partial \Omega
 \end{cases}
\end{align*}
defines a scalar product on $\H^{1/2}(\Gamma)$, and the induced norm $\enorm{u} := \edual{u}{u}^{1/2}$ is an equivalent norm on $\H^{1/2}(\Gamma)$. Let $f \in H^{-1/2}(\Gamma)$. If $\Gamma \subsetneqq \partial\Omega$, suppose additionally that $f \in H^{-1/2}_\subspace(\partial\Omega)$. Then,~\eqref{eq:hypsing} admits a unique solution $u^\star \in \H^{1/2}(\Gamma)$ resp.\ $u^\star \in H^{1/2}_\subspace(\partial\Omega)$, which is also the unique solution $u^\star \in \H^{1/2}(\Gamma)$ of the variational formulation
\begin{align*}
 \edual{u^\star}{v} = \dual{f}{v}
 \quad \text{ for all } v \in \H^{1/2}(\Gamma).
\end{align*}
Given a mesh $\TT_\bullet$ of $\Gamma$, let 
\begin{align*}
 \SSS^1(\TT_\bullet) := \set{v \in \H^{1/2}(\Gamma)}{\forall T \in \TT_\bullet \quad v|_T \text{ is affine}}.
\end{align*}
The Lax--Milgram theorem yields existence and uniqueness of $u_\bullet^\star \in \SSS^1(\TT_\bullet)$ such that
\begin{align*}
 \edual{u_\bullet^\star}{v_\bullet} = \dual{f}{v_\bullet}
 \quad \text{ for all } v_\bullet \in \SSS^1(\TT_\bullet).
\end{align*}
With the corresponding weighted-residual error estimator, it holds that
\begin{align*}
 \enorm{u^\star - u_\bullet^\star} 
 \le \Crel \, \eta_\bullet(u_\bullet^\star) 
 := \! \bigg( \! \sum_{T \in \TT_\bullet} \eta_\bullet(T,u_\bullet^\star)^2 \! \bigg)^{1/2} \!\!\!,
 \text{ where }
 \eta_\bullet(T,u_\bullet^\star)^2
 := h_T \, \norm{f - W u_\bullet^\star}{L^2(T)}^2;
 \end{align*}
see~\cite{cs95, cc97} for $d = 2$ resp.~\cite{cmps04} for $d = 3$.

In~\cite{dissFuehrer,MR3612925}, optimal additive Schwarz preconditioners are derived for this setting. Hence, Algorithm~\ref{algorithm} can also be used in the present setting. 
We refer to~\cite[Section~3.3]{partTwo} for the fact that the \emph{axioms of adaptivity}~\eqref{axiom:stability}--\eqref{axiom:discrete_reliability} from Proposition~\ref{prop:axioms} remain valid for the hyper-singular integral equation. All other arguments in Section~\ref{proof:algorithm} rely only on general properties of the PCG algorithm (Section~\ref{section:pcg}), the properties~\eqref{axiom:stability}--\eqref{axiom:discrete_reliability}, and the Hilbert space setting of $\enorm\cdot$. Overall, this proves that our main results (Theorem~\ref{theorem:algorithm} and Corollary~\ref{corollary:algorithm}) also cover the hyper-singular integral equation.

\bibliographystyle{alpha}
\bibliography{literature}

\end{document}